\documentclass[11pt]{article}
\usepackage{amsmath, amsthm, amssymb, mathtools}
\usepackage{bm, soul}
\usepackage{graphicx}
\usepackage[showzone=false]{datetime2}
\usepackage{fancyhdr}
\usepackage[table]{xcolor}
\usepackage{optidef}
\usepackage{hyperref}
\usepackage{geometry}
\usepackage{booktabs}
\usepackage{algorithm} 
\usepackage{algpseudocode}
\usepackage{tikz}
\usetikzlibrary{positioning,shapes,shapes.geometric,arrows.meta,calc,math,trees,decorations,decorations.pathmorphing,decorations.markings,fadings,spy,graphs,automata}
\usepackage[backend=biber, style=numeric, maxnames=4, urldate=long, url=true, doi=false]{biblatex}
\usepackage{listings}

\makeatletter
\renewcommand\paragraph{%
 \@startsection {paragraph}{4}{\z@ }{3.25ex \@plus 1ex
 \@minus .2ex}{-1em}{\normalfont \normalsize \bfseries }}%
\renewcommand\subparagraph{%
 \@startsection {subparagraph}{5}{\z@ }{3.25ex \@plus 1ex
 \@minus .2ex}{-1em}{\normalfont \normalsize \itshape }}%
\makeatother

\newtheorem{theorem}{Theorem}

\newtheorem{proposition}[theorem]{Proposition}
\newtheorem{lemma}[theorem]{Lemma}

\newcommand{\E}{\mathbb E}

\newcommand{\Var}{\mathrm{Var}}
\newcommand{\avg}{\mathrm{avg}}
\newcommand{\diag}{\mathrm{diag}}

\newcommand{\Markov}{\mathrm{Markov}}
\newcommand{\R}{\mathbb R}

\newcommand{\mdp}{\mathrm{mdp}}
\newcommand{\qdp}{\mathrm{qdp}}
\newcommand{\extract}{\mathrm{extract}}
\newcommand{\exact}{\mathrm{exact}}
\newcommand{\qplex}{\mathrm{qplex}}

\newcommand{\AAA}{\mathcal A}

\newcommand{\PPP}{\mathcal P}

\newcommand{\SSS}{\mathcal S}
\newcommand{\LLL}{\mathcal L}
\newcommand{\ZZZ}{\mathcal Z}
\newcommand{\UUU}{\mathcal U}

\newcommand{\XXX}{\mathcal X}

\newcommand{\ttheta}{\bm \theta}
\newcommand{\TTheta}{\bm \Theta}

\newcommand{\ggamma}{\bm \gamma}
\newcommand{\GGamma}{\bm \Gamma}

\let\oldSS\SS
\DeclareRobustCommand\SS{\ifmmode\bm{S}\else\oldSS\fi}

\newcommand{\old}{\textup{old}}
\newcommand{\new}{\textup{new}}

\newcommand{\bin}{\textup{bin}}
\newcommand{\ctrd}{\textup{ctrd}}

\DeclarePairedDelimiterX{\infdivx}[2]{(}{)}{%
  #1\;\delimsize\|\;#2%
}
\newcommand{\KL}{\mathrm{KL}\infdivx}

\DeclareMathOperator*{\argmax}{arg\,max}

\title{QPLEX Decision Processes: Formulation via Nonlinear Markov Chains and Optimization via Policy Gradients}
\author{Antonius B.\ Dieker, Steven T.\ Hackman, Zitong Wang, Yunhao Yan}
\date{}

\begin{document}
\maketitle

\begin{abstract}
We introduce a QPLEX Decision Process (QDP) as a model for dynamic control of queueing systems with 
non-stationary arrivals, general service distributions, and service-level chance constraints.
QDPs integrate QPLEX, a computational modeling methodology for transient analysis of stochastic systems,
into a nonlinear Markov decision framework. 
Since QPLEX approximations use nonlinear transition probabilities with orders-of-magnitude smaller state spaces, QDPs circumvent the curse of dimensionality associated with general service times.
Via forward and backward iterative schemes, we can rapidly compute gradients 
deterministically on the much smaller state space, eliminating sampling variance.
We further address optimization through natural-gradient-inspired methods with block-diagonal Fisher approximations. 
To illustrate the QDP methodology, we formulate a single-station dynamic pricing problem with non-stationary demand as a QDP.
When the reward structure uses waiting and terminal costs, our approach can find near-optimal policies in seconds on a single CPU; when the reward structure uses penalties for deviating from service-level chance constraints, the optimization landscape is substantially more challenging yet our approach
can find a high-quality, practical policy in approximately a minute on a single CPU. 

\end{abstract}

\section{Introduction}
\label{sec:introduction}

This paper addresses dynamic control of queueing systems
that exhibit challenging characteristics, including
non-stationary arrivals, general (e.g., empirical) service distributions,
periods of significant overload,
and service-level chance constraints (e.g., bounded wait times with high probability).  
When tackling such complex systems, existing methods face practical difficulties.
Markov decision process (MDP) formulations often impose
restrictive assumptions such as stationary parameters
and exponentially distributed service durations to achieve tractability;
moreover, handling chance constraints is problematical in this framework.
Sampling-based alternatives such as reinforcement learning and
simulation-based optimization avoid explicit state enumeration but
suffer from slow convergence, high variance of gradient  estimators (especially in the presence of chance constraints), sensitivity to hyperparameters, and highly suboptimal solutions.

In this paper we propose
a fundamentally different approach that leverages QPLEX \cite{qplexbook}, a recently developed computational modeling and analysis methodology for (discrete-time, discrete-state) 
nonstationary stochastic systems.
QPLEX circumvents the curse of dimensionality associated with general service times
by reimagining information requirements to enable an orders-of-magnitude reduction in state space through conditional independence assumptions. 
This yields a deterministic forward iterative scheme that quickly and accurately generates approximate \emph{transient} distributions 
of performance metrics like the number of customers at stations of a queueing network. 
Since no sampling is involved, this iterative scheme acts as a \emph{deterministic} simulation.

At the core of the QPLEX forward iterative scheme lies a  \emph{nonlinear} Markov chain on the smaller state space,
meaning that  transition mechanism is both \emph{time-dependent} and \emph{distribution-dependent}: the transition probabilities from time $t$ to $t+1$
depend on the decision epoch $t$ and on the state distribution at time $t$.
This paper integrates QPLEX into  \emph{nonlinear} MDPs, a sequential decision-making framework built on nonlinear Markov chains.
A nonlinear MDP uses states, actions, rewards, and transition probabilities just like an MDP, except that the reward functions at time $t$ and the transition
probabilities from time $t$ to $t+1$ not only depend on the state and action chosen but also depend on the state distribution at time $t$.
A \emph{QPLEX Decision Process (QDP)} is a  nonlinear MDP that uses QPLEX-based transition probabilities.  
Any of the QPLEX models described in \cite{qplexbook} can be enriched with actions and rewards to obtain a QDP.  (The QPLEX book provides a foundation for formulating custom models, too.)

As transition probabilities and reward functions become both time-dependent and distribution-dependent,
a change in the policy at time $t$ creates a chain reaction affecting \emph{all} subsequent
transition probabilities and reward functions, so standard dynamic programming approaches no longer apply. Therefore,
we develop an MDP-style policy gradient framework for  nonlinear Markov chains that results in a backward iterative scheme for approximate gradient calculations,
complementing the QPLEX forward iterative scheme for performance evaluation.
It still works with the smaller state space,
and we show in an illustrative example that the backward scheme has
roughly the same computational cost as the forward scheme.
Crucially, no sampling is required.
Thus, our approach sidesteps the high variance inherent in gradient estimators,
in sharp contrast to sampling-based methods for gradient estimation using stochastic simulation.

Even with accurate gradient information using the backward scheme, 
standard first-order methods can struggle.
For instance, with softmax parameterization, standard gradient ascent can
force small step sizes due to high-curvature directions
that yield minimal progress in flatter regions.
So-called \emph{natural gradient ascent} methods aim to overcome this challenge
by exploiting the geometry underlying policy parametrization, but they require inversion of
a large Fisher information matrix.
This paper introduces a natural-gradient-inspired algorithm tailored to  nonlinear MDPs,
which remains computationally feasible by 
utilizing block-diagonal approximations of this Fisher information matrix.
For the special class of state-partitioned tabular policies, this leads to an exponential Q-ascent algorithm similar to familiar reinforcement learning counterparts, 
except that here the Q-functions are not state-action value functions due to nonlinearity.
We show that this algorithm can be viewed as an approximate version of natural gradient ascent when a softmax parameterization is used, and  
we also provide a performance guarantee near pure policies.

For an illustrative single-station dynamic pricing example with non-stationary demand and both waiting and terminal costs, we find policies within 1–2\% of theoretical optimality in seconds on a single CPU.
Standard sampling-based approaches such as Q-learning struggle with this problem due to large effective state spaces, failing to produce acceptable policies within practical time limits.
When there are service-level chance constraints, the optimization landscape is substantially more challenging yet our approach finds a high-quality policy in about a minute. We are unaware of other generic approaches for such decision problems.

\subsection{Literature Review}

\paragraph{Gradient estimation.}
Estimating gradients via stochastic simulation has long been a central challenge in discrete-event system optimization \cite{CassandrasLafortune2008}.
Black-box finite-difference methods are generally applicable but scale poorly with parameter dimensionality, while pathwise methods like Infinitesimal Perturbation Analysis (IPA) offer efficiency but require restrictive structural assumptions to yield valid gradients. Recent work \cite{CheEtAl2024} bridges this gap via differentiable surrogates, yet estimating reliable gradients from complex event dynamics remains computationally intensive.

Policy gradient methods (e.g., REINFORCE, actor–critic) optimize policies directly but face the challenge of high variance estimation.
Because gradients are estimated from sample paths \cite{Williams1992,SuttonEtAl2000}, convergence relies on variance-reduction techniques like baselines and control variates \cite{GreensmithBaxter2004,KondaTsitsiklis2000,Schulman2015TRPO,Schulman2017PPO,HendersonGlynn2002,LamZhou2017}. Despite these advances, the inherent noise in the gradient signal necessitates large sample sizes to achieve precision.

In contrast, our approach operates on a deterministic model of the system's state-distribution dynamics, eliminating Monte Carlo noise altogether. Rather than relying on sampling-based gradient estimation, we compute gradients from the QPLEX surrogate's distributional evolution. This provides a noise-free foundation for optimization, removing the fundamental variance that plagues all sampling-based approaches. This is particularly critical when handling chance constraints involving tail probabilities.

\paragraph{Nonlinear Markov chains and control.}
Nonlinear Markov processes arise naturally when the evolution of a system depends not just on individual states, but on the distribution of states across a population. This concept has deep roots in probability theory, particularly in the study of large interacting particle systems \cite{mckean1966,Sznitman1991,Kolokoltsov2010}. In the continuous-time optimal control setting, this framework has motivated extensive work on mean-field control and McKean-Vlasov dynamics \cite{CarmonaDelarueLachapelle2013}, where the goal is to optimize a large population's collective behavior. These approaches typically lead to a coupled system of PDEs: a Hamilton-Jacobi-Bellman equation for the optimal value function (which depends on the population distribution) and a Fokker-Planck equation for the evolution of that distribution (which depends on the optimal policy). Such coupled systems are generally intractable, both analytically and numerically.

In contrast, we develop a discrete-time computational framework for finite-horizon problems with time-varying transition and reward functions that depend on the state distribution. Our approach avoids differential equations. While related discrete-time formulations have been explored \cite{Baeuerle2023}, the novelty of this paper is a policy gradient framework with an accompanying algorithm.

 \paragraph{Natural gradient policy optimization.}
 Natural gradient methods rescale gradient updates using the Fisher information matrix rather than the objective Hessian, improving optimization efficiency and stability \cite{Amari1998,Kakade2002}. Practical implementations such as trust-region policy optimization (TRPO) \cite{Schulman2015TRPO}, Kronecker-factored curvature (K-FAC) \cite{MartensGrosse2015}, and related block-diagonal approximations make these methods computationally feasible. However, they remain inherently stochastic: both gradients and Fisher information must be estimated from sampled trajectories, introducing variance and computational overhead. Recent work analyzing natural policy gradients in various settings, including queueing control \cite{grosof2024convergence}, establishes convergence guarantees but still relies on Monte Carlo sampling.

Our approach enables noise-free computation of (approximate) gradients and Fisher information by working directly with the system's distributional evolution, allowing natural-gradient updates to be performed deterministically without sampling-based estimators. For computational efficiency, we employ block-diagonal Fisher approximations similar to K-FAC while retaining the variance-free nature of its curvature information.

\subsection{Notation and Terminology}


We often require differentiability of functions defined on convex subsets of Euclidean space. The most delicate case arises when the domain is lower-dimensional, such as the probability simplex or a convex subset thereof.
  We say such a function has the \emph{smooth extension property} if it admits an extension to a neighborhood of the relative interior of its domain, and this extension is continuously differentiable with the gradient extending by continuity to all boundary points in the original domain.
  This property enables the use of standard Euclidean gradients and partial derivatives in ordinary coordinates for all points in the original domain.
 It can typically be readily verified when the function admits an explicit algebraic representation. This ambient approach contrasts with intrinsic manifold methods, which would accommodate weaker differentiability assumptions but require more cumbersome presentation. While the smooth extension is not unique, the orthogonal projection of the resulting ambient gradient onto the tangent space of the affine hull of the domain \emph{is} unique. This ensures that the intrinsic gradient is not affected by the choice of extension.
  The practical benefit of this setup is significant: we work with standard vectors and matrices, and the chain rule applies directly in ordinary coordinates. 

Throughout this paper, we use subscripts and arguments interchangeably,
so we can take gradients with respect to either (if the function is suitably differentiable).
Gradients are column vectors.
Instead of formally defining our notational conventions regarding the use of the gradient symbol $\nabla$, we
use the illustrative example of 
a function $f_\theta(\mu)$ of a pmf $\mu$ on some underlying set and a parameter $\theta$ in some convex set in Euclidean space, where $f$ has the smooth extension property with respect to $\mu$ and $\theta$: 

\begin{enumerate}
  \item If $f$ is real-valued, then $\nabla_\mu f_{\theta}(\mu)$ denotes the gradient of $f$ with respect to $\mu$, holding $\theta$ constant,
    and $\nabla_\theta f_{\theta}(\mu)$ denotes the gradient of $f$ with respect to $\theta$, holding $\mu$ constant. 
We shall often reference the $s$-th component of the gradient $\nabla_\mu f_{\theta}(\mu)$ in ordinary coordinates,
    and we denote this partial derivative by $\frac{\partial f_{\theta}(\mu)}{\partial \mu(s)}$.
    \item If $f$ is vector-valued (e.g., pmf-valued), $\nabla_\mu f_{\theta}(\mu)$ and $\nabla_\theta f_\theta(\mu)$ denote the transposes of the Jacobian of $f_\theta(\mu)$ with respect to $\mu$ and $\theta$, respectively. 
    \item We use the notation $\nabla_{\mu} f_{\theta_0}(\mu_0)$ for the gradient $\nabla_\mu f_\theta(\mu)$ evaluated at $(\mu, \theta) = (\mu_0, \theta_0)$. Similarly, $\nabla_{\theta} f_{\theta_0 }(\mu_0)$ is the gradient $\nabla_\theta f_\theta(\mu)$ evaluated at $(\mu, \theta) = (\mu_0,\theta_0)$.
      Occasionally, for clarity, we use a vertical bar to signify evaluation,
  writing for instance $\nabla_\mu f_\theta(\mu) |_{\mu=\mu_0} $
  to denote the gradient $\nabla_\mu f_\theta(\mu)$ evaluated at $\mu=\mu_0$. 
\end{enumerate}

\subsection{Organization}
Section~\ref{sec:nonlinearMDPs} develops the nonlinear MDP framework and derives a policy gradient theorem, connecting it to the classical MDP case. Section~\ref{sec:statepartitioned} specializes to a tractable policy class, characterizes its local optima, and presents our natural gradient ascent algorithm with convergence guarantees. Section~\ref{sec:dynamicpricingexample} applies the framework to an illustrative dynamic pricing example, constructing the nonlinear transition probabilities (and reward functions) via QPLEX. Section~\ref{sec:numericalillustration} provides numerical experiments on challenging problem instances, demonstrating the practical efficiency of our approach. 
The appendices contain detailed derivations and proofs.

\section{Nonlinear Markov Decision Problems}
\label{sec:nonlinearMDPs}

This section introduces a sequential decision-making framework
built on \emph{nonlinear Markov chains}.
These processes were originally introduced by McKean~\cite{mckean1966}
and play an important role in mean-field stochastic models. 
We use a discrete-time, finite-dimensional, and nonstationary version.
A specification of a nonlinear Markov chain with a finite state space $\SSS$ requires a collection
of conditional pmfs $p^{(t)}_\mu(s'|s)$ parameterized by pmfs $\mu$ on $\SSS$ for $t=0,\ldots,T-1$ for some horizon $T > 0$,
along with an initial distribution $\mu^{(0)}$ on $\SSS$.
The process is then constructed in two steps.
In the first step, a sequence $\{\mu^{(t)}\}_{t=1}^T$ of pmfs on $\SSS$ is recursively defined
via
\begin{equation}
  \label{eq:zetaupdate}
  \mu^{(t+1)}(s' ) = \sum_{s\in\SSS} \mu^{(t)}(s)  \times p^{(t)}_{\mu^{(t)}}(s'|s).
\end{equation}
In the second step, the joint pmf of $(S^{(0)},\ldots, S^{(T)})$ is defined as
\begin{equation*}
  \mu^{(0)}(s^{(0)}) \times \prod_{t=0}^{T-1} p^{(t)}_{\mu^{(t)}}(s^{(t+1)}|s^{(t)}).
\end{equation*}
It can be verified that $\mu^{(t)}$ is the time-$t$ marginal pmf.
These processes derive their name from the observation that the right-hand side of (\ref{eq:zetaupdate}) can be nonlinear in $\mu^{(t)}$.

\subsection{Definition}
A \emph{discrete-time, finite-horizon  nonlinear Markov decision problem} consists of the following model elements: 
\begin{itemize}
\item The number of periods $T$.
\item A fixed underlying finite state space $\SSS = \{ 1, 2, \ldots, |\SSS| \}$.  We let $s, s'$ represent a generic element of $\SSS$.
We write $S^{(t)}$ for the state at time $t$, and we let $s^{(t)}$ represent a generic realization of $S^{(t)}$. 
\item A fixed underlying finite action space $\AAA$. We let $a$ represent a generic element of $\AAA$.
We write $A^{(t)}$ for the action at time $t$ taken upon observing
$S^{(t)}$, and $a^{(t)}$ represent a generic realization of $A^{(t)}$.
\item  \emph{Nonlinear} transition probabilities $p^{(t)}_\mu(s'|s,a)$, $t = 0, 1, \ldots, T-1$, parameterized by a pmf $\mu$ on the state space. 
  In line with nonlinear Markov chain terminology, we call the distribution $p^{(t)}_\mu$ the \emph{kernel} of the environment at time $t$.
We assume that, for each $(s,a,s')$, the restriction of
  $p^{(t)}_\mu(s'|s,a)$ to the set of pmfs $\mu$ on $\SSS$ with $\mu(s)>0$ has the smooth extension property.  
\item Markovian policy functions $\pi^{(t)}_{\theta^{(t)}}(a|s)$ parameterized by $\theta^{(t)}$ from some convex set $\Theta^{(t)}$ in Euclidean space.
  We require that, for each $(s,a)$, $\pi^{(t)}_{\theta^{(t)}}(a|s)$ has the smooth extension property as a function of $\theta^{(t)}$ on $\Theta^{(t)}$.
  
\item  Immediate reward functions $r^{(t)}_\mu(s,a)$, $t = 0, 1, \ldots, T-1$, parameterized by a pmf $\mu$ on the state space, and a 
  terminal reward function $r_{\mu}^{(T)}(s)$.
We assume that restrictions of the $r^{(t)}_\mu(s,a)$ and $r^{(T)}_\mu(s)$ to the set of pmfs $\mu$ on $\SSS$ with $\mu(s)>0$ have the smooth extension property.  

\end{itemize}

The table below classifies immediate reward and transition kernel
  primitives along two dimensions. Our setting is the highlighted in the lower-right cell.
In the standard MDP setting, these primitives are \emph{exogenous} in the sense that their values are known
before any computation is undertaken.  In our setting, these primitives are \emph{endogenous} in the sense that their values only become known after the
state distributions are calculated forward in time.
Our primitives are also nonstationary.

\vspace{2mm}
\begin{center}
\renewcommand{\arraystretch}{1.4}
\begin{tabular}{c|c|c}
 & \textbf{stationary} & \textbf{nonstationary} \\
\hline
\textbf{exogenous}
  & $p(s'|s,a)$, $r(s,a)$
  & $p^{(t)}(s'|s,a)$, $r^{(t)}(s,a)$ \\
\hline
\textbf{endogenous}
  & $p_\mu(s'|s,a)$, $r_\mu(s,a)$
  & \cellcolor{yellow!25} $p^{(t)}_\mu(s'|s)$, $r^{(t)}_\mu(s,a)$ \\
\end{tabular}
\end{center}
\vspace{2mm}

Throughout the remainder of this paper, unless otherwise stated, the symbol $t$ shall represent some non-negative time epoch less than $T$.  

Fix a choice of model elements, an initial $\mu^{(0)}$ on the state space, and a policy parameter vector $\ttheta = (\theta^{(0)}, \ldots, \theta^{(T-1)})$.  
A pmf $q_{\mu^{(0)}, \ttheta}$ of the random vector $(S^{(0)},A^{(0)}, \ldots, S^{(T-1)}, A^{(T-1)}, S^{(T)})$
can be defined in two steps.  In the first step, 
a sequence $\{ \mu_{\ttheta}^{(t)} \}_{t=0}^T$ of pmfs on the state space $\SSS$ is recursively defined, for each $s'\in \SSS$, via
\begin{equation}
   \label{eq:mutdef}
\begin{aligned}
  \mu^{(0)}_{\ttheta}(s')&=\mu^{(0)}(s')\\
  \mu^{(t+1)}_{\ttheta}(s') &= \sum_{s,a} \mu^{(t)}_{\ttheta}(s) \times \pi^{(t)}_{\theta^{(t)}}(a|s) \times p^{(t)}_{\mu^{(t)}_{\ttheta}} (s'| s, a).
\end{aligned}
\end{equation}
We emphasize that $\mu^{(t)}_{\ttheta}$ is constant in $\theta^{(t)},\ldots,\theta^{(T-1)}$,
but for conciseness we do not make this explicit in our notation.
In the second step, the pmf $q_{\mu^{(0)}, \ttheta}$ is defined via 
\begin{equation} \label{qdef0}
  q_{\mu^{(0)}, \ttheta}(s^{(0)}, a^{(0)}, \ldots, a^{(T-1)}, s^{(T)}) = \mu^{(0)}(s^{(0)}) \times \prod_{t=0}^{T-1}
  \left[\pi^{(t)}_{\theta^{(t)}}(a^{(t)}|s^{(t)}) \times p_{\mu^{(t)}_{\ttheta}}^{(t)}(s^{(t+1)}| s^{(t)}, a^{(t)})\right].
\end{equation}
It directly follows from (\ref{qdef0}) that
\begin{equation}
\label{eq:defNLMP}
q_{\mu^{(0)}, \ttheta}(s^{(t+1)}| s^{(0)},a^{(0)},\ldots, s^{(t)}, a^{(t)}) = p^{(t)}_{\mu^{(t)}_{\ttheta}} (s^{(t+1)}| s^{(t)}, a^{(t)}),
\end{equation}
and so the pmf $q_{\mu^{(0)}, \ttheta}$ satisfies
the \emph{nonlinear Markov property}:  the
future state depends only on the current state, the current action, \emph{and the distribution of the current state}, but not on the history of previous states or actions.
Note also that  
\begin{equation}
  \label{eq:piq}
q_{\mu^{(0)}, \ttheta}(a^{(t)}|s^{(0)},a^{(0)}, \ldots, s^{(t)}) = \pi^{(t)}_{\theta^{(t)}}(a^{(t)}|s^{(t)})
\end{equation}
thereby justifying the name of $\pi^{(t)}_{\theta^{(t)}}$.

Given that $q_{\mu^{(0)}, \ttheta}(s^{(0)}) = \mu^{(0)}(s^{(0)})$ and that
\begin{align*}
q_{\mu^{(0)}, \ttheta}(s^{(t+1)}) &= \sum_{s^{(t)}, a^{(t)}} q_{\mu^{(0)}, \ttheta}(s^{(t)}) \times q_{\mu^{(0)}, \ttheta}(a^{(t)} | s^{(t)}) \times q_{\mu^{(0)}, \ttheta}(s^{(t+1)} | s^{(t)}, a^{(t)}) \\
 &= \sum_{s^{(t)}, a^{(t)}} q_{\mu^{(0)}, \ttheta}(s^{(t)}) \times \pi^{(t)}_{\theta^{(t)}}(a^{(t)}|s^{(t)}) \times p^{(t)}_{\mu^{(t)}_{\ttheta}} (s^{(t+1)}| s^{(t)}, a^{(t)}),
\end{align*}
an inductive argument shows that 
\begin{equation}
  \label{eq:muq}
  q_{\mu^{(0)}, \ttheta}(s^{(t)}) =  \mu^{(t)}_{\ttheta}(s^{(t)}),
  \end{equation}
and so $\mu^{(t)}_{\ttheta}$ is the one-dimensional marginal pmf of $S^{(t)}$ under $q_{\mu^{(0)}, \ttheta}$.

For the remainder of this paper, we shall consider $\mu^{(0)}$ fixed and suppress the functional dependence of $\mu^{(0)}$ on $q_{\mu^{(0)}, \ttheta}$.
We denote the expectation operator corresponding to $q_{\ttheta}$ by $\E_{\ttheta}$.

The \emph{nonlinear Markov decision problem} can now be formally stated as:
\begin{equation*} 
\sup_{\ttheta} \E_{\ttheta} \left[ \sum_{t=0}^{T-1} r^{(t)}_{\mu^{(t)}_{\ttheta}}(S^{(t)}, A^{(t)})  + r^{(T)}_{\mu^{(T)}_{\ttheta}}(S^{(T)}) \right].
\end{equation*}
By (\ref{eq:defNLMP}), (\ref{eq:piq}), and (\ref{eq:muq}), the expected immediate reward at time $t < T$ and expected final reward under $q_{\ttheta}$ can be expressed as 
\begin{align}
\E_{\ttheta} \left[r^{(t)}_{\mu^{(t)}_{\ttheta}}(S^{(t)}, A^{(t)})\right] &=  \sum_{s,a} \mu^{(t)}_{\ttheta}(s) \times \pi^{(t)}_{\theta^{(t)}}(a|s) \times  r^{(t)}_{\mu^{(t)}_{\ttheta}} (s,a) \label{expimmrewardtimet} \\
\E_{\ttheta} \left[r^{(T)}_{\mu^{(T)}_{\ttheta}}(S^{(T)})\right] &= \sum_{s} \mu^{(T)}_{\ttheta}(s) \times r^{(T)}_{\mu^{(T)}_{\ttheta}} (s), \nonumber
\end{align}
and so this decision problem can be represented as 
\begin{equation}
  \label{eq:maxJ}
\sup_{\ttheta} J_{\ttheta},
\end{equation}
where
\begin{equation}
  \label{eq:Jtheta}
  J_{\ttheta} = \sum_{t=0}^{T-1} \left( \sum_{s,a} \mu^{(t)}_{\ttheta}(s) \times \pi^{(t)}_{\theta^{(t)}}(a|s) \times  r^{(t)}_{\mu^{(t)}_{\ttheta}} (s,a) \right) + \sum_{s} \mu^{(T)}_{\ttheta}(s) \times r^{(T)}_{\mu^{(T)}_{\ttheta}} (s)
  \end{equation}

Evidently, finite-horizon Markov decision processes are a special case upon letting
the kernels and rewards not depend on $\mu$.

\subsection{The Policy Gradient Theorem}

The algorithms we develop for the decision problem (\ref{eq:maxJ}) require the gradients of $J_{\ttheta}$ with respect to the $\theta^{(t)}$
for all $t < T$.  A change in $\theta^{(t)}$ directly affects the expected immediate reward at time $t$ in (\ref{expimmrewardtimet}), 
since the probabilities of realizing each action given a state at time $t$ will be changed.  
In sharp contrast to the Markov setting, in this nonlinear setting a change in $\theta^{(t)}$ also indirectly affects \emph{all} subsequent rewards
because it directly affects the pmf $\mu^{(t+1)}_{\ttheta}$, see (\ref{eq:mutdef}), and thus \emph{all} subsequent pmfs $\mu^{(t+2)}_{\ttheta},\ldots,\mu^{(T)}_{\ttheta}$, too.

\begin{theorem}[Policy Gradient Theorem]
  \label{thm:pgt}
Suppose a pmf $\mu^{(0)}$ on initial states and
a vector of policy parameters $\ttheta$ are given.
The gradient $\nabla_{\theta^{(t)}} J_{\ttheta}$ of the expected total reward with respect to $\theta^{(t)}$ satisfies
\begin{equation}
  \label{eq:pgt}
  \nabla_{\theta^{(t)}} J_{\ttheta} = \sum_{s, a} \mu^{(t)}_{\ttheta}(s) \times \left(\nabla_{\theta^{(t)}} \pi^{(t)}_{\theta^{(t)}} (a | s) \right)\times Q_{\mu^{(t)}_{\ttheta},\sigma^{(t+1)}_{\ttheta}}^{(t)}(s, a),
  \end{equation}
  where $Q^{(t)}_{\mu,\sigma'}$ is defined as 
 \begin{equation*}
Q^{(t)}_{\mu, \sigma'}(s, a) =  r^{(t)}_{\mu}(s, a) + \sum_{s'}  p^{(t)}_{\mu}(s' | s, a) \times \sigma'(s'),
\end{equation*}
the $\mu^{(t)}_{\ttheta}$ can be calculated with (\ref{eq:mutdef}),
and the $\sigma^{(t)}_{\ttheta}$ can be calculated via 
\begin{equation}
 \label{eq:sigmatrec}
\begin{aligned}
\sigma^{(T)}_{\ttheta} (s) &= r^{(T)}_{\mu^{(T)}_{\ttheta}}(s)  + \sum_{\tilde s} \mu^{(T)}_{\ttheta}(\tilde s)\times \frac{\partial}{\partial \mu(s)} r^{(T)}_{\mu^{(T)}_{\ttheta}}(\tilde s) \\
\sigma^{(t)}_{\ttheta}(s) &= \sum_a \pi^{(t)}_{\theta^{(t)}} (a | s) \times Q^{(t)}_{\mu^{(t)}_{\ttheta}, \sigma^{(t+1)}_{\ttheta}}(s, a) + 
\sum_{\tilde s, a} \mu^{(t)}_{\ttheta} (\tilde s) \times \pi^{(t)}_{\theta^{(t)}} (a | \tilde s) \times \frac{\partial}{\partial \mu(s)} Q^{(t)}_{\mu^{(t)}_{\ttheta}, \sigma^{(t+1)}_{\ttheta}} (\tilde s, a).
\end{aligned}
\end{equation}
\end{theorem}

The figure below depicts the iterative schemes involved in the gradient calculation.
  Arrows indicate the computational dependencies.
  The top row shows the forward scheme defined in (\ref{eq:mutdef}),
  the bottom row shows the backward scheme in (\ref{eq:sigmatrec}), and
  the middle row illustrates how the gradient components are computed from (\ref{eq:pgt}).

\begin{center}
\begin{tikzpicture}[every node/.style={rectangle,draw}, box/.style={minimum width=1.8cm,minimum height=.95cm,rounded corners=1mm},thick,xscale=1.5,yscale=1.2]%
  \node (mu) at (0,2) [box] {\strut$\mu^{(0)}_{\ttheta}$};
  \node (mup) at (2,2) [box] {\strut$ \mu^{(1)}_{\ttheta}$};
  \node (mupp) at (4,2) [box] {\strut$ \mu^{(2)}_{\ttheta}$};
  \node[draw=none] (tdots) at (5.5,2) {$\cdots$};
  \node (muppp) at (7,2) [box] {\strut$ \mu^{(T-1)}_{\ttheta}$};
  \node (mupppp) at (9,2) [box] {\strut$ \mu^{(T)}_{\ttheta}$};
    \node (sigma) at (0,-2) [box] {\strut$\sigma^{(0)}_{\ttheta}$};
  \node (sigmap) at (2,-2) [box] {\strut$ \sigma^{(1)}_{\ttheta}$};
  \node (sigmapp) at (4,-2) [box] {\strut$ \sigma^{(2)}_{\ttheta}$};
  \node[draw=none] (sigmapppghost) at (6,-2) {};
  \node[draw=none] (bdots) at (5.5,-2) {$\cdots$};
  \node (sigmappp) at (7,-2) [box] {\strut$ \sigma^{(T-1)}_{\ttheta}$};
  \node (sigmapppp) at (9,-2) [box] {\strut$ \sigma^{(T)}_{\ttheta}$};
  \node (omega) at (1,0) [box] {\strut$\nabla_{\theta^{(0)}} J_{\ttheta}$};
  \node (omegap) at (3,0) [box] {\strut$ \nabla_{\theta^{(1)}} J_{\ttheta}$};
  \node (omegapp) at (5,0) [box] {\strut$ \nabla_{\theta^{(2)}} J_{\ttheta}$};
  \node[draw=none] (omegappghost) at (6, 0) {};
  \node[draw=none] (cdots) at (6.5,0) {$\cdots$};
  \node (omegappp) at (8,0) [box] {\strut$ \nabla_{\theta^{(T-1)}} J_{\ttheta}$};
  \path[draw,ultra thick,->] (mu) to (mup);
  \path[draw,ultra thick,->] (mup) to (mupp);
  \path[draw,ultra thick,-] (mupp) to (tdots);
  \path[draw,ultra thick,->] (tdots) to (muppp);
  \path[draw,ultra thick,->] (muppp) to (mupppp);
  \path[draw,ultra thick,->] (sigmap) to (sigma);
  \path[draw,ultra thick,->] (sigmapp) to (sigmap);
  \path[draw,ultra thick,->] (bdots) to (sigmapp);
  \path[draw,ultra thick,-] (sigmappp) to (bdots);
  \path[draw,ultra thick,->] (sigmapppp) to (sigmappp);
  \path[draw,ultra thick,->] (sigmap) to ($(omega.south)+(.22,0)$);
  \path[draw,ultra thick,->] (sigmapp) to ($(omegap.south)+(.22,0)$);
  \path[draw,ultra thick,-] (sigmappp) to ($(sigmappp)!0.4!(omegappghost.south east)$);
  \node[draw=none] (cbdots) at (6,-1.2) {$\ddots$};
  \path[draw,ultra thick,->] ($(sigmapppghost)!0.75!(omegapp.south east)-(.25,0)$) to ($(omegapp.south)+(.22,0)$);
  \path[draw,ultra thick,->] (sigmapppp) to ($(omegappp.south)+(.22,0)$);;
  \path[draw,ultra thick,->] (mu) to ($(omega.north)-(.22,0)$);
  \path[draw,ultra thick,->] (mup) to ($(omegap.north)-(.22,0)$);
  \path[draw,ultra thick,->] (mupp) to ($(omegapp.north)-(.22,0)$);
  \path[draw,ultra thick,->] (muppp) to ($(omegappp.north)-(.22,0)$);
  \path[draw,ultra thick,->] (mu) to (sigma);
  \path[draw,ultra thick,->] (mup) to (sigmap);
  \path[draw,ultra thick,->] (mupp) to (sigmapp);
  \path[draw,ultra thick,->] (muppp) to (sigmappp);
  \path[draw,ultra thick,->] (mupppp) to (sigmapppp);
\end{tikzpicture}
\end{center}

\subsection{Connection to the Markov Chain Setting}
The function $Q^{(t)}_{\mu^{(t)}_{\ttheta}, \sigma^{(t+1)}_{\ttheta}}$
appearing in the policy gradient theorem
is \emph{not} the state-action value function. 
In this subsection, we explain why and
also relate our theorem to the standard textbook form
of the policy gradient theorem. 

In the (classical) Markov chain case,
the familiar backward recursion for the state value function $V^{(t)}_{\ttheta,\Markov}$ for Markov decision problems is 
\begin{equation}
  \label{eq:recMarkovV}
  \begin{aligned}
    V^{(T)}_{\ttheta,\Markov}(s) &= r^{(T)}_{\Markov}(s) \\
    V^{(t)}_{\ttheta,\Markov}(s) &= \sum_a \pi^{(t)}_{\theta^{(t)}}(a|s)\times \left[r^{(t)}_{\Markov} (s,a) + \sum_{s'} p^{(t)}_{\Markov} (s'|s,a)\times V^{(t+1)}_{\ttheta,\Markov}(s') \right],
    \end{aligned}
\end{equation}
which can be expressed using state-action value functions $Q_{\ttheta,\Markov}^{(t)}$ as
\begin{equation}
  \label{eq:recMarkovVQ}
  \begin{aligned}
    V^{(T)}_{\ttheta,\Markov}(s) &= r^{(T)}_{\Markov} (s) \\
    Q_{\ttheta,\Markov}^{(t)}(s, a) &= r^{(t)}_{\Markov} (s,a) + \sum_{s'} p^{(t)}_{\Markov} (s'|s,a)\times V^{(t+1)}_{\ttheta,\Markov}(s')  \\
    V^{(t)}_{\ttheta,\Markov}(s) &= \sum_a \pi^{(t)}_{\theta^{(t)}}(a|s)\times Q_{\ttheta,\Markov}^{(t)}(s, a).
    \end{aligned}
  \end{equation}
Thus, in that case, the policy gradient theorem
can be obtained immediately upon noting that the expected total reward 
\begin{equation*}
J_{\ttheta, \Markov} = \sum_{t=0}^{T-1} \sum_{s,a} \mu^{(t)}_{\ttheta}(s) \times \pi^{(t)}_{\theta^{(t)}}(a|s) \times r^{(t)}_{\Markov} (s,a)  + \sum_s \mu^{(T)}_{\ttheta}(s)\times r^{(T)}_{\Markov}(s)
\end{equation*}
can be written as a sum over the first $t$ expected immediate rewards plus
$\sum_{s,a} \mu^{(t)}_{\ttheta}(s) \times \pi^{(t)}_{\theta^{(t)}} (a|s)\times Q^{(t)}_{\ttheta,\Markov}(s,a)$. 
Since $Q^{(t)}_{\ttheta, \Markov}$ and $\mu^{(\tau)}_{\ttheta}$ for $\tau\le t$ are constant in $\theta^{(t)}$, the policy gradient theorem follows:
\begin{equation} \label{eq:JTMarkovgradexpression}
\nabla_{\theta^{(t)}}J_{\ttheta,\Markov}=
\sum_{s,a} \mu^{(t)}_{\ttheta}(s)\times \nabla_{\theta^{(t)}} \pi^{(t)}_{\theta^{(t)}}(a|s) 
\times Q^{(t)}_{\ttheta,\Markov}(s,a).
\end{equation}

Now consider the nonlinear Markov chain case.
Fix $t$ and policy parameter $\theta^{(t)}\in\Theta^{(t)}$.  For a pmf $\mu$ on the state space $\SSS$, define the pmf $\PPP^{(t)}_{\theta^{(t)}}(\mu)$ on $\SSS$ via
\begin{equation} \label{eq:Ptmuthetadef}
[\PPP^{(t)}_{\theta^{(t)}}(\mu)](s') = \sum_{s,a}\mu(s)\times \pi^{(t)}_{\theta^{(t)}}(a|s) \times p^{(t)}_\mu(s'|s,a).
\end{equation}
Note that $\PPP^{(t)}_{\theta^{(t)}}(\mu^{(t)}_{\ttheta}) = \mu^{(t+1)}_{\ttheta}$ for a fixed policy vector $\ttheta$.  
Recursively define the real-valued functions $V^{(t)}_{\mu, \ttheta}$ by setting, for each $s\in \SSS$,
\begin{equation}
\label{eq:nlmcVrec}
\begin{aligned}  
V^{(T)}_{\mu, \ttheta}(s) &= r^{(T)}_{\mu} (s) \\
 V^{(t)}_{\mu, \ttheta}(s) &= \sum_a \pi^{(t)}_{\theta^{(t)}}(a | s) \times r^{(t)}_{\mu}(s, a)  + \sum_{a, s'}  \pi^{(t)}_{\theta^{(t)}}(a | s) \times p_{\mu}^{(t)}(s' | s, a) \times V_{\PPP^{(t)}_{\theta^{(t)}}(\mu), \ttheta}^{(t+1)}(s'),
\end{aligned}
\end{equation}
and define the state value function $\check V^{(t)}_{\ttheta}$ via $\check V^{(t)}_{\ttheta}(s) = V^{(t)}_{\mu^{(t)}_{\ttheta},\ttheta}(s)$.
Expanding (\ref{eq:nlmcVrec})  
in the same manner as we did to derive (\ref{eq:recMarkovVQ}) from (\ref{eq:recMarkovV}) and substituting the appropriate $\mu^{(t)}_{\ttheta}$ for $\mu$
results in the recursion
  \begin{equation*}
\begin{aligned}  
  \check V^{(T)}_{\ttheta}(s) &= r^{(T)}_{\mu^{(T)}_{\ttheta}}(s) \\
  \check Q^{(t)}_{\ttheta}(s,a) &=r^{(t)}_{\mu^{(t)}_{\ttheta}}(s, a)  + \sum_{ s'} p_{\mu^{(t)}_{\ttheta}}^{(t)}(s' | s, a) \times \check V_{\ttheta}^{(t+1)}(s')\\
 \check V^{(t)}_{\ttheta}(s) &= \sum_a \pi^{(t)}_{\theta^{(t)}}(a | s) \times  \check Q^{(t)}_{\ttheta}(s,a),
\end{aligned}
\end{equation*}
where $\check Q^{(t)}_{\ttheta}$ is the state-action value function.
In contrast to this recursion for the state value function,
the recursion for the $\sigma^{(t)}_{\ttheta}$ given by 
\begin{equation*}
\begin{aligned}  
  \sigma^{(T)}_{\ttheta}(s) &= r^{(T)}_{\mu^{(T)}_{\ttheta}} (s) \\
  Q^{(t)}_{\mu^{(t)}_{\ttheta}, \sigma^{(t+1)}_{\ttheta}}(s, a) &= r^{(t)}_{\mu^{(t)}_{\ttheta}}(s, a) + \sum_{s'}  p^{(t)}_{\mu^{(t)}_{\ttheta}}(s' | s, a) \times \sigma^{(t+1)}_{\ttheta}(s')\\
\sigma^{(t)}_{\ttheta}(s) &= \sum_a \pi^{(t)}_{\theta^{(t)}} (a | s) \times Q^{(t)}_{\mu^{(t)}_{\ttheta}, \sigma^{(t+1)}_{\ttheta}}(s, a) + 
\sum_{\tilde s, a} \mu^{(t)}_{\ttheta} (\tilde s) \times \pi^{(t)}_{\theta^{(t)}} (a | \tilde s) \times \frac{\partial}{\partial \mu(s)} Q^{(t)}_{\mu^{(t)}_{\ttheta}, \sigma^{(t+1)}_{\ttheta}} (\tilde s, a)\\
\end{aligned}
\end{equation*}
is fundamentally \emph{different} due to the presence of the extra partial derivative term in the last equation. This is a
consequence of the fact that changing $\mu^{(t)}_{\ttheta}$ also changes
the transition probabilities $p_{\mu^{(t)}_{\ttheta}}^{(t)}(s'|s,a)$ in the nonlinear setting.

The policy gradient theorem in the nonlinear Markov chain case
uses $Q^{(t)}_{\mu^{(t)}_{\ttheta}, \sigma^{(t+1)}_{\ttheta}}$, which \emph{includes} this
extra term, instead of the state-action value function $\check Q^{(t)}_{\ttheta}$.
In the Markov chain case, where $r^{(t)}_\mu$ and $p^{(t)}_\mu(s'|s,a)$ do not depend on $\mu$, the two functions $Q^{(t)}_{\mu^{(t)}_{\ttheta}, \sigma^{(t+1)}_{\ttheta}}$ and $\check Q^{(t)}_{\ttheta}$ are identical.
We see that $Q^{(t)}_{\mu^{(t)}_{\ttheta}, \sigma^{(t+1)}_{\ttheta}}$
is \emph{not} the state-action value function, even though we use the symbol $Q$.

We next relate our policy gradient theorem to
the policy gradient theorem found in textbooks such as \cite{SuttonBarto2018}.
We stay in the Markov setting.
After augmenting the state with time to incorporate nonstationarity,
the collection of value functions $\{V^{(t)}_{\ttheta,\Markov}: 0\le t\le T\}$ can be viewed as a \emph{single} value function on the set of augmented states,
where the value at $(s,t)$ is given by $V^{(t)}_{\ttheta,\Markov}(s)$.
Similarly, the value at an augmented state-action pair $(s,t,a)$ is 
$Q^{(t)}_{\ttheta,\Markov}(s,a)$.
The agent's policy is modeled via $\tilde \pi_\zeta(a|s,t)$
using a single---but typically high-dimensional---policy parameter $\zeta$.
The probability $\tilde \mu_\zeta(s,t)$ of ever reaching the augmented state $(s,t)$
can be calculated recursively via  
\begin{equation}
  \label{eq:forwardextendedstate}
  \begin{aligned}
\tilde \mu_\zeta(s,0)&=\mu^{(0)}(s)\\
\tilde \mu_\zeta(s',t+1) &= \sum_{s,a} \tilde \mu_\zeta(s,t) \times \tilde \pi_\zeta(a|s,t) \times p^{(t)}_{\Markov} (s'|s,a).
\end{aligned}
\end{equation}
Writing $\tilde J_{\zeta,\Markov}$ for the expected total reward,
the policy gradient theorem in \cite[Section~13.2]{SuttonBarto2018} takes the form
\begin{equation*}
  \nabla_{\zeta}\tilde J_{\zeta,\Markov}=\sum_{s,t} \tilde \mu_\zeta(s,t)\times \sum_{a}
  \nabla_{\zeta} \tilde \pi_{\zeta}(a|s,t) 
\times Q^{(t)}_{\ttheta,\Markov}(s,a).
\end{equation*}
Suppose that $\zeta$ is taken as the vector $\ttheta=(\theta^{(0)},\ldots,\theta^{(T-1)})$
and that $\tilde \pi_\zeta(a|s,t)$ takes the form $\pi^{(t)}_{\theta^{(t)}}(a|s)$,
so we have that $\tilde \mu_\zeta(s,t) = \tilde \mu^{(t)}_{\ttheta}(s)$ and $\tilde J_{\zeta,\Markov}=J_{\ttheta,\Markov}$.
Since $\nabla_{\zeta} \tilde J_{\zeta,\Markov}$ can then be written as $(\nabla_{\theta^{(0)}} J_{\ttheta,\Markov},\ldots,\nabla_{\theta^{(T-1)}}J_{\ttheta,\Markov})^{\top}$, 
the standard policy gradient for $t$ takes the form in (\ref{eq:JTMarkovgradexpression}).

\subsection{Natural Gradient}
While the standard gradient $\nabla_{\ttheta} J_{\ttheta}$ points in the direction of steepest ascent in Euclidean parameter space, the \emph{natural gradient} 
$\widehat \nabla_{\ttheta} J_{\ttheta}$ identifies the steepest direction according to the intrinsic geometry of the parameter space, as measured by the Fisher information metric \cite{Amari1998,Kakade2002}. 
Specifically, the natural gradient is defined via
\begin{equation*}
\widehat \nabla_{\ttheta} J_{\ttheta}= F(\ttheta)^\dagger \times \nabla_{\ttheta} J_{\ttheta},
\end{equation*}
where $F(\ttheta)^{\dagger}$ is the Moore-Penrose inverse of the \emph{Fisher information matrix} 
\begin{equation*}
F(\ttheta) =  \E_{\ttheta} \Big( \nabla_{\ttheta} \log q_{\ttheta}(S^{(0)},A^{(0)},\ldots, S^{(T)})  \times \left(\nabla_{\ttheta} \log q_{\ttheta}(S^{(0)},A^{(0)},\ldots, S^{(T)}) \right)^\top \Big). 
\end{equation*}
The square matrix $F(\ttheta)$
can be represented using the following block structure:  
\begin{equation*}
F(\ttheta) = \left[
\begin{array}{cccc}
F^{(0, 0)}(\ttheta) & F^{(0, 1)}(\ttheta) & \cdots & F^{(0,T-1)}(\ttheta) \\
F^{(1, 0)}(\ttheta) & F^{(1, 1)}(\ttheta) & \cdots & F^{(1,T-1)}(\ttheta) \\
\vdots     &  \vdots    & \ddots & \vdots \\
F^{(T-1, 0)}(\ttheta) & F^{(T-1, 1)}(\ttheta) & \cdots & F^{(T-1,T-1)}(\ttheta) \\
\end{array}
\right],
\end{equation*}
where the matrices $F^{(t_1,t_2)}(\ttheta)$ are given by
\begin{equation} \label{Ft1t2def}
F^{(t_1, t_2)}(\ttheta) = \E_{\ttheta} \left( \nabla_{\theta^{(t_1)}} \log q_{\ttheta}(S^{(0)},A^{(0)},\ldots, S^{(T)})  \times \left(\nabla_{\theta^{(t_2)}} \log q_{\ttheta}(S^{(0)},A^{(0)},\ldots, S^{(T)}) \right)^\top \right).
\end{equation}
Of course, $F^{(t_2, t_1)}(\ttheta)$ is the transpose of $F^{(t_1, t_2)}(\ttheta)$.  
Here the expectation of a matrix should be interpreted elementwise.

Proposition \ref{prop:fisherrecursion} below shows how to calculate each block $F^{(t_1, t_2)}(\ttheta)$ of $F(\ttheta)$
in terms of the `policy function' Fisher information matrices $K^{(t)}(\ttheta)$ defined by
\begin{equation*}
  K^{(t)}(\ttheta) = \E_{\ttheta} \left(\nabla_{\theta^{(t)}} \log \pi^{(t)}_{\theta^{(t)}}(A^{(t)} | S^{(t)}) \times (\nabla_{\theta^{(t)}} \log \pi^{(t)}_{\theta^{(t)}}(A^{(t)} | S^{(t)}))^\top \right)
\end{equation*}
and the transition kernel Fisher information matrices $M^{(t)}(\ttheta)$ defined by
\begin{equation*}
 M^{(t)}(\ttheta) =  \E_{\ttheta} \left(\nabla_\mu \log p^{(t)}_{\mu^{(t)}_{\ttheta}}(S^{(t+1)}|S^{(t)}, A^{(t)})\times\left(\nabla_\mu \log p^{(t)}_{\mu^{(t)}_{\ttheta}}(S^{(t+1)}|S^{(t)},A^{(t)})\right)^\top\right).
  \end{equation*}
It is sufficient to focus on $0 \le t_1 \le t_2 \le T-1$ due to symmetry of $F(\ttheta)$.

In the statement of the proposition below (and for the remainder of this paper) we use the shorthand notation $\nabla_{\mu^{(t)}_{\ttheta}} \mu^{(t+1)}_{\ttheta} =
  \nabla_{\mu}  \PPP^{(t)}_{\theta^{(t)}}(\mu)\big|_{\mu=\mu^{(t)}_{\ttheta}}$ and 
  $\nabla_{\theta^{(t)}} \mu^{(t+1)}_{\ttheta} =  \nabla_{\theta^{(t)}}  \PPP^{(t)}_{\theta^{(t)}}(\mu) \big|_{\mu=\mu^{(t)}_{\ttheta}}$.
  Using this notation, for $t_1<t_2$, we furthermore write
\begin{equation*}
  \nabla_{\theta^{(t_1)}} \mu^{(t_2+1)}_{\ttheta} = (\nabla_{\theta^{(t_1)}} \mu^{(t_1+1)}_{\ttheta}) \times (\nabla_{\mu^{(t_1+1)}_{\ttheta}} \mu^{(t_1+2)}_{\ttheta}) \times \cdots\times
 (\nabla_{\mu^{(t_2)}_{\ttheta}} \mu^{(t_2+1)}_{\ttheta}),
\end{equation*}
which is the chain rule when 
$\mu^{(t_2+1)}_{\ttheta}$ is viewed as the composition of $\PPP^{(t_2)}_{\theta^{(t_2)}},\ldots, \PPP^{(t_1)}_{\theta^{(t_1)}}$ evaluated at $\mu^{(t_1)}_{\ttheta}$.

\begin{proposition}
  \label{prop:fisherrecursion}
For the diagonal blocks of $F(\ttheta)$, we have 
\begin{equation*}
  F^{(t, t)}(\ttheta) =
    K^{(t)}(\ttheta)   + (\nabla_{\theta^{(t)}} \mu^{(t+1)}_{\ttheta}) \times G^{(t+1)}(\ttheta) \times (\nabla_{\theta^{(t)}} \mu^{(t+1)}_{\ttheta})^\top.
\end{equation*}
For the off-diagonal blocks, with $0 \le t_1 < t_2 \le T-1$, we have
\begin{equation*}
  F^{(t_1, t_2)}(\ttheta) =
  (\nabla_{\theta^{(t_1)}} \mu^{(t_2+1)}_{\ttheta}) \times G^{(t_2+1)}(\ttheta) \times (\nabla_{\theta^{(t_2)}} \mu^{(t_2+1)}_{\ttheta})^\top,
\end{equation*}
where the $G^{(t)}(\ttheta)$ can be recursively calculated using
$G^{(T)}(\ttheta)=0$ and, for $t=1,\ldots,T-1$,
\begin{equation*}
G^{(t)}(\ttheta)  = M^{(t)}(\ttheta)+(\nabla_{\mu^{(t)}_{\ttheta}} \mu^{(t+1)}_{\ttheta}) \times G^{(t+1)} (\ttheta)\times (\nabla_{\mu^{(t)}_{\ttheta}} \mu^{(t+1)}_{\ttheta})^\top.
\end{equation*}
\end{proposition}

\section{State-Partitioned Tabular Policies}
\label{sec:statepartitioned}

There are several challenges to calculate the natural gradient from the Fisher information matrix.
First, the matrices $\nabla_{\theta^{(t)}} \mu^{(t+1)}_{\ttheta}$ and $\nabla_{\mu^{(t)}_{\ttheta}} \mu^{(t+1)}_{\ttheta}$
must be calculated, which may be expensive depending on the size of the state space
and the dimension of the parameter space.
Second, appropriate matrix products must be calculated to find the $G^{(t)}(\ttheta)$.
Third, the policy Fisher information matrix $F(\ttheta)$ needs to be inverted, which
may again come at a significant expense.
In this section we describe and justify a computationally tractable algorithm
in the important special case when the policy is a \emph{state-partitioned tabular policy}.

\subsection{Preliminaries}
A state-partitioned tabular policy requires a deterministic partition of the state space $\SSS$ into disjoint subsets $\SSS_z$ indexed by a finite set $\ZZZ$, where each subset $\SSS_z$ is assigned to a distinct \emph{tabular expert} $z \in \ZZZ$.
The policy parameter $\theta^{(t)}$ becomes a collection
$(\theta^{(t)}_z)_{z \in \ZZZ}$.
Each $\theta^{(t)}_z$ is a pmf over the action set $\AAA$, so the set
$\Theta^{(t)}$ of feasible parameters is the $|\ZZZ|$-fold Cartesian product of the probability simplex over $\AAA$.
The policy is then given by
\begin{equation}
  \label{eq:partitionedtabular}
  \pi^{(t)}_{\theta^{(t)}}(a|s) = \theta^{(t)}_{z(s)}(a),
  \end{equation}
where $z(s)$ denotes the deterministic assignment function mapping state $s$ to its corresponding expert.
A \emph{standard tabular policy} corresponds to the special case
when each expert is identified with a corresponding state so that $|\ZZZ|=|\SSS|$ and each $\SSS_z$ is a singleton.
A policy parameter vector $\ttheta$ is a \emph{pure policy} if, for each $z$ and $t$,
the support of $\theta_z^{(t)}$ is a singleton.

The forward iterative scheme takes the form
\begin{equation}
  \label{eq:expertforward}
\begin{aligned}
  \mu^{(0)}_{\ttheta}(s')&=\mu^{(0)}(s')\\
  \mu^{(t+1)}_{\ttheta}(s') &= \sum_{s,a} \mu^{(t)}_{\ttheta}(s) \times \theta^{(t)}_{z(s)}(a) \times p^{(t)}_{\mu^{(t)}_{\ttheta}} (s'| s, a)
\end{aligned}
\end{equation}
and the backward iterative scheme takes the form
\begin{equation}
  \label{eq:expertbackward}
  \begin{aligned}
    \sigma^{(T)}_{\ttheta}(s)&=r^{(T)}_{\mu^{(T)}_{\ttheta}}(s)+\sum_{\tilde s}\mu^{(T)}_{\ttheta}(\tilde s)\times \frac{\partial}{\partial \mu(s)}  r^{(T)}_{\mu^{(T)}_{\ttheta}}(\tilde s)\\
     \sigma^{(t)}_{\ttheta}(s) &=
\sum_a \theta^{(t)}_{z(s)} (a) \times Q^{(t)}_{\mu^{(t)}_{\ttheta}, \sigma^{(t+1)}_{\ttheta}}(s, a) + 
\sum_{\tilde s, a} \mu^{(t)}_{\ttheta} (\tilde s) \times \theta^{(t)}_{z(\tilde s)} (a) \times
\frac{\partial}{\partial \mu(s)} Q^{(t)}_{\mu^{(t)}_{\ttheta},
  \sigma^{(t+1)}_{\ttheta}} (\tilde s, a),
\end{aligned}
\end{equation}
where we use, as before,
\begin{equation}
  \label{eq:expertQ}
Q^{(t)}_{\mu,\sigma'}(s,a)=  r^{(t)}_{\mu}(s,a) + \sum_{s'} p^{(t)}_{\mu}(s'|s,a)\times \sigma'(s').
\end{equation}

The function $\overline Q^{(t)}_{z,\ttheta}$ on $\AAA$
associated with time $t$ and expert $z$ defined via
\begin{equation}
  \label{eq:barQ}
  \overline Q^{(t)}_{z, \ttheta}(a) = \begin{cases}
    \sum_{s\in\SSS_z} \mu^{(t)}_{\ttheta} (s|\SSS_z)\times
    Q^{(t)}_{\mu^{(t)}_{\ttheta},\sigma^{(t+1)}_{\ttheta}}(s,a) & \text{if } \mu_{\ttheta}^{(t)}(\SSS_z) > 0\\
    0 &\text{otherwise}
    \end{cases}
  \end{equation}
is a weighted average of the $Q^{(t)}$ function and plays a pivotal role in our algorithm.
Here the conditional pmf on states $\mu^{(t)}_{\ttheta}(s|\SSS_z)$ associated with time $t$ and expert $z$
  with $\mu^{(t)}_{\ttheta}(\SSS_z)>0$ is given by, for $s\in\SSS_z$, 
\begin{equation*}
 \mu^{(t)}_{\ttheta}(s|\SSS_z) =\frac{\mu^{(t)}_{\ttheta}(s)}{\mu^{(t)}_{\ttheta}(\SSS_z)}.
\end{equation*}
For a standard tabular policy, we simply have 
$\overline Q^{(t)}_{z, \ttheta}(a) =Q^{(t)}_{\mu^{(t)}_{\ttheta}, \sigma^{(t+1)}_{\ttheta}} (z, a)$.

We close this subsection with a characterization of the local maxima of $ J_{\ttheta}$ for state-partitioned tabular policies.
The gradient $\nabla_{\ttheta} J_{\ttheta}$ consists of
$T\times |\ZZZ|$ subvectors of the form $\nabla_{\theta^{(t)}_z} J_{\ttheta}$
corresponding to time $t$ and expert $z$.
By the policy gradient theorem (Theorem~\ref{thm:pgt}), in view of (\ref{eq:partitionedtabular}), the elements of this subvector can be written as
\begin{equation}
  \label{eq:omegast}
  \frac{\partial J_{\ttheta} }{\partial \theta_z^{(t)}(a)}
  = \mu^{(t)}_{\ttheta}(\SSS_z) \times   \overline Q^{(t)}_{z, \ttheta}(a).
\end{equation}

For the proposition below, keep in mind that the domain of $J_{\ttheta}$  is a product of probability simplices, which is a closed, convex set that includes its boundary points (i.e., points with zero entries are permitted).

\begin{proposition}
  \label{prop:criticalpoints}
Consider a state-partitioned tabular policy.
If a policy parameter vector $\ttheta$ is a local maximum of $J_{\ttheta}$ then, for each time $t$ and expert $z$  with $\mu^{(t)}_{\ttheta}(\SSS_z)>0$, the support of $\theta^{(t)}_z$ is contained in $\argmax_a \overline Q^{(t)}_{z, \ttheta}(a) $.  
  \end{proposition}

In the Markov chain setting, a global maximum can be found using policy iteration,
since any pure policy
can be improved by selecting the best action given the incumbent state-action value function.
A pure policy is (globally) optimal if it cannot be improved in this manner.
Although the preceding proposition states that each (local) maximum
satisfies a similar ``fixed-point'' equation, unlike in the Markov chain setting,
there is no guarantee that repeatedly updating
pure policies by maximizing the incumbent $\overline Q^{(t)}_{z, \ttheta}$
functions results in a converging sequence
of policies, let alone that this sequence converges to a local maximum.
This is because in the nonlinear setting the $\overline Q^{(t)}_{z, \ttheta}$ are \emph{not}, in general, state-action value functions.

\subsection{Exponentiated Q-Ascent Algorithm}
Our algorithm starts from an arbitrary $\ttheta$ with all $\theta^{(t)}_z$ in the (relative) interior of the probability simplex over $\AAA$,
and all iterates continue to lie in this (relative) interior.
The updating equation (\ref{eq:expupdatepi}) below, together with the connection to the natural gradient ascent algorithm we develop in the next subsection, motivates the name \emph{exponentiated Q-ascent algorithm}.

\paragraph{Algorithm steps.} 
Given a learning rate $\eta>0$, we execute these steps:
\begin{enumerate}
\item Calculate the $\mu^{(t)}_{\ttheta}$ with forward iterative scheme (\ref{eq:expertforward}).
\item Calculate the $\sigma^{(t)}_{\ttheta}$ with backward iterative scheme (\ref{eq:expertbackward}).
  \item Calculate the $\overline Q^{(t)}_{z, \ttheta}$ using the definitions in (\ref{eq:expertQ}) and (\ref{eq:barQ}).
\item Update $\ttheta$ via, for every $t$ and $z$,
  \begin{equation}
    \label{eq:expupdatepi}
    \theta^{(t)}_z(a)  \leftarrow \frac{\theta^{(t)}_z(a)\times \exp\left(\eta\times   \overline Q^{(t)}_{z, \ttheta}(a) \right)}
    {\sum_{\tilde a} \theta^{(t)}_z(\tilde a)\times \exp\left(\eta\times   \overline Q^{(t)}_{z, \ttheta}(\tilde a) \right)}.
\end{equation}
(If $\mu^{(t)}_{\ttheta}(\SSS_z) = 0$, then $\overline Q^{(t)}_{z, \ttheta}$ is zero and it is understood that no update takes place.)
The update uses the current values of $\theta^{(t)}_z$ on the right-hand side to compute all new values of $\theta^{(t)}_z$
simultaneously before overwriting.
\item
Go to Step 1 until 
\begin{equation*}
\sum_{t=0}^{T-1}\sum_{z\in\ZZZ} \mu^{(t)}_{\ttheta}(\SSS_z) \times \Var_{A\sim \theta^{(t)}_z} \left(\overline Q^{(t)}_{z, \ttheta}(A)\right)<\epsilon.
\end{equation*}
Here
\begin{equation*}
\Var_{A\sim \theta^{(t)}_z} (\overline Q^{(t)}_{z,\ttheta}(A)) = \sum_a \theta_z^{(t)}(a) \times (\overline Q^{(t)}_{z,\ttheta}(a))^2 - 
\left(E_{A\sim \theta^{(t)}_z} (\overline Q^{(t)}_{z,\ttheta}(A))\right)^2,
\end{equation*}
where
\begin{equation*}
  E_{A\sim \theta^{(t)}_z} (\overline Q^{(t)}_{z,\ttheta}(A)) = \sum_a \theta_z^{(t)}(a) \times \overline Q^{(t)}_{z,\ttheta}(a).
\end{equation*}
\end{enumerate}

\paragraph{Stopping criterion.}
Let $\hat \theta^{(t)}_z$ denote the updated $\theta^{(t)}_z$.
Recall that the Kullback-Leibler (KL) divergence between pmfs $p_1$ and $p_2$ on the same finite set is defined via
\begin{equation*}
  \KL{p_1}{p_2} = \sum_y p_1(y) \log\left(\frac{p_1(y)}{p_2(y)}\right).
\end{equation*}
The right-hand side should be interpreted as a sum over the support of $p_1$ only, and
it should be interpreted as $\infty$ when the denominator is zero for at least one of the summands.
From the updating rule (\ref{eq:expupdatepi}) we have that
\begin{align*}
  \KL{\theta^{(t)}_z}{\hat \theta^{(t)}_z} &=
                             \sum_a \theta^{(t)}_z(a) \times\log \left(\frac{\theta^{(t)}_z(a)}{\hat\theta^{(t)}_z(a)}\right)\\
  &= \sum_a\theta^{(t)}_z(a) \times \left(\log\left(\sum_{\tilde a} \theta^{(t)}_z(\tilde a)\times\exp\left(\eta\times \overline Q^{(t)}_{z, \ttheta}(\tilde a)\right)\right)- \eta \times \overline Q^{(t)}_{z, \ttheta}(a)\right)\\
  &=\log\left(\sum_{\tilde a} \theta^{(t)}_z(\tilde a)\times\exp\left(\eta\times \overline Q^{(t)}_{z, \ttheta}(\tilde a)\right)\right)- \eta\times E_{A\sim \theta^{(t)}_z} (\overline Q^{(t)}_{z, \ttheta}(A)).
  \end{align*} 
The first term can be seen as a log-moment generating function of $\overline Q^{(t)}_{z, \ttheta}$ under $\theta^{(t)}_z$,
and its Taylor expansion around $\eta=0$ is
  \begin{equation*}
\log\left(\sum_{\tilde a} \theta^{(t)}_z(\tilde a)\times\exp\left(\eta\times \overline Q^{(t)}_{z, \ttheta}(\tilde a)\right)\right)\approx
    \eta\times E_{A\sim \theta^{(t)}_z} (\overline Q^{(t)}_{z, \ttheta}(A)) + \frac {\eta^2}2 \times \Var_{A\sim \theta^{(t)}_z} (\overline Q^{(t)}_{z, \ttheta}(A)).
  \end{equation*}
So the stopping criterion is approximately
\begin{equation}
  \label{eq:stoppingapprox}
\sum_{t=0}^{T-1} \sum_{z\in\ZZZ} \mu^{(t)}_{\ttheta}(\SSS_z)\times \KL{\theta^{(t)}_z}{\hat \theta^{(t)}_z} <\frac{\eta^2}2 \epsilon.
\end{equation}
It follows from (\ref{qdef0}) and the chain rule of Kullback-Leibler divergence (see \cite[Thm.\ 2.5.3]{cover1999elements}) that
  \begin{multline*}
    \KL{q_{\ttheta}}{q_{\hat\ttheta}} =\sum_{t=0}^{T-1} \sum_{z\in\ZZZ} \mu^{(t)}_{\ttheta}(\SSS_z)\times \KL{\theta^{(t)}_z} {\hat \theta^{(t)}_z}
\\    + \sum_{t=0}^{T-1} \sum_{s,a} \mu^{(t)}_{\ttheta}(s)\times \theta^{(t)}_{z(s)}(a) \times \KL{p^{(t)}_{\mu^{(t)}_{\ttheta}}(\cdot|S^{(t)}=s, A^{(t)}=a)}{p^{(t)}_{\mu^{(t)}_{\hat\ttheta}}(\cdot|S^{(t)}=s, A^{(t)}=a)}.
\end{multline*}
The first term represents the contribution due to the immediate effect of the policy change, while the second
term represents the contribution due to the effect of the policy change on the transition probabilities.
Since the left-hand side of (\ref{eq:stoppingapprox}) coincides with the first term,
the stopping criterion can be interpreted as a proxy for requiring that  $\KL{q_{ \ttheta}}{q_{\hat\ttheta}}$ be small
whenever the contribution of the second term is of second-order.

\paragraph{Adaptive step size variant.}
Our algorithm naturally supports an adaptive updating scheme inspired by the line search mechanism of Trust Region Policy Optimization (TRPO) \cite{Schulman2015TRPO}, which dynamically manages the policy update step size. At each iteration, a candidate policy is generated via (\ref{eq:expupdatepi}) and the objective value of this candidate policy is evaluated. If this value is an improvement over the objective value of the current policy, the update is accepted and the original learning rate $\eta$ is used for the next iteration. Otherwise, the learning rate $\eta$ is reduced (e.g., halved), the candidate is recomputed, and its corresponding objective value is evaluated. This process continues until there is an improvement in objective value. The resulting adaptive scheme prevents overshooting and ensures the update remains within a `trusted' region of reliable improvement, thus mirroring the approach used in TRPO.

\subsection{Approximate Natural Gradient}

In this subsection we show that the exponentiated Q-ascent
algorithm can be viewed as an approximate version of natural gradient
ascent after appropriately transforming the parameter space.

The updating scheme maintains pmfs $\theta^{(t)}_z$ in the (relative) interior of the probability simplex over $\AAA$.
Define the centered log-probabilities
\begin{equation*}
\gamma^{(t)}_{z,a}=\log(\theta^{(t)}_z(a))-\frac1{|\AAA|}\sum_{\tilde a\in \AAA}\log(\theta^{(t)}_z(\tilde a)),
\end{equation*}
so that $\sum_a \gamma^{(t)}_{z,a}=0$.
There is a one-to-one correspondence between the pmf $\theta^{(t)}_z$  and the vector $\gamma^{(t)}_z = (\gamma^{(t)}_{z,1},\ldots,\gamma^{(t)}_{z,|\AAA|})$ with $\sum_a \gamma^{(t)}_{z,a}=0$, with the inverse given by the \emph{softmax}:
\begin{equation}
  \label{eq:softmax}
  \theta^{(t)}_z (a) = \frac{\exp(\gamma^{(t)}_{z,a})}{\sum_{\tilde a} \exp(\gamma^{(t)}_{z,\tilde a})}.
  \end{equation}
Accordingly, any state-partitioned tabular policy in the relative interior of $\TTheta$ can be reparameterized by $\ggamma$, the collection of $\gamma^{(t)}_z$ for each time $t$ and expert $z$.
As long as we work in this interior, 
the policy can be defined directly using the $\gamma^{(t)}_z$  parameters via
\begin{equation*}
\pi^{(t)}_{\gamma^{(t)}}(a|s) = \frac{\exp(\gamma^{(t)}_{z(s),a})}{\sum_{\tilde a \in \AAA} \exp(\gamma^{(t)}_{z(s),\tilde a})}.
\end{equation*}
This reparameterization allows us to work entirely with vectors $\ggamma$ for which $\sum_a\gamma^{(t)}_{z,a}=0$ for each time $t$ and expert $z$,
bypassing the collection of pmfs $\ttheta$.
To emphasize this reparametrization, we denote the parameter space by $\GGamma$ 
rather than $\TTheta$.
The objective function becomes $J_{\ggamma}$ and the function $\overline Q^{(t)}_{z,\ttheta}$ becomes $\overline Q^{(t)}_{z,\ggamma}$.
We let $\ttheta(\ggamma)$ denote the softmax transformation back to the pmf representation.

In lieu of working with the natural gradient, our algorithm uses an \emph{approximate natural gradient} $\widetilde \nabla_{\ggamma} J_{\ggamma}$ defined as 
\begin{equation*}
\widetilde \nabla_{\ggamma} J_{\ggamma} = \tilde F(\gamma)^\dagger \times \nabla_{\ggamma} J_{\ggamma}, 
\end{equation*}
which is identical in form to the natural gradient, except that it uses the block diagonal matrix
$\tilde F(\ggamma)$ with diagonal blocks 
$\tilde F^{(t,t)}(\ggamma) = K^{(t)}(\ggamma)$ instead of the matrix $F(\ggamma)$. That is,
$\tilde F(\ggamma)$ is obtained by setting $G^{(t)}(\ggamma)=0$ for all $t$ in Proposition~\ref{prop:fisherrecursion},
which amounts to linearizing the nonlinear Markov chain for purposes of the information geometry.
We emphasize that, since the approximate natural gradient is the
product of $\tilde F(\ggamma)^\dagger$ and $\nabla_{\ggamma} J_{\ggamma}$, these nonlinearities \emph{are} incorporated in the approximate natural gradient as they remain present in the standard gradient term $\nabla_{\ggamma} J_{\ggamma}$.
The matrix $\tilde F(\ggamma)$ exhibits a two-level block-diagonal structure.  
At the first level, it is block-diagonal across time indices, with each diagonal block corresponding to a fixed time step.  
At the second level, each time-specific block is itself block-diagonal over experts.  
The diagonal block associated with the pair $(t, z)$ is given by  
  \begin{equation*}
    \tilde F^{(t,z), (t,z)}(\ggamma) =
    \mu^{(t)}_{\ggamma}(\SSS_z) \times \hat F(\gamma_z^{(t)}),
    \end{equation*}
where, for $\psi\in\R^{|\AAA|}$,  
\begin{equation*}
  \hat F(\psi) = \sum_a \hat\pi_\psi(a)\times (\nabla_\psi \log \hat \pi_\psi(a))\times (\nabla_\psi \log \hat \pi_\psi(a))^\top
\end{equation*}
is the Fisher information matrix corresponding to the softmax distribution $\hat \pi_\psi$ defined via
  \begin{equation}
    \label{eq:hatpi}
    \hat\pi_\psi (a) = \frac{\exp(\psi_a)}{\sum_{\tilde a} \exp(\psi_{\tilde a})}.
    \end{equation}

The next proposition shows that the elements of $\widetilde \nabla_{\ggamma} J_{\ggamma}$ are appropriately centered values of the $\overline Q^{(t)}_{z,\ggamma}$.

\begin{proposition}
  \label{prop:naturalgradientmoesoft}
  Consider a state-partitioned tabular policy parameterized by the softmax parameter vector $\ggamma$.
  The subvector $\widetilde\nabla_{\gamma_{z,a}^{(t)}}  J_{\ggamma}$ of the approximate natural gradient can be written as 
\begin{equation} \label{eq:naturalgradientmoesoft}
\widetilde\nabla_{\gamma_{z,a}^{(t)}}  J_{\ggamma} =   
\overline Q^{(t)}_{z,\ggamma}(a) - \frac1{|\AAA|} \sum_{\tilde a} \overline Q^{(t)}_{z,\ggamma}(\tilde a).
\end{equation}
\end{proposition}

This proposition shows that 
our exponentiated Q-ascent algorithm corresponds to 
the approximate natural gradient ascent scheme
\begin{equation}
  \label{eq:approxnaturalgradascient}
    \gamma^{(t)}_z \leftarrow    \gamma^{(t)}_z  + \eta \times \widetilde\nabla_{\gamma^{(t)}_z}  J_{\ggamma}.
    \end{equation}
Indeed, upon substitution of this into (\ref{eq:softmax})
using Proposition~\ref{prop:naturalgradientmoesoft},
we obtain the updating scheme (\ref{eq:expupdatepi}).

\subsection{Performance Guarantee}
The justification for using our exponentiated Q-ascent algorithm
is provided by Theorem~\ref{thm:approximatenaturalgradient} below.  
It shows that it
is approximately an appropriate \emph{natural} gradient ascent algorithm
near any pure policy, even though this natural gradient ascent 
scheme for the softmax parameters
  \begin{equation*} 
    \gamma^{(t)}_z \leftarrow    \gamma^{(t)}_z  + \eta \times \widehat\nabla_{\gamma^{(t)}_z}  J_{\ggamma}
    \end{equation*}
is computationally intractable as discussed earlier.
This result justifies the use of the exponentiated Q-ascent algorithm, which replaces the natural gradient $\widehat\nabla_{\gamma^{(t)}_z}  J_{\ggamma}$ above 
with the approximate natural gradient $\widetilde\nabla_{\gamma^{(t)}_z}  J_{\ggamma}$ as in (\ref{eq:approxnaturalgradascient}),
for training nonlinear Markov decision problems.

The rank condition in the theorem below is an identifiability requirement.
Denote the set of `unreachable' time-expert pairs $(t, z)$ for which $\mu^{(t)}_{\ggamma}(\SSS_z)=0$ by $\UUU(\ggamma)$.
Since we work in ambient space when representing gradients, 
the number of columns of the matrix $F(\ggamma)$ is
$T \times |\ZZZ| \times |\AAA|$.
Adding a constant to each $\psi_a$ in (\ref{eq:hatpi})
  leaves the value of the right-hand side unchanged, so for each pair $(t,z)$ 
  the gradient of the log-trajectory probabilities vanishes along the direction
  where the same constant is added to the elements of $\gamma^{(t)}_z$.
The linear span of these directions, a subspace of dimension $T\times |\ZZZ|$, is contained in the null space of the Fisher information matrix $F(\ggamma)$. Furthermore,
for each $(t,z)$ such that $\mu^{(t)}_{\ggamma}(\SSS_z)=0$,
$\gamma^{(t)}_z$ has no effect on the trajectory probabilities, forming a subspace of dimension $|\AAA| - 1$ within the null space for each such pair.
We are left with an \emph{upper bound} on the rank of $F(\ggamma)$, and
are requiring that the rank of $F(\ggamma_n)$ equals this upper bound, i.e., there are no further reductions
due to, for example, redundant parameters, equivalent actions, and structural symmetries.

\begin{theorem}
  \label{thm:approximatenaturalgradient}
Let $\ttheta^*$ be a pure policy, and let $\{\ggamma_n\}$ be a sequence in $\GGamma$ with $\ttheta(\ggamma_n)\to \ttheta^*$ as $n\to \infty$.
  Define $\mu^{(t)}_{*} = \lim_{n\to\infty} \mu^{(t)}_{\ggamma_n}$
  and $R^{(t)}_{\mu} = \{(s, a,s'): \mu(s)>0, p^{(t)}_\mu(s'|s,a)>0\}$.
  Suppose that 
  \begin{enumerate}
  \item the set $\UUU(\ggamma_n)$ is eventually independent of $n$ and the null space of $F(\ggamma_n)$ 
    is a fixed subspace of dimension $T\times |\ZZZ|+(|\AAA|-1)\times |\UUU(\ggamma_n)|$;
    \item for all $t$, $\{s: \mu^{(t)}_{\ggamma_n}(s)>0\}$ eventually equals $\{s:\mu^{(t)}_{*}(s)>0\}$; and
  \item for all $t$, $R^{(t)}_{\mu^{(t)}_{\ggamma_n}}$ eventually equals $R^{(t)}_{\mu^{(t)}_{*}}$.    \end{enumerate}
Then
  \begin{equation*}
    \lim_{n\to\infty}\left\|\widehat \nabla_{\ggamma}   J_{\ggamma_n} -  \widetilde\nabla_{\ggamma}  J_{\ggamma_n}\right\|=0.
    \end{equation*}
  \end{theorem}

At first blush, this theorem may appear unsurprising.
In the Markov chain case we always have $\widehat \nabla_{\ggamma}   J_{\ggamma}=\widetilde\nabla_{\ggamma}  J_{\ggamma}$. 
In the nonlinear case, however, changes in the policy parameter $\gamma^{(t)}$ propagate into the transition probabilities \emph{beyond} time $t$.
The theorem states that, in softmax coordinates, these downstream effects do not substantially alter the natural-gradient direction near pure policies.

\section{A Dynamic Pricing Example}
\label{sec:dynamicpricingexample}

\subsection{Problem Description}
We consider a single station queueing system with the following characteristics. 
It consists of $n \geq 1$ homogeneous servers that provide some service
and a staging area (the `buffer') with room for $b\ge 0$ customers.
Customers seeking service may either be served immediately or, if there is space,
must wait in the buffer until a server becomes available.
Customers for which there no available space in the buffer are lost. 
Customers occupy a server while they are in service, and
this server can serve no other customer during that time.
They are served according to a nonpreemptive, nonidling service discipline.
All customers exit the system upon completion of service,
at which time they release their server.

We consider a discrete-time model with time indexed by the nonnegative integers,
over a fixed horizon $T>0$.
All service durations are independent and identically distributed,
and are independent of everything else.
Service durations are integer-valued, assumed to be at least one time period, and have a maximum value. 
We let $g$ denote the service duration pmf.  
We use the convention that customers
departing or arriving between times $t$ and $t+1$ do so
\emph{immediately before} time $t+1$,
with arrivals immediately being able to use servers that are freed up (if any).

Fix some time $t$.  The distribution of the number of customers arriving between times $t$ and $t+1$ 
depends on the \emph{price} of the service, which is the action $a$ for this system.  
For each price choice $a \in \AAA$, the number of arrivals follows a Poisson distribution with arrival rate $\lambda_a^{(t)}$.
The price choice $a$ can depend on the number of customers in the system at time $t$.  
Each customer who is able to enter the system is charged the quoted price $a$. 
There is a fixed, per-period cost $c_W$ charged to each customer waiting in the buffer.
A terminal penalty $c_T$ is incurred for each customer still in the system at time $T$.
There is also a service-level requirement: the probability that the number of customers in the system at time $t$ exceeds a pre-determined
threshold $\hat z$ should not exceed a pre-determined probability $\alpha$.

\subsection{Nonlinear Markov Decision Problem Representation}
\label{sec:nonlinearMC_DPexample}

\paragraph{State space.}
The QPLEX paradigm avoids the curse of dimensionality of a Markov chain representation by re-imagining the information requirements 
so as to work with a state space of much \emph{lower} dimension.
The state space $\SSS = \ZZZ\times\LLL$ consists of pairs $(z, \ell)$, 
where $z$ represents the number of customers in the system, and $\ell$  
represents the \emph{remaining service duration of a uniformly chosen at random customer in service}.  This state space is very small.  
Each $z$ is called a \emph{counter} and each $\ell$ is called a \emph{label}.
The symbols $z, z'$ and $\ell, \ell'$ shall be reserved for generic components of $\ZZZ$ and $\LLL$, respectively.
We shall simply write $z$ and $\ell$ for the components of state $s$.
Similarly, given a state $s'$, we shall write $z'$ and $\ell'$ for the components of $s'$.

Given a pmf $\mu$ on the state space $\SSS$, we write $\mu(z) = \sum_\ell \mu(z, \ell)$ for the value of the marginal of $\mu$ with respect to its first component.
If $\mu(z)>0$, we furthermore write $\mu_{|z}(\ell) = \mu(z,\ell)/\mu(z)$ for 
the conditional pmf of the label $\ell$ given the counter $z$ under $\mu$.

\paragraph{Markovian policy functions.}  
We consider a state-partitioned tabular policy
with $|\ZZZ|$ experts, each assigned to a specific counter $z$.  
Here $\SSS_z= \{z\}\times \LLL$ and note that $\mu(\SSS_z) = \mu(z)$.  
Thus, the parameter vector $\theta^{(t)}$ is a vector of $|\ZZZ|$ pmfs.

\paragraph{Transition functions.} In the QPLEX calculus the functions $p_\mu^{(t)}(s' | s, a)$ have the following specific form.  
There exist functions $\hat p_{\xi}^{(t)}(z',\ell' | z, a)$ parameterized by a pmf $\xi$ on $\LLL$ such that, whenever $\mu(z)>0$,
\begin{equation}
  \label{eq:kernelqplex}
  p^{(t)}_\mu(s'|s,a) = \hat p^{(t)}_{\mu_{|z}}(z',\ell' | z, a)
\end{equation}
for any pmf $\mu$ on $\SSS$, $s=(z,\ell)$, $a$, and $s'=(z',\ell')$.
We now describe these functions.  In what follows, $\xi$ denotes a generic pmf on $\LLL$.

Fix some time $t$, some pmf $\mu$ on $\SSS$, some state $(z, \ell)$ for which $\mu(z)>0$, and some action $a$. 
Under the QPLEX calculus, the transition function $\hat p_{\mu_{|z}}^{(t)}(z', \ell' | z, a)$ takes on the following specific `chain-rule' form: 
\begin{equation}
  \label{eq:qplexchainruleform}
\hat p^{(t)}_{\mu_{|z}}(z', \ell' | z, a) = \sum_{d} q_{\mu_{|z}}(d | z) \times \rho^{(t)}(z'|z, d, a) \times \sum_{k'} q(k' | z, d, z') \times q_{\mu_{|z}}(\ell' | z, k').
\end{equation}
We describe each term below.  

The first term represents the probability that $d$ customers complete service by time $t+1$.
The QPLEX paradigm
assumes the remaining service durations of the $\min(z, n)$ customers in service 
are i.i.d.\ with pmf given by $\mu_{|z}$.  Under this assumption each customer in service has probability 
$\mu_{|z}(1)$ of completing their service by time $t+1$, and so this probability is given by $q_{\mu_{|z}}(d | z)$, where
\begin{equation} \label{eq:qxidgivenz}
q_{\xi}(d | z) = \Pr\left[ \bin\left (\min(z,n), \xi(1)\right) = d \right].
\end{equation}

The second term represents the probability that there are $z'$ customers in the system at time $t+1$ given that the number of 
customers in the system at time $t$ is $z$, there are $d$ customers completing service during times $t$ and $t+1$, and the
price choice is $a$.  The number of customers arriving to the system follows a Poisson distribution with arrival rate $\lambda_a^{(t)}$.
There may not be room to admit all arriving customers to the system due to a full buffer. If there are currently $z$ customers in the system, $d$ of which leave during this period, 
then $n+b+d-z$ customers can be admitted. Therefore, 
\begin{equation*}
\rho^{(t)} (z' | z,d,a) 
= \begin{cases}
\textrm{Pr}(\textrm{Poisson}(\lambda_a^{(t)})=z'+d-z) &\text{if } z' < n + b ,\\
\textrm{Pr}(\textrm{Poisson}(\lambda_a^{(t)})\ge n+b+d-z) &\text{if } z' = n + b .
\end{cases}
\end{equation*}

It remains to represent the probability that a uniformly chosen at random customer in service at time $t+1$ has remaining service duration $\ell'$
given that the number of customers in the system at times $t$ and $t+1$ are $z$ and $z'$, respectively, and the number of customers completing service is $d$.  
This probability is obtained in two steps (the third and fourth terms)
by conditioning on the customer's type at time $t+1$:  Either a customer is continuing their service, in which this customer is assigned type $k' = \old$, or
this customer is beginning their service, in which this customer is assigned type $k' = \new$.  
(A `new' customer could be a customer who had been waiting in the buffer.) 
The third term represents the probability that a uniformly chosen at random customer at time $t+1$ has type $k'$ given the values of $z$, $d$, and $z'$.   
Of the $\min(z', n)$ customers in service at time $t+1$, there are $\min(z, n) - d$ customers of type `old' and 
$\min(z',n) - (\min(z,n) - d)$ customers of type `new'. Thus, using the convention that the type is `new' if $z'=0$, the probability of selecting a customer of type $k'$ is   
\begin{equation}
  \label{eq:qk'}
q(k'|z,d,z') = 
\begin{cases}
1-\frac{\min(z, n)-d}{\min(z',n)} &\text{if } k'=\new \text{ and } z' > 0 \\
\frac{\min(z, n)-d}{\min(z',n)} &\text{if } k'=\old \text{ and } z' > 0 \\
1 &\text{if } k'= \new \text{ and } z' = 0 \\
0 &\text{otherwise.}
\end{cases}
\end{equation}
Note that if $\mu_{|z}(1) = 1$, then all customers in service complete their service, and so$q(\new | z, \min(z,n), z') = 1$ and $q(\old | z, \min(z,n), z') = 0$.

If the customer type is `new', then the 
probability this customer will have a remaining service duration of $\ell'$ is simply the service duration probability $g(\ell')$.  
If the customer type is `old', which can only happen if $\mu_{|z}(1) < 1$, its remaining service duration is one less than it was at time $t$, and it must be greater than one
(otherwise this customer would have completed their service).
This leads to the term $q_{\mu_{|z}}(\ell'|z,k')$, where
\begin{equation} \label{eq:qxiell'givenzk'}
q_{\xi}(\ell'|z,k') = 
\begin{cases}
\frac{\xi(\ell'+ 1)}{1 - \xi(1)} & \text{if } k'=\old\\
g(\ell') & \text{otherwise.}
\end{cases}
\end{equation}
For example, in vector form, if $\mu_{|z}(\ell) = (0.1, 0.2, 0.3, 0.4)$, then $q_{\mu_{|z}}(\ell' | z, \old) = (\frac{2}{9}, \frac{1}{3}, \frac{4}{9}, 0)$.  

\paragraph{Reward functions.}  
Fix the number of customers in the system $z$ at time $t$ and price $a$.  Suppose there are $d$ customers who complete service by time $t+1$,
$y$ customers who arrive to the system between times $t$ and $t+1$, and the chosen price is $a$.  Since the available capacity to house customers is $n + b + d - z$,
the revenue received by time $t+1$ is $\min(y, n+b+d -z) \times a$, and so the revenue portion of the reward function $r_\mu^{(t)}(s, a)$ is
\begin{equation*}
\sum_{d,y} q_{\mu_{|z}}(d|z) \times\textrm{Pr}(\textrm{Poisson}(\lambda_a^{(t)})=y)
  \times \min(y, n + b + d - z) \times a.  
\end{equation*}
The cost portion of the reward function consists of the fixed, per-period waiting cost $c_W \times \max(0, z-n)$ and
a service-level penalty cost term of the form
$C \times \left( \max\{ \sum_{\hat z > n} \mu(\hat z) - \alpha, \; 0 \} \right)^k$ for constants $C$ and $k$
used to incorporate the service-level requirement.
The terminal reward function $r^{(T)}_\mu(s)$ is $-c_T\times z$, representing the terminal penalty.

\subsection{Efficient Calculation}
This nonlinear Markov chain exhibits special structure that can be exploited for efficient computation of the policy gradient. 
Here, we present only a high-level overview of two key ingredients.  See Appendix~\ref{app:effcalc} for a comprehensive and more general presentation that applies to a variety of systems that can be modeled with QPLEX, not only the dynamic pricing example.

The first ingredient facilitating efficient calculation 
stems from the composite structure of a state $s=(z,\ell)$,
the form of the kernel (\ref{eq:kernelqplex}), and
the fact that the reward function can be written in the following form with $s=(z,\ell)$:
\begin{equation}
  \label{eq:rewardsep}
  \begin{aligned}
    r^{(t)}_\mu(s,a) &= c^{(t)}_\mu + \hat r^{(t)}_{\mu_{|z}}(z,a)\\
    r^{(T)}_\mu(s) &= c^{(T)}_\mu + \hat r^{(T)}_{\mu_{|z}}(z).
    \end{aligned}
\end{equation}
This implies that, for $s=(z,\ell)$,
\begin{equation}
  \label{eq:Qtsummary}
  Q^{(t)}_{\mu,\sigma'} (s,a) =c^{(t)}_\mu + \hat Q^{(t)}_{\mu_{|z}, \sigma'}(z,a)
  \end{equation}
  with appropriately defined $\hat Q^{(t)}_{\xi, \sigma'}(z,a)$.
  In these equations, $\mu_{|z}$ should be read as some fixed pmf $\xi_0$ on $\LLL$
  where $\mu(z)=0$.
Note that the right-hand side is independent of $\ell$, that
the first term is constant in $(z,\ell,a)$, and that the second term
only depends on $\mu$ through $\mu_{|z}$.
Due to this structure, several sums in the forward and backward iterative schemes are no longer required,
and the computation of $\overline Q^{(t)}_{z,\ttheta}$
required by the algorithm is less expensive than a naive calculation.

The second ingredient facilitating efficient calculation is the fact that
the policy gradient, like the nonlinear transition probabilities in (\ref{eq:qplexchainruleform}), can be written as a sum of product terms.  
The sum-product algorithm offers a powerful tool for such calculations,
as it systematically decomposes the computation into smaller, manageable operations. By eliminating redundant calculations and reducing computational complexity, 
it enables efficient handling of sums of products, which commonly arise in probabilistic inference and other computational tasks. 
Here is a very simple concrete example to illustrate the idea.  Consider the sum of products function of $x_2, x_4$ given by
$\sum_{x_1, x_3} f_1(x_1, x_2) \times f_2(x_2, x_3) \times f_3(x_3, x_4)$ and suppose   
each variable takes $m$ possible values.  A naive calculation takes order $m^4$ time.  By rewriting the expression 
as $(\sum_{x_1} f_1(x_1, x_2)) \times (\sum_{x_3} f_2(x_2, x_3) \times f_3(x_3, x_4) )$
and separately calculating each of the two functions in parentheses first, the sought-after sum can be calculated in order $m^2$.

\section{Numerical Experiments}
\label{sec:numericalillustration}

In the dynamic pricing example one seeks a \emph{pure} policy on prices, namely, a specification of the price to be quoted for each time $t$ and number of customers in the system $z$.
Such count-based policies have been widely studied in the dynamic pricing and revenue management literature
for multiserver systems with general service duration distributions, see, for example, \cite{balseiroma, adamjiaqi, ziya2006}.

 Given an initial policy, which we always take to be uniform on prices for each $(t,z)$,
the exponentiated Q-ascent algorithm (applied to a specific problem instance) will produce a sequence of \emph{randomized} policies on prices, each of which is
a collection of distributions over prices, one for each $(t,z)$.
Throughout this section, when we speak of the QDP policy, it should be interpreted as the \emph{last} randomized policy in this sequence, converted to a pure policy
by selecting, for each $(t,z)$, the price corresponding to the mode of the distribution, with ties broken in favor of the lowest price.
Throughout these experiments, we set the stopping tolerance to $\epsilon=10^{-6}$,
unless otherwise stated for runs that stop after a fixed number of training episodes.
All training times we report are for an Apple M4 chip.

\subsection{Algorithm Performance}
\label{sec:exampleA}

To assess the performance of our algorithm, we study two cases for which the optimal solutions are available:
(i) small-scale instances where the full-information MDP can be solved to optimality via Bellman equations; and
(ii) instances with geometric service duration distributions, which can be solved to optimality via Bellman equations as the state no longer needs to incorporate residual service durations.

\subsubsection{Experimental Design Elements}

Each experiment starts with an empty system. The time horizon is $T = 50$.
The set of prices is fixed to $\{ 0.1, 0.2, \ldots, 1.0, 1.1 \}$.  (No customers arrive if the price is set to 1.1.)
We next describe design elements common to all our experiments. 

\paragraph{Service duration distributions.}
We consider four discrete \emph{uniform} service-time distributions with different supports:
\begin{itemize}
\item The \emph{uniform} distribution `Uni', uniform on $\{1,2,\ldots,20\}$,
\item The \emph{medium high uniform} distribution `UniM', uniform on $\{11,\ldots,20\}$,
  \item The \emph{high} uniform distribution `UniH', uniform on $\{16,\ldots,20\}$, and
  \item The \emph{barbell} distribution `BB', uniform 
    on $\{1,\ldots,5,16,\ldots,20\}$.
    \end{itemize}
    
\paragraph{Arrival rate functions.}
For each time $t$ and price $a$ chosen at time $t$ we let 
$\lambda^{(t)} (a)$ represent the corresponding Poisson rate of arriving customers.
Each experiment must specify the $\lambda^{(t)} (a)$ for each of the 50 periods and 11 price levels or, equivalently,
ten non-trivial {\em arrival rate functions} $\{ \lambda^{(t)} (a): 0 \le t < T \}$, one for each price level $a$.  

Each arrival rate function has the form 
\begin{equation}
\lambda^{(t)} (a) = \lambda_{\avg}(a) \times s^{(t)}.
\end{equation}
The {\em shape} function $\{ s^{(t)}: 0\le t < T\}$ has a time-average $\frac1T \sum_{t=0}^{T-1} s^{(t)}$ equal to 1, and so
$\lambda_{\avg}(a)$ is the average arrival rate.  
We use four possible choices for the shape function depicted in the diagram below:  
decreasing (`DEC'), increasing (`INC'), alternating (`ALT'), and constant (`CON').
(We scaled each shape function so that its maximum is 1 to more easily 
reveal its shape.) 

\begin{center}
  \includegraphics[width=.9\linewidth,trim=6 8 6 2,clip]{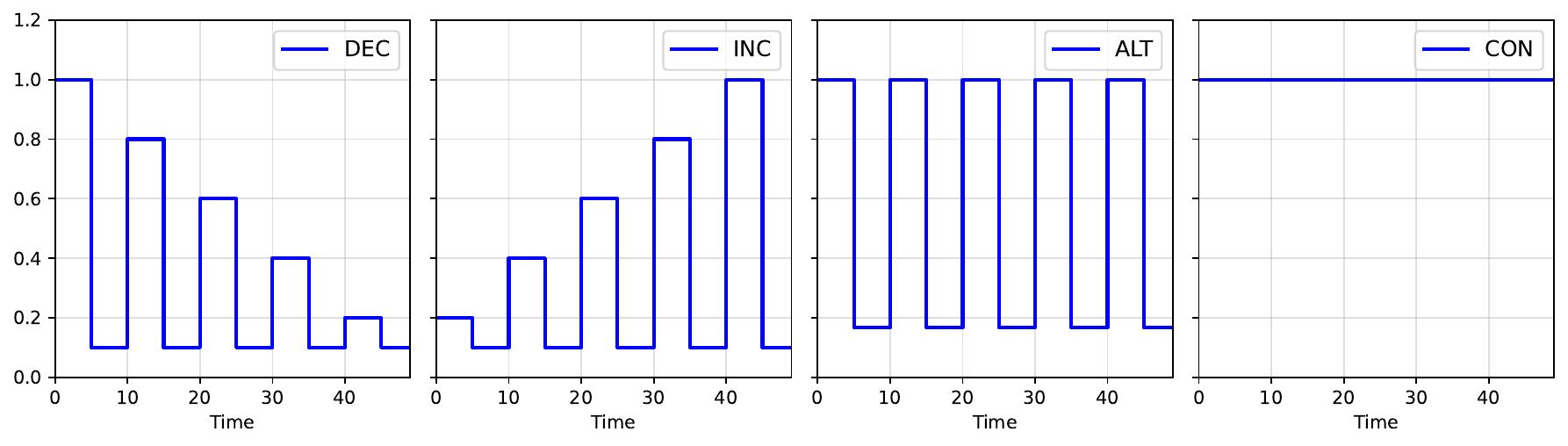}
  \end{center}


To pin down the average arrival rates, we use a linear demand model so that
\begin{equation*}
\lambda^{(t)}(a) =\lambda^{(t)}(0.1) \times (1.1-a).
\end{equation*}
In other words, a price of 1.1 corresponds to rejecting all customers
and $\lambda^{(t)}(0.1)$ is the maximum arrival rate associated with the minimum price. 
Given the expected service duration $E(S)$ and 
the number of servers $n$, we set the maximum arrival rate
in terms of the maximum average utilization $u_{\avg}(0.1)$ via
\begin{equation} \label{eq:uavg}
\lambda_{\avg}(0.1) = \frac{u_{\avg}(0.1) \times n}{E(S)},
\end{equation}
consistent with Little's Law.  We set $u_{\avg}(0.1) = 5$.  
Keep in mind that an average utilization of 5
would only be realized if the price is set to its minimum level throughout the horizon.  Note that the average utilization at price $a = 1.0$ is 0.5.

Putting it all together, the price-sensitive arrival rates are expressed in terms of
the parameters of the experimental design as follows:
\begin{equation}
  \label{eq:pricesensitivearrivals}
  \lambda^{(t)}(a) =  \frac{n \times u_{\avg}(0.1) }{E(S)} \times s^{(t)} \times (1.1-a).
\end{equation}

\subsubsection{Full-Information MDP}
\label{sec:fullinformationmdp}

A suitable state space for a Markov chain formulation includes the number of customers in the system $z$, but must also include relevant information about
the remaining service durations of the customers in service.  As it is immaterial which server serves which customer for this system,
it is sufficient (and efficient) to keep track of the numbers of customers $h(\ell)$ who have remaining service duration $\ell$ for each possible $\ell = 1, \ldots, \ell_{\max}$,
where $\ell_{\max}$ is the right end point of the service duration support,
so $\ell_{\max}=20$ in our design.
Note that $\sum_\ell h(\ell) = \min(z, n)$.  (See \cite[Ch.\ 2]{qplexbook} for details.) 
The number of states in such a formulation is $O\!\left(n^{\ell_{\max}}\right)$, the cardinality of which 
grows polynomially in $n$ with degree $\ell_{\max}$, quickly becoming prohibitively large and rendering the full-information MDP formulation computationally intractable.
For small instances, the optimal policy can be found via Bellman's optimality equations.
It is also possible to evaluate arbitrary policies (including count-based policies) via Bellman's equations.

We work with the following experimental design, which yields a total of 192 experiments.  

\begin{center}
\begin{tabular}{c|c}
Design Parameter & Value\\
  \hline
number of servers $n$ &3, 5\\
buffer size $b$ &3\\
shape function $s^{(t)}$ &DEC, INC, ALT, CON\\
maximum average utilization $u_{\avg}(0.1)$ & 5\\
service duration pmf $g$& Uni, UniM, UniH, BB\\
per-period holding cost $c_W$ &0.05, 0.1\\
terminal cost $c_T$& 0.5, 1.0, 1.5\\
\end{tabular}
\end{center}

\paragraph{Results.}

We compare the value estimates produced by several policies and evaluation methods. We denote each value by $v$, using a subscript to indicate the policy and a superscript to indicate the evaluation method.
Throughout these experiments, we set the learning rate to $\eta=1.0$.

\subparagraph{Policy evaluation.}
Given the QDP (pure) policy generated by our algorithm, QPLEX can be used to approximate its value.
(Given the policy parameter $\ttheta$, the forward iterative scheme yields the $\mu^{(t)}_{\ttheta}$, from which the sum of the per-period expected rewards are easily calculated.)
We shall call this the {\em QPLEX value} and denoted it by $v^{\qplex}_{\qdp}$.   
Not surprisingly, the exact value $v_{\qdp}^{\exact}$ of this QDP policy, which
can be obtained by using the Bellman equations, will not always equal the QPLEX value.
So how close are these two values across all experiments?
{\em The QPLEX and exact values of the QDP policies are practically identical.}
The histogram below shows the relative errors $(v_{\qdp}^{\exact}-v_{\qdp}^{\qplex})/v_{\qdp}^{\exact}$ across the experiments.
The maximum absolute relative error across all experiments is 0.65\%, with an average absolute relative error of just 0.12\%.
The key takeaway is that QPLEX value of a QDP policy is practically the same as the exact value.  

\begin{center}
\includegraphics[width=8cm]{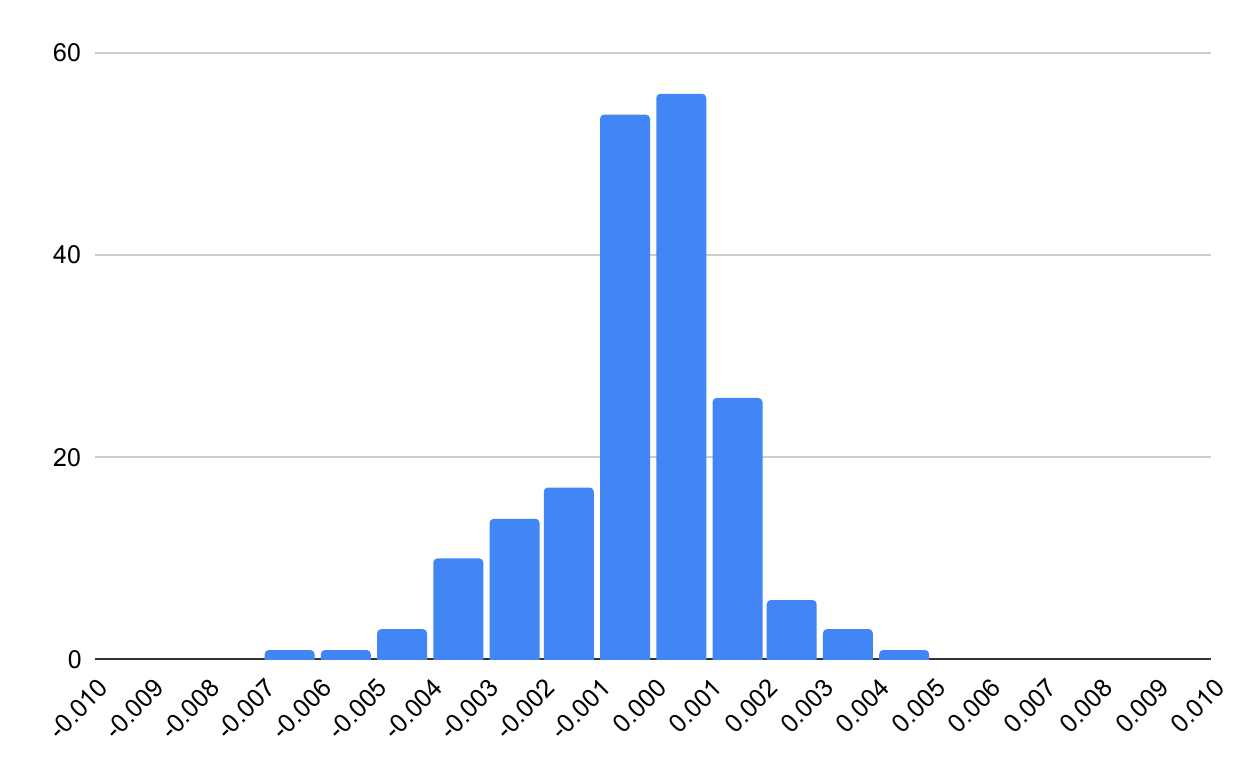}
\end{center}

\subparagraph{Relative optimality gap upper bound.}
We let $v_{\mdp}^{\exact}$ denote the value associated with the optimal full-information policy.
This is an upper bound for the relative optimality gap of the policy found by our algorithm, which aims to identify the optimal count-based policy.
The figure below shows the histogram of relative optimality gap upper
bounds 
$(v_{\mdp}^{\exact}-v_{\qdp}^{\exact})/v_{\mdp}^{\exact}$.
The {\em maximum} relative optimality gap upper bound is 3.6\% and for 89\% of the experiments it is less than 2.5\%.

\begin{center}
\includegraphics[width=8cm]{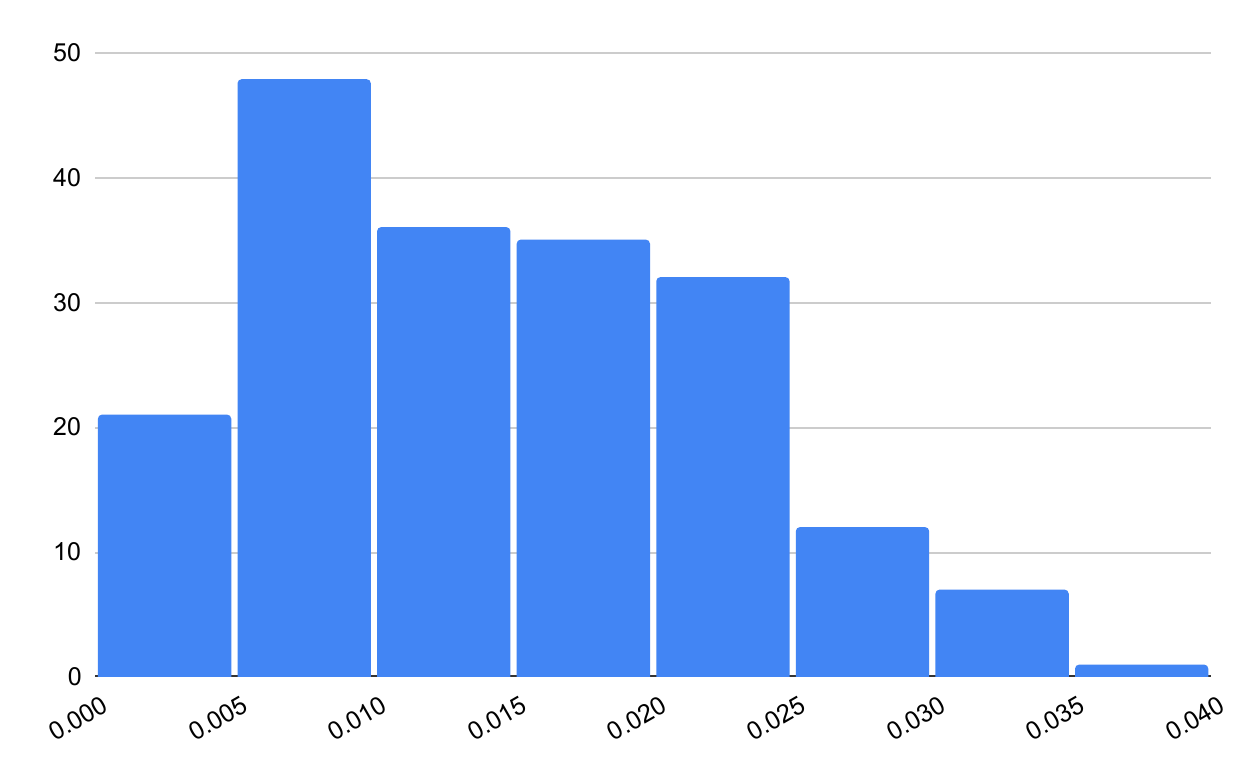}
\end{center}

\subparagraph{Proxy for relative optimality gap.}
The next question we address is how much of the relative optimality gap shown in the previous figure is due to the fact that 
the optimal full-information policy sets prices for every $(t, z, h)$
whereas the QDP policy sets prices only for every $(t, z)$.
In theory this question is answered by comparing the QDP policy with the optimal count-based policy.  
Even for the small problem instances considered here, it is computationally prohibitive to calculate the optimal count-based policy.

As a proxy for the optimal count-based policy, we extract a count-based policy from the optimal full-information policy
and compare its value to the value of the QDP policy.
Let $\mu^{(t)}_{\mdp}(z, h)$ denote the full-information state distribution at time $t$ associated with the optimal full-information policy.
Set $Q_{\mdp}^{(T)}(z, h, a) = 0$ for all $z$, $h$, and $a$, and set $V_{\mdp}^{(T)}(z) = 0$ for all $z$.  
We obtain the extracted count-based policy by working backwards in time using the following two equations: 
\begin{align}
Q^{(t)}_{\mdp}(z, h, a) &= r^{(t)}(z, h, a)  + \sum_{z'} p^{(t)}(z' | z, h) \times V^{(t+1)}_{\mdp}(z') \\
V^{(t)}_{\mdp}(z) &= \max_a \sum_h \mu^{(t)}_{\mdp}(h | z) \times Q^{(t)}_{\mdp}(z, h, a).
\end{align}
It  is understood that $V^{(t)}_{\mdp}(z) = 0$ and $Q^{(t)}_{\mdp}(z,h,a)=0$ if $\mu^{(t)}_{\mdp}(z) = 0$.  
The first equation is used to obtain the $Q$ values for $t = T-1$, and then the second equation determines both the $V$ values and optimal count-based policy at time $T-1$.
Now repeat the procedure for $t = T-2$, etc. 
It should be noted that the resulting policy is not, in general, optimal within the class of count-based policies. 

The figure below shows the histogram of $(v_{\extract}^{\exact}-v_{\qdp}^{\exact})/v_{\extract}^{\exact}$ values. 
The smallest value is $-0.05\%$, the largest value is 0.17\%, and the average is 0.017\%.  {\em The histogram provides strong evidence
that the value of the count-based policy found by QDP is likely to be extremely close to the value of the optimal count-based policy.}
The negative values shows that, despite requiring substantial computational effort to construct, 
the extracted policy is not necessarily superior to the QDP policy either.

\begin{center}
\includegraphics[width=8cm]{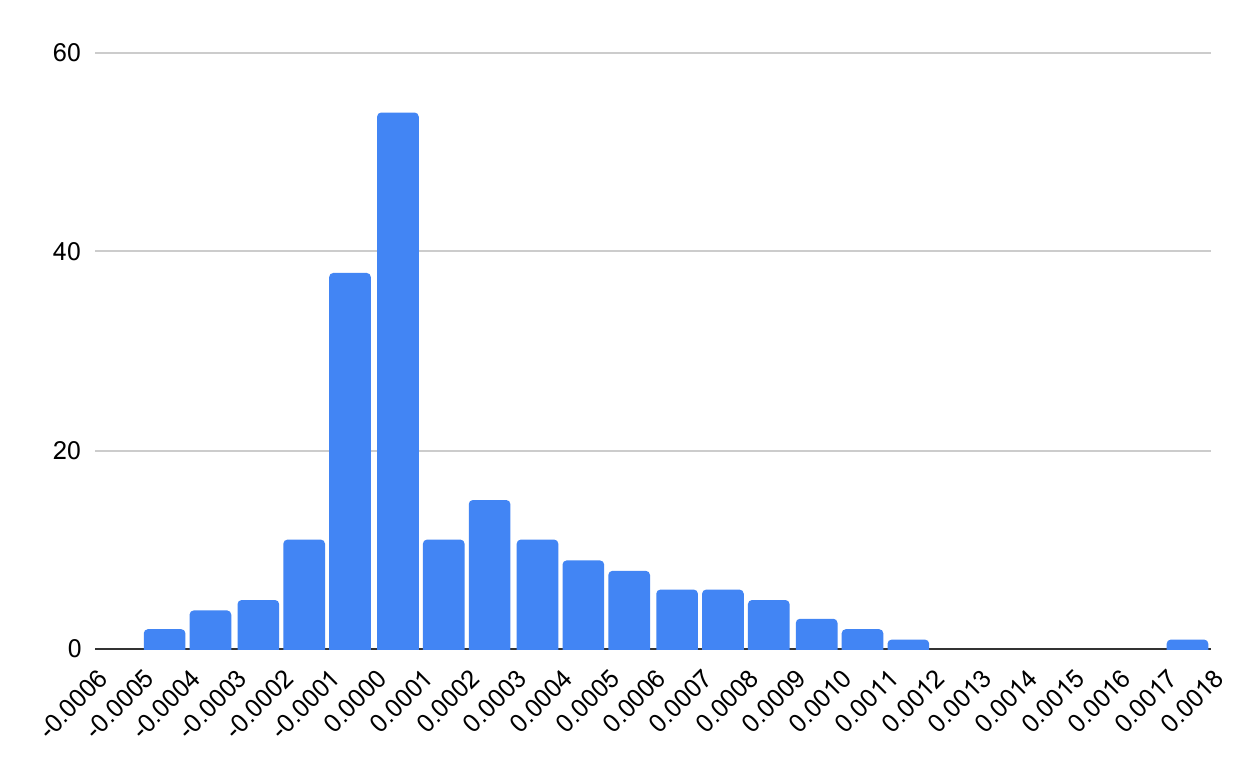}
\end{center}

\subsubsection{Geometric Service Durations}

With memoryless service durations, problem instances can be solved to optimality using Bellman equations where a state is simply the number of customers $z$.
With a much smaller state space, we can find optimal policies for systems with much larger values for the number of servers $n$.
The design parameters that are changed are shown in the table below.  Note that the means of the service times match the means in our full-information design.
We use a tolerance of $10^{-6}$ to cut off the right tail and ensure via a root finding procedure that we match the mean of the original distribution.)
Since there are now three instead of four possible service duration distributions, 
there are $(3/4)\times 192 = 144$ experiments in the design.

\begin{center}
\begin{tabular}{c|c}
Design Parameter & Value\\
  \hline
number of servers $n$ &20, 40\\
buffer size $b$ &5\\
service duration pmf $g$& truncated geometric with means 10.5, 15.5, 18.0
\end{tabular}
\end{center}

We make a comparison with QDP using the truncated geometric service
distributions, without explicitly leveraging the memoryless service durations.
(A dedicated policy gradient implementation with the reduced state space
would run faster but would result in the same accuracy, modulo any
  truncation effects.)
We again use a learning rate of $\eta=1.0$.
The \emph{exact} value of the QDP policy matches the value of the optimal policy
in 94 of the 144 problem instances.
The largest relative optimality gap is $2.86\times 10^{-7}$, so
the QDP policy is, to all intents and purposes, optimal.
Keep in mind we work with a truncated geometric distribution, so this discrepancy could simply be due to the truncation.

\subsubsection{Training}
We next discuss the training process. 
For this purpose, we select the problem instance that has the highest relative optimality gap upper bound
in the experimental design for the full-information MDP.  The parameters for that problem instance are: the number of servers $n$ is 3, 
the service duration distribution $g$ is UniH, the shape function $s^{(t)}$ is INC, the per-period holding cost $c_W$ is $0.05$, and the terminal cost $c_T$  is 1.5.

\paragraph{Accuracy of values during training.}
Recall that we showed that the QPLEX value of a QDP policy is nearly identical to its exact value. 
This QDP policy is the last randomized policy generated by the algorithm converted to a pure policy.  We start by assessing whether 
the QPLEX values of the randomized policies generated throughout training closely track their exact values as well.

The figure below shows both the QPLEX and the exact values
of the randomized policies encountered during training, recorded every 50 training episodes. We can see that indeed the values are virtually identical
throughout training.  We remark that since QDP is a deterministic solver, the resulting (interpolated) curve
does not exhibit the zigzag behavior typically observed in conventional reinforcement learning.

\begin{center}
  \includegraphics[height=5.2cm]{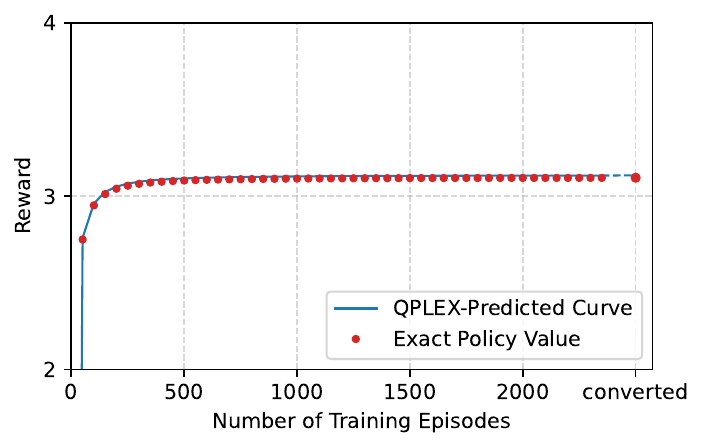}
  \end{center}

  \paragraph{Accuracy of ascent directions during training.}
  The QPLEX values of the randomized policies
generated throughout training use only the forward iterative scheme.
We next turn to the accuracy of the approximations
generated by the backward iterative scheme.
Specifically, since the functions
\begin{equation}
  \overline Q^{(t)}_{z, \ttheta}(a) -\frac1{|\AAA|} \sum_{\tilde a}   \overline Q^{(t)}_{z, \ttheta}(\tilde a)
\end{equation}
together represent
the direction of ascent (in logit coordinates) at the policy parameter vector $\ttheta$,
we compare these functions throughout training with analogs derived from
the full-information Markov chain.

Consider a fixed incumbent (count-based) policy parameter vector $\ttheta$. 
Let $Q^{(t)}_{\ttheta,\mdp}(z, h, a)$ denote the state-action value function associated with the MDP on the full-information state space induced by the policy,
and let $\mu^{(t)}_{\ttheta,\mdp}(z,h)$ denote the corresponding state distributions.
(See Equations~(\ref{eq:recMarkovVQ}) and (\ref{eq:forwardextendedstate}) with $s=(z,h)$.)
Write $Q^{(t)}_{\ttheta,\qdp}(z, a)$ instead of $\overline Q^{(t)}_{z, \ttheta}(a)$.
To compare these two functions, we extract a new Q-function from the full-information Q-function via
\begin{align*}
    Q^{(t)}_{\ttheta,\extract}(z,a) =\sum_{h} \mu^{(t)}_{\ttheta,\mdp}(h|z)\times Q^{(t)}_{\ttheta,\mdp}(z,h,a),
\end{align*}
where once again it is understood that the right-hand side is zero if $\mu^{(t)}_{\ttheta,\mdp}(z) = 0$. (In the dynamic pricing example, the `supports'   of $ Q^{(t)}_{\ttheta,\extract}$  and     $Q^{(t)}_{\ttheta,\qdp}$ match.)
We compare $Q^{(t)}_{\ttheta,\extract}(z,a) -\frac1{|\AAA|} \sum_{\tilde a}  Q^{(t)}_{\ttheta,\extract}(z,\tilde a)  $ with $Q^{(t)}_{\ttheta,\qdp}(z,a) -\frac1{|\AAA|} \sum_{\tilde a} Q^{(t)}_{\ttheta,\qdp}(z,\tilde a) $, viewed as vectors indexed by $(t,z,a)$, using two metrics:
the cosine of the angle between the two vectors and the ratio of their Euclidean lengths. 
Both metrics over the course of training are shown in the figure below.
The vectors exhibit near-perfect alignment: 
the cosine values correspond to angles of less than $2^\circ$ on a $90^\circ$ scale. 
Apart from an initial 0.5\% norm ratio, all later ratios stay below 0.2\%.

\begin{center}
    \includegraphics[height=5.2cm]{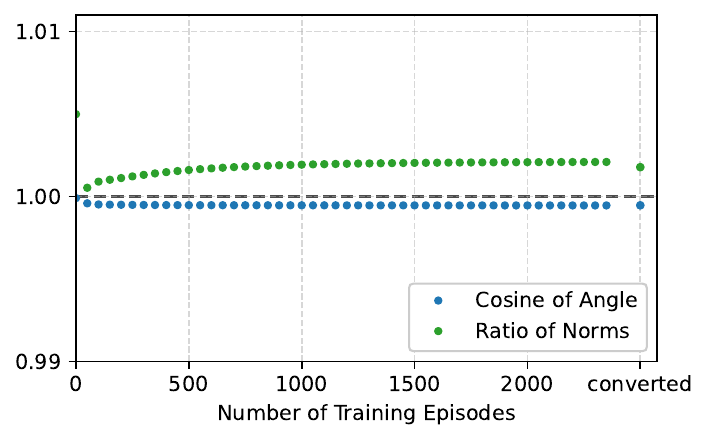}
  \end{center}

\paragraph{Setting learning rates.}
Setting the learning rate requires trading off quality and time. 
For problems without service level constraints such as the ones we discussed
so far, in our experience, properly chosen learning rates require only 5-50 training episodes to achieve a near-optimal solution.
The results are shown in the table below. 
  
\begin{center}
\begin{tabular}{l|rrrrrrr}
Learning rate $\eta$ & $1$ & $10$ & $100$ & $1{,}000$ & $10{,}000$ & $100{,}000$ & $1{,}000{,}000$ \\
\hline
QPLEX value of QDP policy & 3.1200 & 3.1200 & 3.1200 & 3.1200 & 3.1193 & 3.0973 & 0.5340 \\
Number of episodes & 2372 & 469 & 111 & 66 & 26 & 4 & 1 \\
Training time (s) & 12.333 & 2.346 & 0.536 & 0.299 & 0.113 & 0.020 & 0.007
\end{tabular}
\end{center}

\subsection{Comparison with Q-Learning}
\label{sec:exampleB}

In this subsection, we emphasize the computational efficiency of QDP on large-scale problems and compare it with a Q-learning algorithm. This efficiency stems from the QPLEX infrastructure, which is fully deterministic and bypasses the curse of dimensionality. 
For this purpose we modify the example of Section~\ref{sec:exampleA} as follows.
The number of servers $n$ is set to $20$ and the buffer size $b$ is set to $5$.
We use the DEC arrival rate shape, the UniM service duration distribution,
and we use $c_W=0.1$ and $c_T=1.5$.

Under the MDP formulation, the resulting full-information state space 
for this example contains approximately $4.82\times 10^{11}$ states,
and therefore exact dynamic programming is computationally intractable.
To provide a baseline for comparison with QDP, we instead consider reinforcement learning methods 
that are commonly used to overcome the curse of dimensionality. The full-information state space is far too large to implement Q-learning in a purely tabular way, so we use state aggregation. 
Specifically, we merge all underlying full-information time-augmented states of the form $(t,z,h)$ into the single state index $(t,z)$ in the Q-table, 
even though this induces so-called \emph{state aliasing} because those states differ in their transition dynamics. 
More precisely, for each state $(z,h)$ at time $t$, the procedure consists of four steps.
In Step 1, we sample an action $a$ using the $\epsilon$-greedy policy based on the incumbent Q-function. (Based on empirical results, we set $\epsilon = 1$, corresponding to pure exploration, as this choice produced the best performance.)
In Step 2, we sample a transition to $(z',h')$.  
In Step 3, we update the Q-function via the temporal-difference rule 
\begin{equation}
  q^{(t)}(z,a) \leftarrow q^{(t)}(z,a) + \alpha\times \left(r^{(t)}(z,a) + \max_{a'} q^{(t+1)}(z',a') - q^{(t)}(z,a)\right),
\end{equation}
where $q^{(t)}(z,a)$ denotes the estimate of the aggregate state-action value associated with state $z$ and action $a$ at time $t$.
In Step 4, we increment the time index $t$ and repeat along the sampled trajectory until the horizon is reached.

We run our state-aggregate Q-learning algorithm with six constant learning rates (0.1, 0.05, 0.025, 0.01, 0.005, and 0.0025) and train for 100 million episodes. 
During training, we evaluate the incumbent policy via one million Monte Carlo simulations every million episodes.
The figure below shows the learning curves, highlighting the curve for the best learning rate in retrospect.
Also shown is the value of the QDP policy, executed with learning rate $\eta=1000.0$ and our usual stopping tolerance $\epsilon=10^{-6}$.

\begin{center}
        \includegraphics[
            height=0.2\textheight,
        ]{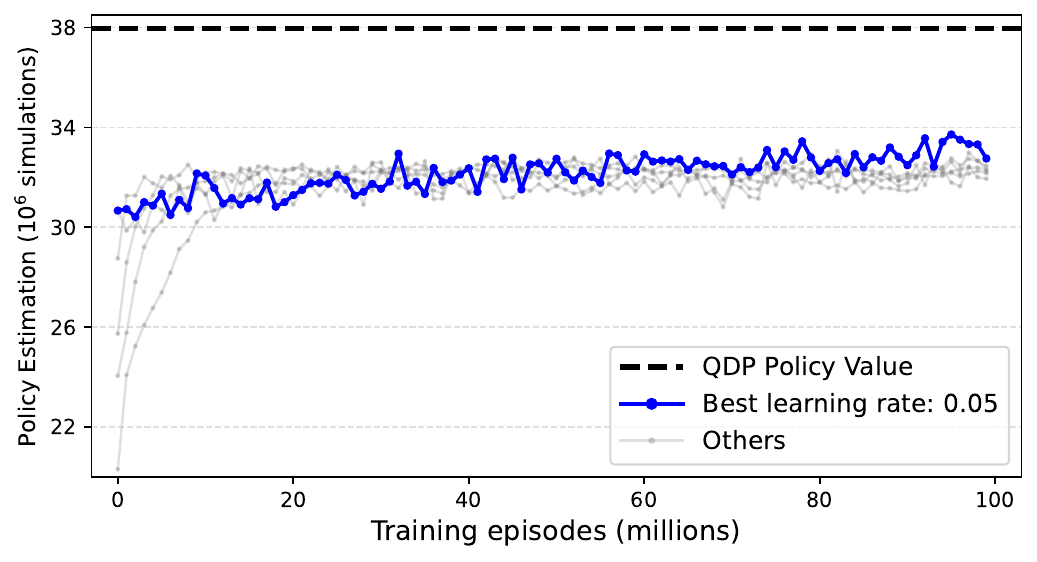}
      \end{center}

      The best trained value obtained by state-aggregated Q-learning is 13.6\% lower than the value of the QDP policy. 
      The learning curve does not seem to converge to the optimal value, but it plateaus at the value of a suboptimal aggregated policy and the policies are markedly inferior across all learning rates. 
      The QPLEX value of the QDP policy is confirmed via simulation, see below. 
      Even though training QDP is more expensive on a per-episode basis, the overall training requirement is dramatically reduced.
      
\begin{center}
\begin{tabular}{lccc}
\toprule
 Policy & Sim. Value (99.7\% CI) & Training Episodes \\
\midrule
Best Q-learning & $32.7489 \pm 0.0034$ & $100$ million \\
QDP  & $37.9646 \pm 0.0025$ & $28$\\
\bottomrule
\end{tabular}
\end{center}

Comparing the policies below,
we see that the QDP policy is monotone even without any structural input.
A monotone policy can be quite practical.
(We do encounter non-monotone
  policies in the presence of chance constraints, see Sections~\ref{sec:exampleC} and \ref{sec:exampleD}.)
The state-aggregate Q-learning policy fails to capture this monotonicity
and exhibits more erratic behavior.
We also note that the policy choice when $z = 25$ for the first 10 periods is `N/A' (not applicable).
Here is why.  In the QPLEX model described in Section~\ref{sec:dynamicpricingexample},
new arrivals can begin service if a server becomes available.  Consequently, 
the price chosen when $z = n+b = 25$ is relevant as long as there is a {\em positive} probability of a service completion during the period.  
Since the service duration distribution is UniM in this example, arrivals in the first period cannot finish their service before $t = 11$.

\begin{center}
      \includegraphics[height=0.39\textheight]{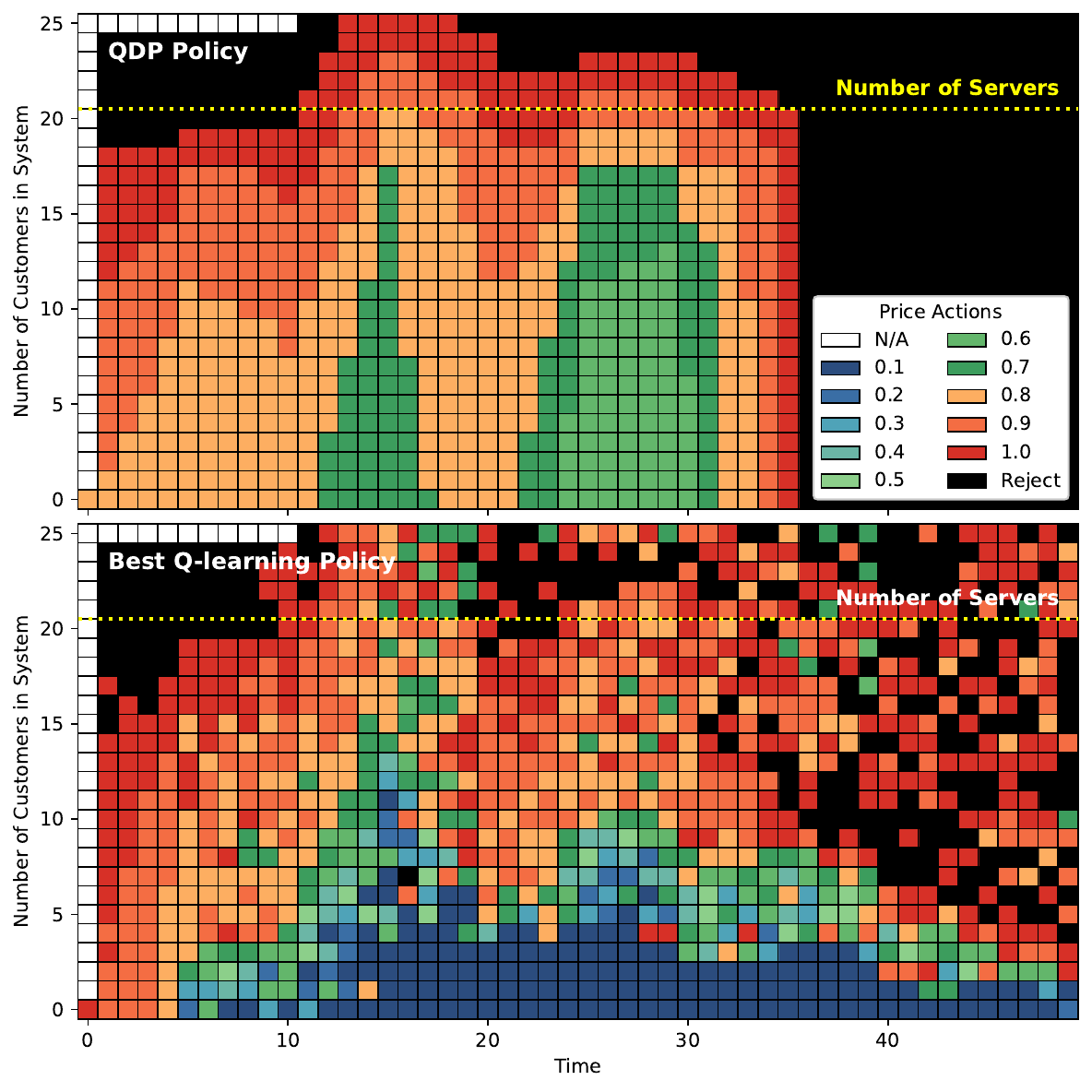}
\end{center}

The QDP policy
captures the nonstationarity and the interaction between the arrival rate functions and the service distribution, which is uniform on $\{11,\ldots,20\}$. More specifically:  

\begin{enumerate}
\item The QDP policy uses the highest prices (0.8, 0.9, and 1.0) for $t=0,\ldots,11$ (the first two arrival plateaus), 
reserving capacity for anticipated high-value demand even at the cost of initially incurring more idleness.
  
\item It rejects arrivals when $t\le 4$ and $z=19$, even with one idle server, because even the highest price can bring multiple customers at once and force buffering; with the chosen service duration distribution, near-term completions are impossible, so this buffering risk incurs the substantial waiting cost $c_W=0.1$ per period. This phenomenon disappears later as completion probabilities increase.
  
\item Between plateaus, the QDP policy lowers prices anticipatively (e.g., starting at $t=13$--$16$ ahead of the true valley $t=15$--$19$), balancing reserving capacity for future high-value arrivals against admitting customers to start service earlier.
  
\item The lowest-price region occurs around $t=27$ rather than $t=15$, reflecting recognition of the arrival rate function shape.

\item From $t=36$ onward, the QDP policy correctly rejects all customers to avoid terminal cost risk. 
  This is indeed optimal and here is an explanation.
  A customer who arrives in period $36$ and 
  can immediately begin their service will complete it by time $T = 50$ (and thus avoid the terminal penalty) only if their service duration is less than 14, which occurs with probability 0.3.  
Since the expected reward for such an arrival is at most $1 -  (1 - 0.3) \times (1.5) = - 0.05 < 0$, it does not pay to accept any customers at time $t = 36$, and clearly any time $t > 36$.
\emph{The Q-learning policy does not prescribe this behavior}.

\end{enumerate}

\subsection{Incorporating Service Level Chance Constraints}

We replace the waiting and terminal cost structure with a more challenging per-period service level chance constraint
\begin{align}\label{eq:servicelevelcons}
    \mathbb{P}(\text{the buffer is nonempty at time } t)\le \alpha,
\end{align}
for all $t=1,\ldots,T$. Such constraints substantially increase the problem difficulty for conventional approaches, and none of the exact solution approaches from Section~\ref{sec:exampleA} apply even in very special settings such as small problem instances.
Our algorithm, on the other hand, can readily accommodate such chance constraints.
Recall that we optimize a penalized objective that trades off performance and constraint violation, see Section~\ref{sec:nonlinearMC_DPexample}.

\subsubsection{The QDP Policy}
\label{sec:exampleSLCconstraints}

We consider the example in Section~\ref{sec:exampleB} with the following modifications. 
First, we replace the costs with a per-period chance constraint by setting $c_W=c_T=0$ and using $\alpha= 0.05$, thus limiting the probability of a nonempty buffer. 
We apply our algorithm with learning rate $\eta=1.0$, constraint-violation penalty parameter $C=100.0$, and penalty exponent $k=1.0$.
For this example the iterate values exhibit persistent oscillations with this constant step size, and so we terminate after 5,000 episodes.

The maximum QPLEX-predicted value of 48.7499 is achieved at episode 3{,}308,
and the QPLEX-predicted value reaches 99\% of this maximum at episode 407 and all subsequent episodes.  
It takes 67 seconds to complete 500 episodes and 10 minutes to complete all 5{,}000 episodes.
The (pure) QDP policy after 5{,}000 episodes has a QPLEX-predicted \emph{revenue} of 48.7092
  and simulation with 10 million replications estimates the revenue 
  to be $48.6418\pm 0.0029$ with $99.7\%$ confidence.
The figure below shows the QPLEX-predicted probabilities of a nonempty buffer at each time for the QDP policy 
as well as the estimated probabilities using simulation.
This QDP policy effectively meets the service level constraints for all practical purposes:
the maximum estimated probability is 0.0523 in the simulation,
and quite a few of the constraints are (empirically) nearly binding.
  
\begin{center}
\includegraphics[width=8cm]{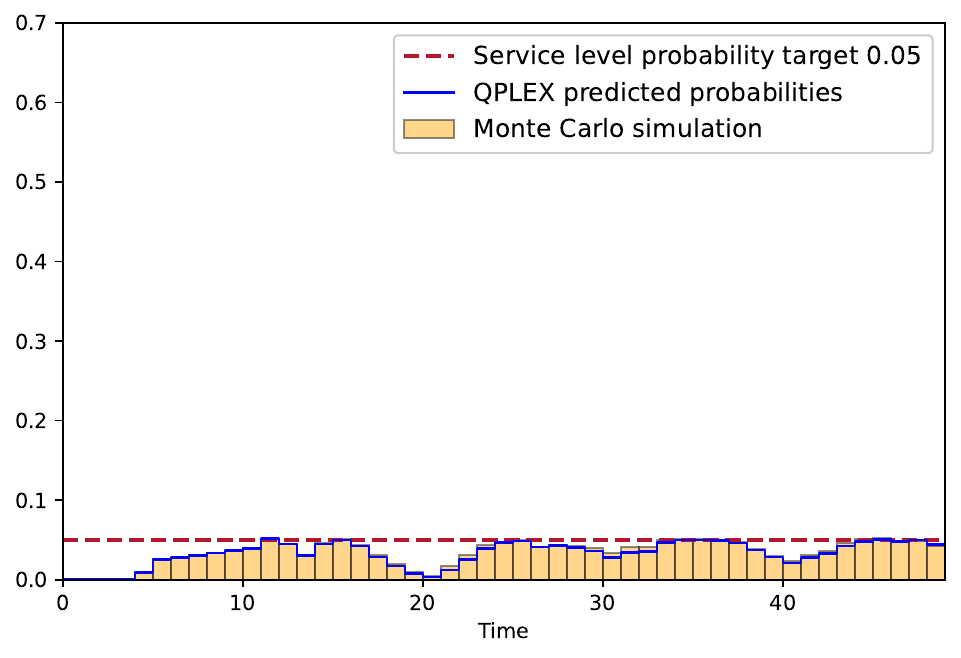}
\end{center}

The policy found by our algorithm is shown below.
We note that the policy is \emph{not} necessarily monotone
during the intervals where the service-level constraints are
nearly binding. 
Prices are highest in the periods immediately preceding such intervals.

\begin{center}
    \includegraphics[width=8.3cm]{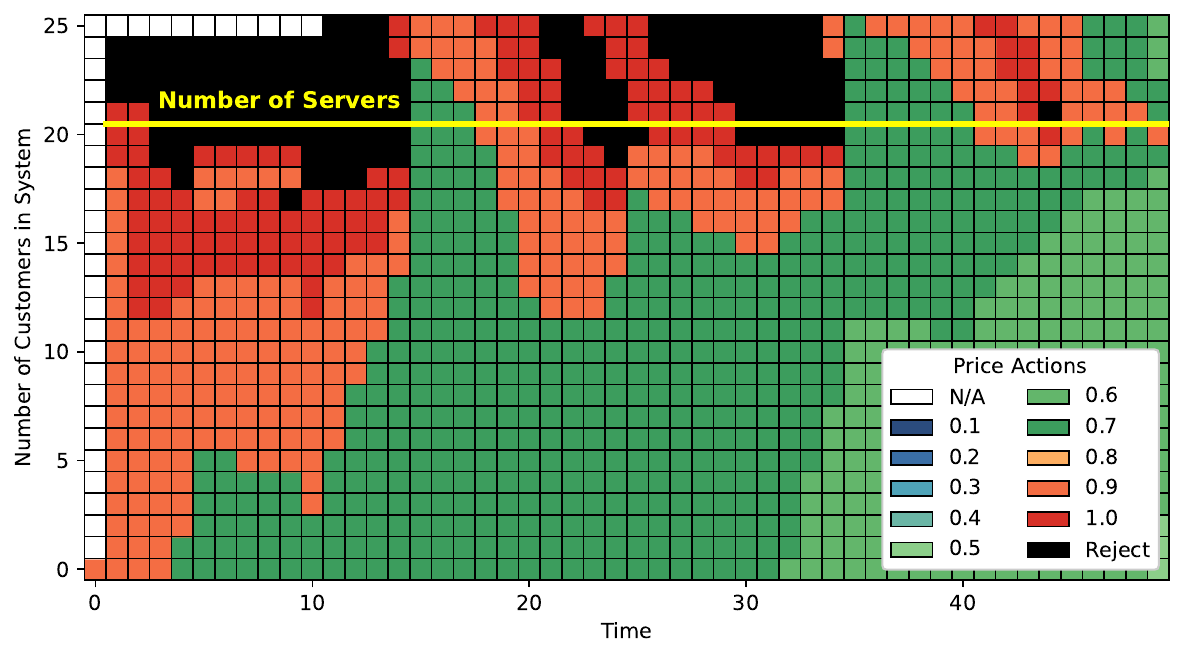}
\end{center} 

We remark that our algorithm extends to a broad class of chance constraints and distribution-dependent objectives; the service-level constraint in (\ref{eq:servicelevelcons}) is only a special case.

\subsubsection{A Comparison with Exhaustive Simulation}
\label{sec:exampleC}

Let's consider how one may use discrete event simulation (DES) to solve this problem.
For a given policy, DES can simultaneously estimate the objective and the probability of feasibility. 
For this problem instance, the number of replications required to ensure a reasonable confidence that the service level constraints are being met is not an insignificant challenge.
The real challenge, however, is that the total number of dynamic policies for this problem instance is approximately $(11^{25})^{50}$.
Consequently, to use DES one either needs to restrict the policy class to reduce the number of policies considered or use some search algorithm.
 
We explore the option of restricting the policy class with exhaustive simulation. One possibility is to impose {\em time-homogeneity} in the policy space, which amounts in this setting to choosing 26 prices, one price for each $z$.
Unfortunately, the total number of policies under this restriction is approximately $11^{26}$, which is still far too many for DES to handle.
One practical approach to further narrowing down the set of admissible policies is
to restrict the policy class so that a policy assigns a shared price for $z$ in (say) 
the sets $\{0,\ldots,10\}$, $\{11,\ldots,15\}$, $\{16,\ldots,20\}$, and $\{21,\ldots,25\}$,
in effect using four experts,
and to limit the available prices to 0.2, 0.4, 0.6, 0.8, 1.0, and 1.1 (reject).
Under these restrictions, any admissible policy can be encoded as a four-dimensional vector
$(a_{\textsc{low}},\,a_{\textsc{mid}},\,a_{\textsc{high}},\,a_{\textsc{buf}})$, 
where each component is the posted price for a customer-count block. 
With six possible prices and four customer-count blocks,
there are $6^4=1296$ candidate policies, which is \emph{now} computationally feasible. 
We evaluate each candidate policy using one million Monte Carlo replications, determine the chance-constraint feasibility via the empirical distribution, 
and rank all feasible policies by simulated average reward. The top six policies appear in the table below.  
\begin{center}
\begin{tabular}{clc}
\toprule
rank & $(a_{\textsc{low}},a_{\textsc{mid}},a_{\textsc{high}},a_{\textsc{buf}})$ & mean $\pm 3\,\text{std}$ \\
\midrule
1 & $(0.8,\,1.0,\,1.1,\,0.8)$ & $37.6973 \pm 0.0078$ \\
2 & $(0.8,\,1.0,\, 1.1,\,0.6)$ & $37.6968 \pm 0.0078$ \\
3 & $(0.8,\,1.0,\, 1.1,\,0.4)$ & $37.6827 \pm 0.0078$ \\
4 & $(0.8,\,1.0,\, 1.1,\,1.0)$ & $37.6518 \pm 0.0078$ \\
5 & $(0.8,\,1.0,\, 1.1,\,0.2)$ & $37.6242 \pm 0.0078$ \\
6 & $(0.8,\,1.0,\, 1.1,\, 1.1)$ & $37.5271 \pm 0.0081$ \\
\bottomrule
\end{tabular}
\end{center}

In practice, one million replications per policy suffice to assess feasibility: within this restricted class, the chance constraint is rarely binding. Most policies either violate the $0.05$ service-level requirement by a significant margin or satisfy it with substantial slack. For the top six policies, we then run 100 million simulations to validate the ranking, which provides overwhelming empirical evidence that $(0.8,\,1.0,\,1.1,\,0.8)$ is the optimal feasible policy within the restricted class, with estimated revenue of $37.6966 \pm 0.00079$ with 99.7\% confidence. 
Note that there is a significant drop in revenue (23\%) as compared to the QDP policy. 

Our algorithm can readily accommodate the restricted admissible policy class through a slight modification of the architecture, by ``sharing parameters'' across different $(t,z)$.
Here is a general setup.
Let $\bar z$ represent a set of counters and let $\bar t$ represent a set of time epochs.
Assume we are given a collection of pairs $\{ (\bar z, \bar t) \}$ such that
the $\bar z$ partition the set of counters and the $\bar t$ partition the set of time epochs. 
In our example, the collection of pairs consists of the four elements  $\{ ((0, \ldots, 10), (0, \ldots, 49)), \ldots, ((21, \ldots, 25), (0, \ldots, 49)) \}$.
We now need to update one pmf for each pair $(\bar z, \bar t)$, and we modify the updating step (\ref{eq:expupdatepi}) as follows:
  \begin{equation}
    \theta^{(\bar t)}_{\bar z}(a)  \leftarrow \frac{\theta^{(\bar t)}_{\bar z}(a)\times \exp\left(\eta\times  \sum_{t\in \bar t} \sum_{z\in \bar z}\in \overline Q^{(t)}_{z, \ttheta}(a) \right)}
    {\sum_{\tilde a} \theta^{(\bar t)}_{\bar z}(\tilde a)\times\exp\left(\eta\times  \sum_{t\in \bar t} \sum_{z\in \bar z}\in \overline Q^{(t)}_{z, \ttheta}(\tilde a) \right)}.
\end{equation}
Using this modification, our algorithm can readily accommodate other types of policy class restrictions, 
such as non-dynamic pricing (one price for each $z$, independent of $t$) and time-varying but offline (non-adaptive) pricing (one price for each $t$, independent of $z$).

Adopting this modification of the architecture, our algorithm produces the very same (presumably optimal) policy using only $100$ training episodes (with learning rate $\eta=0.001$ 
but otherwise the same hyperparameters as before).

\subsubsection{Comparison with Memoryless-Service Approximation}
\label{sec:exampleD}

One often avoids working with general service duration distributions and instead works with memoryless distributions to substantially reduce the size of the state space. 
In the presence of chance constraints, however, even in that case,
such problems cannot be solved by standard methods because
accurate (approximate) tail probabilities are needed.
The nonlinear Markov chain framework from this paper
applies in this small-state space setting, using standard Markov
transition probabilities but with 
nonlinearities appearing in the immediate rewards.
However, QPLEX is not needed.

This subsection considers the same example as in Section~\ref{sec:exampleSLCconstraints}, but without any restriction on the policy class and with the full action set restored. We compare the QDP policy with the policy found by working with geometric
service duration distributions (with the same mean) and the
reduced state space.
As we demonstrate next, the memoryless-service
approximation can be highly misleading and may lead to policies that exhibit
severe violations of the service-level constraints.

We use our algorithm to solve this problem using a truncated geometric service duration distribution, which, modulo truncation, is equivalent to working with the small-state setting.
(We again use a learning rate of 1.0 for 5{,}000 steps and a tolerance of $10^{-6}$ to cut off the right tail and ensure via a root finding procedure that we match the mean of the original distribution.) 
Evaluating this `geometric-approximation' policy using QPLEX yields a QPLEX-predicted revenue of 54.8306 if we use the (incorrect) truncated geometric service duration distribution
and a QPLEX-predicted revenue of 48.8075 if we use the (correct) UniM service duration distribution.
Thus, there is a significant gap in predicted revenue when using the correct service distribution. 
Note that the geometric-approximation policy results in a slightly higher revenue than the QDP policy's QPLEX-predicted revenue of 48.7092.

The figure below shows for the geometric-approximation policy the probabilities of a nonempty buffer at each time under the memoryless assumption,
as well as the estimated probabilities using simulation with the correct service distribution. We can see that the probabilities under the memoryless assumption suggest that the
service level constraints are satisfied, but the true probabilities tell a markedly different story:  many constraints are severely violated.
(This is the reason for the higher revenue obtained with the policy that calculated using the truncated geometric service duration distribution.)
The conclusion for this example is that 
using the memoryless approximation in lieu of the true service duration distribution as a means to reduce run times can lead to
policies that will turn out to be \emph{highly} infeasible, and suggests that using this approximation to solve some problem instances 
may not be reliable.
Even in the absence of chance constraints, exponential-service approximations can be unreliable for problems similar to the example in Section~\ref{sec:exampleB}, as they may lead to substantially
inferior policies compared with those obtained by our algorithm.

\begin{center}
\includegraphics[width=8cm]{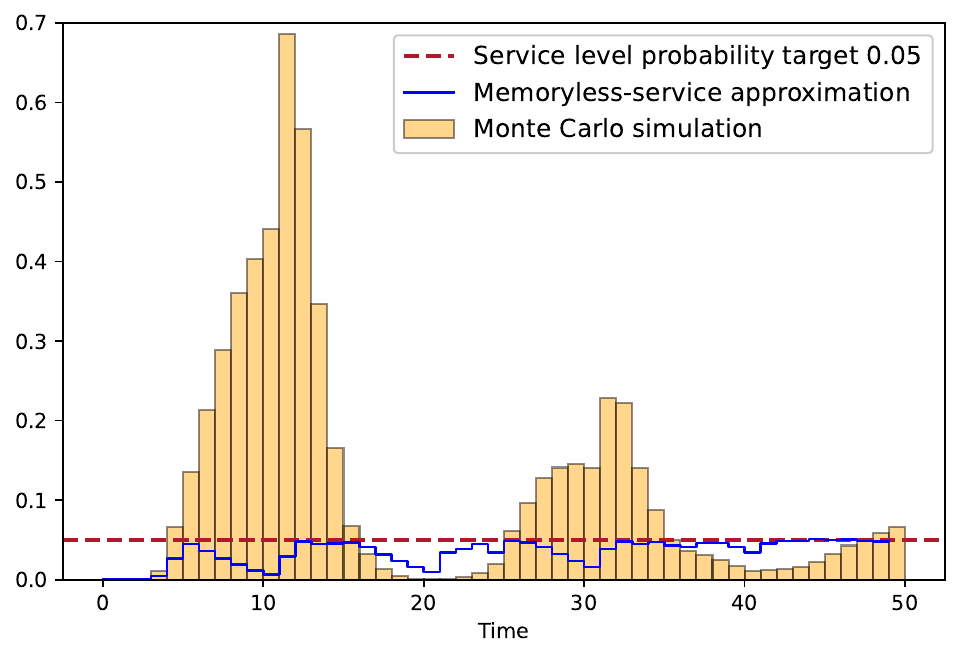}
\end{center}



  \appendix

\section{Proofs}
\subsection{Proof of Theorem~\ref{thm:pgt}}

Using the definition
\begin{equation*}
J^{(t)}_{\mu, \ttheta} = \sum_s \mu(s) \times V^{(t)}_{\mu, \ttheta}(s),
\end{equation*}
and (\ref{eq:Ptmuthetadef}), we observe that (\ref{eq:nlmcVrec}) implies the recursion
\begin{equation}
  \label{eq:nlmcJrec}
\begin{aligned}
J^{(T)}_{\mu,\ttheta}  &= \sum_s\mu(s)\times r^{(T)}_{\mu} (s) \\
  J^{(t)}_{\mu, \ttheta} &=
                           \sum_{s, a} \mu(s) \times \pi^{(t)}_{\theta^{(t)}} (a | s) \times r^{(t)}_{\mu} (s, a)   + J^{(t+1)}_{\PPP^{(t)}_{\theta^{(t)}}(\mu), \ttheta}.
  \end{aligned}
\end{equation}
The objective function $J_{\ttheta}$ in (\ref{eq:Jtheta}) can then be expressed as
\begin{equation*}
J_{\ttheta} = \sum_{\tau=0}^{t-1} \left \{ \sum_{s, a} \mu^{(\tau)}_{\ttheta}(s) \times \pi^{(\tau)}_{\theta^{(\tau)}} (a | s) \times r_{\mu^{(\tau)}_{\ttheta}}^{(\tau)} (s, a) \right\}
+ J^{(t)}_{\mu^{(t)}_{\ttheta}, \ttheta},
\end{equation*}
so that in particular $J_{\ttheta}=J^{(0)}_{\mu^{(0)},\ttheta}$.
Since the $\mu^{(\tau)}_{\ttheta}$ for $\tau \le t$ are constant in $\theta^{(t)}$, 
we thus obtain that the policy gradient $\nabla_{\theta^{(t)}} J_{\ttheta}$
can be expressed as 
\begin{equation}
  \label{eq:policygradientJt}
\nabla_{\theta^{(t)}} J_{\ttheta} = \nabla_{\theta^{(t)}} J^{(t)}_{\mu, \ttheta} \Big|_{\mu =  \mu^{(t)}_{\ttheta}}.
\end{equation}
Define, for $0\le t\le T$, 
\begin{equation}
  \label{eq:sigmatheta}
\sigma^{(t)}_{\ttheta} = \nabla_{\mu} J^{(t)}_{\mu, \ttheta} \Big|_{\mu =  \mu^{(t)}_{\ttheta}}.
\end{equation}
(The requisite differentiability in (\ref{eq:policygradientJt}) and (\ref{eq:sigmatheta}) follows from (\ref{eq:Ptmuthetadef}) and (\ref{eq:nlmcJrec}).)
From (\ref{eq:nlmcJrec}), (\ref{eq:policygradientJt}), (\ref{eq:sigmatheta}), and the chain rule we then obtain that
\begin{equation}
\label{eq:omegatthetaform}
\nabla_{\theta^{(t)}} J_{\ttheta} = \sum_{s,a} \mu^{(t)}_{\ttheta}(s) \times \nabla_{\theta^{(t)}} \pi^{(t)}_{\theta^{(t)}} (a | s) \times r^{(t)}_{\mu^{(t)}_{\ttheta}}(s, a)
+ \nabla_{\theta^{(t)}} \PPP^{(t)}_{\theta^{(t)}}(\mu) \Big|_{\mu = \mu_{\ttheta}^{(t)}} \times \sigma^{(t+1)}_{\ttheta}. 
   \end{equation}
As a result of (\ref{eq:Ptmuthetadef}), the second term on the right-hand side of (\ref{eq:omegatthetaform}) is
\begin{equation*}
\nabla_{\theta^{(t)}} \PPP^{(t)}_{\theta^{(t)}}(\mu) \Big|_{\mu = \mu_{\ttheta}^{(t)}} \times \sigma^{(t+1)}_{\ttheta} =
\sum_{s'} \left( \sum_{s,a} \mu^{(t)}_{\ttheta}(s) \times \nabla_{\theta^{(t)}} \pi^{(t)}_{\theta^{(t)}} (a | s) \times p^{(t)}_{\mu^{(t)}_{\ttheta}}(s' | s, a) \right) \times \sigma_{\ttheta}^{(t+1)}(s').
\end{equation*}
After grouping terms, the policy gradient can be expressed as
\begin{align*}
\nabla_{\theta^{(t)}} J_{\ttheta} &= \sum_{s, a} \mu^{(t)}_{\ttheta}(s) \times \nabla_{\theta^{(t)}} \pi^{(t)}_{\theta^{(t)}} (a | s) \times
\left( r^{(t)}_{\mu^{(t)}_{\ttheta}}(s, a) + \sum_{s'}  p^{(t)}_{\mu^{(t)}_{\ttheta}}(s' | s, a) \times \sigma^{(t+1)}_{\ttheta}(s') \right) \\
&= \sum_{s, a} \mu^{(t)}_{\ttheta}(s) \times \nabla_{\theta^{(t)}} \pi^{(t)}_{\theta^{(t)}} (a | s) \times Q_{\mu^{(t)}_{\ttheta},\sigma^{(t+1)}_{\ttheta}}^{(t)}(s, a),
\end{align*}
which establishes (\ref{eq:pgt}).

Next we turn to establishing (\ref{eq:sigmatrec}).
From (\ref{eq:nlmcJrec}) and the chain rule we obtain 
   \begin{align}
     \sigma^{(T)}_{\ttheta}&= r^{(T)}_{\mu^{(T)}_{\ttheta}} +\sum_{\tilde s} \mu^{(T)}_{\ttheta}(\tilde s)\times \nabla_\mu r^{(T)}_{\mu^{(T)}_{\ttheta}}(\tilde s) \nonumber \\
     \sigma^{(t)}_{\ttheta} &=  \nabla_\mu \left.\left( \sum_{\tilde s,a} \mu(\tilde s) \times \pi^{(t)}_{\theta^{(t)}}(a | \tilde s) \times r_{\mu}^{(t)}(\tilde s, a) \right)\right|_{\mu = \mu_{\ttheta}^{(t)}}  + \nabla_{\mu} \PPP^{(t)}_{\theta^{(t)}}(\mu) \Big|_{\mu = \mu_{\ttheta}^{(t)}} \times \sigma^{(t+1)}_{\ttheta}. \label{eq:sigmatthetaformt}
   \end{align}
Component $s$ of the first term on the right-hand side of (\ref{eq:sigmatthetaformt}) is 
\begin{equation*}
\sum_a \pi^{(t)}_{\theta^{(t)}} (a | s) \times r^{(t)}_{\mu^{(t)}_{\ttheta}}(s, a) + \sum_{\tilde s, a} \mu^{(t)}_{\ttheta} (\tilde s) \times \pi^{(t)}_{\theta^{(t)}} (a | \tilde s)
\times \frac{\partial}{\partial \mu(s)} r^{(t)}_{\mu^{(t)}_{\ttheta}} (\tilde s, a).
\end{equation*}
Using (\ref{eq:Ptmuthetadef}),  
component $s$ of the second term on the right-hand side of (\ref{eq:sigmatthetaformt}) is 
\begin{equation*}
\sum_{s'} \left\{ \sum_a \pi^{(t)}_{\theta^{(t)}} (a | s) \times p_{\mu^{(t)}_{\ttheta}}^{(t)}(s' | s, a) + \sum_{\tilde s, a} \mu^{(t)}_{\ttheta}(\tilde s) \times \pi^{(t)}_{\theta^{(t)}} (a | \tilde s) 
\times \frac{\partial}{\partial \mu(s)} p^{(t)}_{\mu^{(t)}_{\ttheta}}(s' | \tilde s, a) \right\} \times \sigma^{(t+1)}_{\ttheta}(s').
\end{equation*}
After grouping terms, (\ref{eq:sigmatthetaformt}) can be written as 
\begin{multline*}
\sigma^{(t)}_{\ttheta}(s) = \sum_a \pi^{(t)}_{\theta^{(t)}} (a | s) \times \left[ r^{(t)}_{\mu^{(t)}_{\ttheta}}(s, a) + \sum_{s'}  p^{(t)}_{\mu^{(t)}_{\ttheta}}(s' | s, a) \times \sigma^{(t+1)}_{\ttheta}(s') \right] \\
+ \sum_{\tilde s, a} \mu^{(t)}_{\ttheta} (\tilde s) \times \pi^{(t)}_{\theta^{(t)}} (a | \tilde s) \times
\left[ \frac{\partial}{\partial \mu(s)} r^{(t)}_{\mu^{(t)}_{\ttheta}} (\tilde s, a) + \sum_{s'} \frac{\partial}{\partial \mu(s)} p^{(t)}_{\mu^{(t)}_{\ttheta}}(s' | \tilde s, a)  \times \sigma^{(t+1)}_{\ttheta}(s')  \right],
\end{multline*}
which is (\ref{eq:sigmatrec}) for $t < T$ expressed in terms of $Q^{(t)}_{\mu^{(t)}_{\ttheta}, \sigma^{(t+1)}_{\ttheta}}(s, a)$.

\subsection{Proof of Proposition~\ref{prop:fisherrecursion}}

We fix pmf $\mu^{(0)}$ on state space $\SSS$ and a policy parameter vector $\ttheta$.  
The pmfs $\mu^{(t)}_{\ttheta}$ are recursively defined via $\mu^{(t+1)}_{\ttheta} = \PPP^{(t)}_{\theta^{(t)}}(\mu^{(t)}_{\ttheta})$, see (\ref{eq:Ptmuthetadef}).
For $\tau > t$ the pmf $\mu^{(\tau)}_{\ttheta}$ can be thought of as a function of $\mu^{(t)}_{\ttheta}$ and $\theta^{(t)},\ldots,\theta^{(T-1)}$,
and so with this understanding we write $\nabla_{\theta^{(t)}} \mu^{(\tau)}_{\ttheta}$ and $\nabla_{\mu^{(t)}_{\ttheta}} \mu^{(\tau)}_{\ttheta}$.  
(Keep in mind that these are matrices of sizes $M \times |\SSS|$ and $|\SSS| \times |\SSS|$, respectively, where $M$ is the ambient dimension of $\Theta^{(t)}$.) 
All expectations are understood to be under $q$.  
To shorten the expressions to follow we define random variables
\begin{align*} 
  Y^{(t)}(\ttheta) &= \nabla_{\theta^{(t)}} \log \pi^{(t)}_{\theta^{(t)}}(A^{(t)} | S^{(t)}),\\
  Z^{(t)}(\ttheta) &= \nabla_{\mu} \log p^{(t)}_{\mu^{(t)}_{\ttheta}}(S^{(t+1)} | S^{(t)}, A^{(t)}),
\end{align*}
In what follows, $0 \le t_1 \le t_2 <T$.

From the definition of $q_{\ttheta}$ in (\ref{qdef0}), $\log q_{\ttheta}(S^{(0)},A^{(0)}, \ldots, S^{(T)}) $ equals
\begin{equation*}
 \log \mu^{(0)}(S^{(0)}) + \sum_{\tau=0}^{T-1} \log \pi^{(\tau)}_{\theta^{(\tau)}}(A^{(\tau)}|S^{(\tau)}) + \sum_{\tau=0}^{T-1} \log p_{\mu^{(\tau)}_{\ttheta}}^{(\tau)}(S^{(\tau+1)}| S^{(\tau)}, A^{(\tau)}).
\end{equation*}
Using the chain rule, we obtain that
\begin{align*}
&\nabla_{\theta^{(t)}} \log q_{\ttheta}(S^{(0)},A^{(0)}, \ldots, S^{(T)}) \\ &= \nabla_{\theta^{(t)}} \log \pi^{(t)}_{\theta^{(t)}}(A^{(t)}|S^{(t)}) + 
\sum_{\tau=t+1}^{T-1} (\nabla_{\theta^{(t)}} \mu^{(\tau)}_{\ttheta}) \times \nabla_{\mu^{(\tau)}_{\ttheta}} \log p_{\mu^{(\tau)}_{\ttheta}}^{(\tau)}(S^{(\tau+1)}| S^{(\tau)}, A^{(\tau)}) \\
&= Y^{(t)}(\ttheta) + \sum_{\tau = t+1}^{T-1} (\nabla_{\theta^{(t)}} \mu^{(\tau)}_{\ttheta}) \times Z^{(\tau)}(\ttheta).
\end{align*}
Thus, $F^{(t_1, t_2)}(\ttheta)$ in (\ref{Ft1t2def}) can be written as
\begin{equation*}
\E_{\ttheta} \left[ \left(Y^{(t_1)}(\ttheta) + \sum_{\tau = t_1+1}^{T-1} (\nabla_{\theta^{(t_1)}} \mu^{(\tau)}_{\ttheta}) \times Z^{(\tau)}(\ttheta) \right) \times 
\left(Y^{(t_2)}(\ttheta) + \sum_{\tau = t_2+1}^{T-1} (\nabla_{\theta^{(t_2)}} \mu^{(\tau)}_{\ttheta}) \times Z^{(\tau)}(\ttheta) \right)^\top  \right].
\end{equation*}
There are many terms in the expression inside the expectation; however, most of these terms vanish. 

Fix some finite set $\XXX$, and let $p_\zeta$ be a pmf of $\XXX$ indexed by some parameter $\zeta$ such that
$\nabla p_\zeta(x)$ exists for each $x \in supp(p_\zeta)$.  For each $x \in supp(p_\zeta)$ the {\em score function} 
$\nabla_{\zeta} \log p_\zeta(x)$ has a well-known property that its expectation (under $p_\zeta$) is zero. Indeed, 
\begin{equation} \label{eq:keyscorefctfact}
\sum_x p_\zeta(x)\times \nabla_{\zeta} \log p_\zeta(x)
= \sum_{x} \nabla_{\zeta} p_\zeta(x) = 0,
\end{equation}
since $\sum_{x} p_\zeta(x) = 1$.

From the definition of $q_{\ttheta}$ in (\ref{qdef0}) and the nonlinear Markov property in (\ref{eq:defNLMP}), we deduce that
\begin{equation} \label{eq:Ztcond=0}
\E_{\ttheta}[Z^{(t)}(\ttheta) | S^{(0)}, A^{(0)}, \ldots, S^{(t)}, A^{(t)}] 
= \sum_{s^{(t+1)}} p_{\mu^{(t)}_{\ttheta}}(s^{(t+1)} | S^{(t)}, A^{(t)}) \times \nabla_{\mu^{(t)}_{\ttheta}} \log p_{\mu^{(t)}_{\ttheta}}(s^{(t+1)} | S^{(t)}, A^{(t)})
= 0,
\end{equation}
where the second equality follows by applying (\ref{eq:keyscorefctfact}) with $p = p_{\mu^{(t)}_{\ttheta}}(\cdot | S^{(t)}, A^{(t)})$ and $\zeta = \mu^{(t)}_{\ttheta}$.
From the definition of $q_{\ttheta}$ and the Markov policy property in (\ref{eq:piq}),
\begin{equation} \label{eq:Ytcond=0}
\E_{\ttheta}[Y^{(t)}(\ttheta) | S^{(0)}, A^{(0)}, \ldots, S^{(t)}] = \sum_{a^{(t)}} \pi^{(t)}_{\theta^{(t)}}(a^{(t)} | S^{(t)}) \times \nabla_{\theta^{(t)}} \log \pi^{(t)}_{\theta^{(t)}}(a^{(t)} | S^{(t)}) = 0,
\end{equation}
where the second equality follows by applying (\ref{eq:keyscorefctfact}) with $p = \pi^{(t)}_{\theta^{(t)}}(\cdot | S^{(t)})$ and $\zeta = \theta^{(t)}$.
Application of the tower law 
to (\ref{eq:Ztcond=0}) and (\ref{eq:Ytcond=0}) yields the following lemma.

\begin{lemma}
  \label{lem:reinserted}  
  The random vectors $Y^{(t)}(\ttheta)$ and $Z^{(t)}(\ttheta)$ satisfy the following identities:
  \begin{enumerate}
    \item 
      $\E_{\ttheta}[Z^{(t)}(\ttheta)] = \E_{\ttheta}[Y^{(t)}(\ttheta)] = 0$ for all $t$.
    \item
$\E_{\ttheta}[Z^{(t_1)}(\ttheta) \times (Z^{(t_2)}(\ttheta))^\top] = \E_{\ttheta}[Y^{(t_1)}(\ttheta) \times (Y^{(t_2)}(\ttheta))^\top] = 0$ for all $t_1,t_2$ with $t_1 \neq t_2$.
\item
$\E_{\ttheta}[Z^{(t_1)}(\ttheta) \times (Y^{(t_2)}(\ttheta))^\top] = 0$ for all $t_1,t_2$.
\end{enumerate}
\end{lemma}
As a direct consequence of the identities in this lemma, $F^{(t_1, t_2)}(\ttheta)$ can be written as
\begin{align} 
F^{(t_1, t_2)}(\ttheta) &= \E_{\ttheta}[Y^{(t_1)}(\ttheta) \times (Y^{(t_2)}(\ttheta))^\top]
+ \sum_{\tau = t_2+1}^{T-1} (\nabla_{\theta^{(t_1)}} \mu^{(\tau)}_{\ttheta}) \times \E_{\ttheta}[ Z^{(\tau)}(\ttheta) \times (Z^{(\tau)}(\ttheta))^\top] \times (\nabla_{\theta^{(t_2)}} \mu^{(\tau)}_{\ttheta})^\top \nonumber \\
&= \E_{\ttheta}[Y^{(t_1)}(\ttheta) \times (Y^{(t_2)}(\ttheta))^\top]
+ \sum_{\tau = t_2+1}^{T-1} (\nabla_{\theta^{(t_1)}} \mu^{(\tau)}_{\ttheta}) \times M^{(\tau)}(\ttheta) \times (\nabla_{\theta^{(t_2)}} \mu^{(\tau)}_{\ttheta})^\top .   \label{Gt1t2derivedform0}
\end{align}
For $\tau > t_2$, the chain rule implies that
\begin{align}
\nabla_{\theta^{(t_1)}} \mu^{(\tau)}_{\ttheta} &= (\nabla_{\theta^{(t_1)}} \mu^{(t_2+1)}_{\ttheta}) \times (\nabla_{\mu^{(t_2+1)}_{\ttheta}} \mu^{(\tau)}_{\ttheta}), \nonumber \\
\nabla_{\theta^{(t_2)}} \mu^{(\tau)}_{\ttheta} &= (\nabla_{\theta^{(t_2)}} \mu^{(t_2+1)}_{\ttheta}) \times (\nabla_{\mu^{(t_2+1)}_{\ttheta}} \mu^{(\tau)}_{\ttheta}), \label{eq:gradthetamutau}
\end{align}
and so the summand on the right-hand side of (\ref{Gt1t2derivedform0}) can be expressed as
\begin{equation*}
(\nabla_{\theta^{(t_1)}} \mu^{(t_2+1)}_{\ttheta}) \times G^{(t_2+1)}(\ttheta) \times (\nabla_{\theta^{(t_2)}} \mu^{(t_2+1)}_{\ttheta})^\top,
\end{equation*}
where the matrices $G^{(t)}(\ttheta)$ are defined, for $1\le t\le T$, via 
\begin{equation}
  \label{eq:Gtrep1}
G^{(t)}(\ttheta) = \sum_{\tau = t}^{T-1} (\nabla_{\mu^{(t)}_{\ttheta}} \mu^{(\tau)}_{\ttheta}) \times M^{(\tau)}(\ttheta)  \times (\nabla_{\mu^{(t)}_{\ttheta}} \mu^{(\tau)}_{\ttheta})^\top,
\end{equation}
which should be interpreted as zero for $t=T$.
Thus, 
\begin{equation*}
F^{(t_1, t_2)}(\ttheta) = \E_{\ttheta}[Y^{(t_1)}(\ttheta) \times (Y^{(t_2)}(\ttheta))^\top]  + (\nabla_{\theta^{(t_1)}} \mu^{(t_2+1)}_{\ttheta}) \times G^{(t_2+1)} (\ttheta)\times (\nabla_{\theta^{(t_2)}} \mu^{(t_2+1)}_{\ttheta})^\top,
\end{equation*}
where the first term on the right-hand side of (\ref{Gt1t2derivedform0}) is only nonzero if $t_1 = t_2 = t$, in which case it equals $K^{(t)}(\ttheta)$, as claimed.

We proceed to the calculation of $G^{(t)}(\ttheta)$.
Since $\nabla_{\mu^{(t)}_{\ttheta}} \mu^{(t)}_{\ttheta} $ is the identity matrix and, similarly to (\ref{eq:gradthetamutau}), for $\tau>t$,
\begin{equation*}
  \nabla_{\mu^{(t)}_{\ttheta}} \mu^{(\tau)}_{\ttheta} = (\nabla_{\mu^{(t)}_{\ttheta}} \mu^{(t+1)}_{\ttheta}) \times (\nabla_{\mu^{(t+1)}_{\ttheta}} \mu^{(\tau)}_{\ttheta}),
\end{equation*}
it follows from its representation in (\ref{eq:Gtrep1}) that $G^{(t)}(\ttheta)$ can be recursively calculated as claimed.

\subsection{Proof of Proposition~\ref{prop:criticalpoints}}

The necessary condition for optimality states that for each $t$ and $z$, for all feasible directions $d \in \R^{|\AAA|}$,
\begin{equation*}
d^\top \times \nabla_{\theta^{(t)}_z} J_{\ttheta} \le 0.
\end{equation*}
In view of (\ref{eq:omegast}), this is equivalent to having,
for each $t$ and $z$ with $\mu^{(t)}_{\ttheta}(\SSS_z)>0$,
\begin{equation*}
\sum_a d_a \times \overline Q^{(t)}_{z,\ttheta} (a)\le 0.
\end{equation*}
The sum of the elements of a feasible direction $d$ must be zero.  
If all elements $\ttheta$ are positive, then the elements of each $d$ may be taken to be positive or negative.  A straightforward proof by contradiction readily shows that, for each $t$ and $z$, the $\overline Q^{(t)}_{z,\ttheta}(a)$ must then all be equal.

Consider some time $t$ and expert $z$ for which at least one of the elements
of $\theta^{(t)}_z$ is zero.
Here, a feasible direction $d$ has the additional property that $d_{\tilde a} \ge 0$ for each $\tilde a$
with $\theta^{(t)}_z(\tilde a)=0$.
Taking $d_{\tilde a}=0$ for each such $\tilde a$, the argument in the previous paragraph
shows that the $\overline Q^{(t)}_{z,\ttheta}(a)$ for $a$ in the support in $\theta^{(t)}_z$ must then all be equal.
Let $\bar q^{(t)}_z$ denote this common value. 
It remains to show that $\overline Q^{(t)}_{z,\ttheta}(\bar a)\le \bar q^{(t)}_z$ if $\theta^{(t)}_z(\bar a) = 0$.
Fix such an $\bar a$. Set $d_{\tilde a} = 0$ for all $\tilde a \neq \bar a$ with $\theta^{(t)}_z(\tilde a) = 0$, so that
\begin{equation*}
  \sum_a d_a \times \overline Q_{z,\ttheta}^{(t)} (a) =d_{\bar a} \times \overline Q^{(t)}_{z,\ttheta}(\bar a)
-d_{\bar a} \times \bar q^{(t)}_z.
\end{equation*}
Since $d_{\tilde a}$ can be taken as positive, we deduce that $\overline Q^{(t)}_{z,\ttheta}(\bar a)\le \bar q^{(t)}_z$, as required.

\subsection{Proof of Proposition~\ref{prop:naturalgradientmoesoft}}

We begin with a lemma that gives easily computable expressions for $\hat F(\psi)$ and its Moore-Penrose inverse $\hat F(\psi)^\dagger$
used in the proof.  
The matrix $P$ below is given by 
\begin{equation} \label{eq:matrixPdef}
P = I - \frac{1}{|\AAA|} \bm 1 \times \bm 1^\top,
\end{equation}
where $\bm 1$ is the column vector of size $|\AAA|$ of all ones.
The symmetric matrix $\hat F(\psi)$ has one singularity (each row and column sum vanishes), and so its null space is the linear span of $\bm 1$
and its range consists of all vectors $v$ orthogonal to $\bm 1$.  
Since  $P$ is idempotent ($P^2 = P$) and $P \times v = v$ for all $v$ such that $v^\top \times \bm 1 = 0$,
the symmetric matrix $P$ is the orthogonal projection matrix onto the range of $\hat F(\psi)$.
Observe that $P \times v$ merely \emph{centers} $v$ by subtracting the average from each component so that the sum of the components becomes zero, e.g., if $v = (1, 3, 8)^\top$, the average is 4, so $P \times v = (-3, -1, 4) ^\top$.  Since centering a vector $v$ involves subtracting the average from each component, this operation is unchanged if the same constant is subtracted from every entry of $v$, as the mean (and thus the centered result) remains the same.

\begin{lemma} \label{lem:softmaxFIM}
We have: 
\begin{itemize}
\item[(a)] $\hat F(\psi) = \diag(\hat\pi_\psi) - \hat\pi_\psi\times \hat\pi_\psi^\top$.
\item[(b)] $\hat F(\psi)^\dagger =  P \times \diag(\hat\pi_\psi)^{-1} \times P$.
\end{itemize}
\end{lemma}

\begin{proof}
Interpreting $\hat\pi_{\psi}$ as
a (column) vector rather than a function, the corresponding score function is 
\begin{equation} \label{eq:scoresoftmax}
\nabla_{\psi} \log \hat \pi_{\psi}(a) = e_a - \hat \pi_{\psi}
\end{equation}
where $e_a$ denotes the $a$-th unit vector.  
Using this identity, the Fisher information
matrix  $\hat F(\psi)$ corresponding to $\hat \pi_\psi$ is 
\begin{align*}
\hat F(\psi) &= \sum_a \hat \pi_{\psi}(a) \times \nabla_{\psi} \log \hat \pi_{\psi}(a) \times (\nabla_{\psi} \log \hat \pi_{\psi}(a))^\top \\
&= \sum_a \hat \pi_{\psi}(a) \times (e_a \times e_a^\top - e_a \times \hat \pi_{\psi}^\top - \hat \pi_{\psi} \times e_a^\top + \hat \pi_{\psi} \times \hat \pi_{\psi}^\top)\\
&= \diag(\hat \pi_{\psi}) - 2\times \left( \sum_a \hat \pi_{\psi}(a) \times e_a\right) \times \hat \pi_{\psi}^\top + \hat \pi_{\psi} \times \hat \pi_{\psi}^\top  \\
&= \diag(\hat\pi_{\psi}) - \hat \pi_{\psi} \times \hat \pi_{\psi}^\top,
\end{align*}
which establishes part (a).
As for part (b), we first establish that $\hat F(\psi)$ and 
$P \times \diag(\hat \pi_{\psi})^{-1} \times P$ are inverses on the subspace orthogonal to $\bm 1$, i.e., 
\begin{equation*}
  \hat F(\psi) \times (P \times \diag(\hat \pi_{\psi})^{-1} \times P)= (P \times \diag(\hat \pi_{\psi})^{-1} \times P) \times \hat F(\psi) = P.
\end{equation*} 
Since the sum of the components of each row or column of $\hat F(\psi)$ is zero it follows that 
$P \times \hat F(\psi) = \hat F(\psi) \times P = \hat F(\psi)$.  Due to symmetry it remains to 
show that $P \times \diag(\hat \pi_{\psi})^{-1} \times \hat F(\psi) = P$.
Indeed, 
\begin{equation*}
P \times \diag(\hat \pi_{\psi})^{-1} \times \hat F(\psi) = P \times (I - \bm 1 \times \hat \pi_{\psi}^\top) = P - P \times \bm 1 \times \hat \pi_{\psi}^\top = P,
\end{equation*} 
since $P \times \bm 1 = 0$. 
The claim now follows from the fact that $P \times \diag(\hat \pi_{\psi})^{-1} \times P$ annihilates $\bm 1$ and therefore the null space of $\hat F(\psi)$.
\end{proof}

We now proceed to the proof of this proposition.
Both sides of (\ref{eq:naturalgradientmoesoft}) are trivially equal to zero for time-expert pairs $(t,z)$ with $\mu^{(t)}_{\ggamma}(\SSS_z)=0$.
Due to the block diagonal structure of $\tilde F(\ggamma)$, each subvector $\widetilde\nabla_{\gamma^{(t)}_z}  J_{\ggamma}$
corresponding to time $t$ and expert $z$ with $\mu^{(t)}_{\ggamma}(\SSS_z)>0$ satisfies
\begin{equation} \label{eq:natgradrep}
\widetilde\nabla_{\gamma^{(t)}_z}  J_{\ggamma}=\frac  1{\mu^{(t)}_{\ggamma}(\SSS_z)}\times \hat F(\gamma^{(t)}_z)^\dagger \times \nabla_{\gamma^{(t)}_z} J_{\ggamma}
= \frac  1{\mu^{(t)}_{\ggamma}(\SSS_z)}\times P\times \diag(\hat\pi_{\gamma_z^{(t)}})^{-1} \times \nabla_{\gamma^{(t)}_z} J_{\ggamma},
\end{equation}
where the second equality follows from Lemma~\ref{lem:softmaxFIM}(b) and the fact that $\nabla_{\gamma_z^{(t)}} J_{\ggamma}$ is orthogonal to $\bm 1$.
In what follows we fix such a $(t,z)$ and 
show that the right-hand side of (\ref{eq:natgradrep}) has the desired equivalent form.

From the policy gradient theorem (Theorem~\ref{thm:pgt}) we deduce that
\begin{align*}
  \nabla_{\gamma^{(t)}} J_{\ggamma} &= \sum_{s, a} \mu^{(t)}_{\ggamma}(s) \times \left(\nabla_{\gamma^{(t)}} \hat\pi_{\gamma^{(t)}_{z(s)}}(a) \right)\times Q_{\mu^{(t)}_{\ggamma},\sigma^{(t+1)}_{\ggamma}}^{(t)}(s, a)\\
&= \sum_{z} \mu^{(t)}_{\ggamma}(\SSS_z)\times
\sum_a \left(\nabla_{\gamma^{(t)}} \hat\pi_{\gamma^{(t)}_z}(a) \right)\times \overline Q^{(t)}_{z,\ggamma}(a),
\end{align*}
which, using (\ref{eq:scoresoftmax}), in turn implies that
\begin{equation*}
\nabla_{\gamma_z^{(t)}} J_{\ggamma}  = \mu^{(t)}_{\ggamma} (\SSS_z) \times \sum_{\tilde a} \hat \pi_{\gamma_z^{(t)}}(\tilde a) \times \overline Q^{(t)}_{z,\ggamma}(\tilde a) \times \left( e_{\tilde a} - \hat \pi_{\gamma_z^{(t)}} \right).
\end{equation*}
Since 
\begin{equation*}
  \hat \pi_{\gamma_z^{(t)}}(\tilde a) \times P\times \diag(\hat\pi_{\gamma_z^{(t)}})^{-1}\times\left( e_{\tilde a} - \hat \pi_{\gamma_z^{(t)}} \right) = P\times \left( e_{\tilde a} - \bm 1 \right) = P\times e_{\tilde a} = e_{\tilde a} - \frac1{|\AAA|} \bm 1,
\end{equation*}
the right-hand side of (\ref{eq:natgradrep}) becomes the desired expression.

\subsection{Proof of Theorem~\ref{thm:approximatenaturalgradient}}

Throughout, fix an arbitrary pure policy $\ttheta^*$.
For brevity, we suppress the dependence on $n$ and simply write $\ggamma$ for $\ggamma_n$ 
(with $n$ sufficiently large) and  $\ttheta(\ggamma)\to\ttheta^*$ for the convergence along the subsequence $\{\ggamma_n\}$.
Since the set $\UUU(\ggamma_n)$ is independent of $n$, we denote it simply by $\UUU$.

We write $B$ for the block diagonal matrix consisting of $T \times |\ZZZ|$ blocks $B^{((t,z),(t,z))}$ each of size $|\AAA| \times |\AAA|$, where block 
$B^{((t,z),(t,z))} = 0$ for all `unreachable' time-expert pairs $(t,z) \in \UUU$ and $B^{((t,z),(t,z))}  = P$ otherwise, where $P$ is defined in (\ref{eq:matrixPdef}).

\begin{lemma}
  Under the conditions of Theorem~\ref{thm:approximatenaturalgradient}, we have that
\begin{equation*}
  \widehat \nabla_{\ggamma}   J_{\ggamma} - \widetilde\nabla_{\ggamma}  J_{\ggamma}  =
  \left[(\tilde F(\ggamma)^\dagger \times F(\ggamma))^\dagger - B\right] \times   \widetilde\nabla_{\ggamma}  J_{\ggamma}.
\end{equation*}
\end{lemma}
\begin{proof}
We first note that $B=\tilde F(\ggamma)^\dagger \times \tilde F(\ggamma)$ by the definitions of $B$ and $\tilde F(\ggamma)$. 
From the standard Moore-Penrose identity $\tilde F(\ggamma)^\dagger  \times \tilde F(\ggamma) \times \tilde F(\ggamma)^\dagger= \tilde F(\ggamma)^\dagger  $, we obtain that
\begin{equation*}
  B\times  \widetilde\nabla_{\ggamma}  J_{\ggamma} =
  B\times \tilde F(\ggamma)^\dagger \times \nabla_{\ggamma} J_{\ggamma}
  =\tilde F(\ggamma)^\dagger  \times \tilde F(\ggamma) \times \tilde F(\ggamma)^\dagger \times \nabla_{\ggamma} J_{\ggamma} = \tilde F(\ggamma)^\dagger \times \nabla_{\ggamma} J_{\ggamma} = \widetilde\nabla_{\ggamma}  J_{\ggamma}.
\end{equation*}

It remains to show that
\begin{equation*}
    \widehat \nabla_{\ggamma}   J_{\ggamma}  =
  (\tilde F(\ggamma)^\dagger \times F(\ggamma))^\dagger \times   \widetilde\nabla_{\ggamma}  J_{\ggamma}.
\end{equation*}
Since $\widehat \nabla_{\ggamma}   J_{\ggamma}  = F(\ggamma)^\dagger \times \nabla_{\ggamma}   J_{\ggamma}$ and
$\widetilde\nabla_{\ggamma}  J_{\ggamma} = \tilde F(\ggamma)^\dagger \times \nabla_{\ggamma}   J_{\ggamma}$,
it suffices to argue that
$F(\ggamma)^\dagger =(\tilde F(\ggamma)^\dagger \times F(\ggamma))^\dagger \times \tilde F(\ggamma)^\dagger$.
By the local identifiability condition in the theorem,
the null spaces of $F(\ggamma)$ and $\tilde F(\ggamma)$ are eventually identical and both eventually have dimension equal to $|\ZZZ|\times T+(|\AAA|-1)\times |\UUU|$.
Both matrices are invertible on the orthogonal complement of this null space and the 
Moore-Penrose inverses act as ordinary inverses. The required identity then follows from the standard matrix identity 
$(A_1^{-1}\times A_2)^{-1}\times A_1^{-1} = A_2^{-1}$ for nonsingular matrices $A_1$ and $A_2$,
applied to the restrictions of $F(\ggamma)$ and $\tilde F(\ggamma)$ to their common range space.
\end{proof}

By the local identifiability condition in the theorem,
$F(\ggamma)$ and $\tilde F(\ggamma)$ have a common range space, and so $B$ equals the orthogonal projection matrix on this common range space.
In particular, $B^\dagger =B$.
Since Moore-Penrose inversion is continuous on the set of matrices with constant rank,
we immediately obtain the desired result $(\tilde F(\ggamma)^\dagger \times F(\ggamma))^\dagger \to B$ once 
we establish that, as $\ttheta(\ggamma)\to\ttheta^*$,
\begin{equation*}
\tilde F(\ggamma)^\dagger \times F(\ggamma) \to  B.
\end{equation*}

Our starting point is the set of identities in Proposition~\ref{prop:fisherrecursion}.
We start with the block diagonal element
  $\tilde F^{(t,t)}(\ggamma)^\dagger \times F^{(t,t)}(\ggamma) $.
  The term $K^{(t)}(\ggamma)$ in $F^{(t,t)}(\ggamma)$ is a block diagonal matrix with
block diagonal elements of the form
$\mu^{(t)}_{\ggamma}(\SSS_z)\times \hat F(\gamma^{(t)}_z)$.
As a result, 
\begin{equation*}
\tilde F^{(t,t)}(\ggamma)^\dagger \times K^{(t)}(\ggamma)
\end{equation*}
equals the block diagonal matrix consisting of $|\ZZZ|$ blocks $E^{((t,z),(t,z))}$.
To complete the proof, it remains to establish that, for $t_1 \le t_2$,
\begin{equation}
  \label{eq:summandFt1t2}
  \tilde F^{(t_1,t_1)} (\ggamma)^\dagger \times \nabla_{\gamma^{(t_1)}} \mu^{(t_2+1)}_{\ggamma}
  \times G^{(t_2+1)}(\ggamma) \times (\nabla_{\gamma^{(t_2)}} \mu^{(t_2+1)}_{\ggamma})^\top \to 0
\end{equation}
and that, for $t_1>t_2$,
\begin{equation}
  \label{eq:summandFt1t2swap}
  (\tilde F^{(t_1,t_1)} (\ggamma)^\dagger \times \nabla_{\gamma^{(t_1)}} \mu^{(t_1+1)}_{\ggamma}) \times G^{(t_1+1)}(\ggamma) \times (\nabla_{\mu^{(t_2+1)}_{\ggamma}} \mu^{(t_1+1)}_{\ggamma})^\top \times (\nabla_{\gamma^{(t_2)}} \mu^{(t_2+1)}_{\ggamma})^\top \to 0.
\end{equation}
We focus on (\ref{eq:summandFt1t2}) since the arguments for (\ref{eq:summandFt1t2swap}) are identical.
Since
\begin{equation*} 
\nabla_{\gamma^{(t_1)}} \mu^{(t_2+1)}_{\ggamma} = \nabla_{\gamma^{(t_1)}} \mu^{(t_1+1)}_{\ggamma} \times \nabla_{\mu^{(t_1+1)}_{\ggamma}} \mu^{(t_2+1)}_{\ggamma},
\end{equation*}
we shall show that, as $\ttheta(\ggamma)\to\ttheta^*$,
\begin{equation*}
  (\tilde F^{(t_1,t_1)} (\ggamma)^\dagger \times \nabla_{\gamma^{(t_1)}} \mu^{(t_1+1)}_{\ggamma}) \times \nabla_{\mu^{(t_1+1)}_{\ggamma}} \mu^{(t_2+1)}_{\ggamma}\times G^{(t_2+1)}(\ggamma) \times (\nabla_{\gamma^{(t_2)}} \mu^{(t_2+1)}_{\ggamma})^\top \to 0.
\end{equation*}

The following lemma establishes that the last term in (\ref{eq:summandFt1t2}) converges to zero.

\begin{lemma} \label{lem:gradgammatofmutgoesto0}
  As $\ttheta(\ggamma)\to\ttheta^*$, we have that, for every $0\le t< T$,
\begin{equation*}
  \nabla_{\gamma^{(t)}} \mu^{(t+1)}_{\ggamma} \to 0.
\end{equation*}
\end{lemma}

\begin{proof}
Fix some $z$ and $s'$.  We show that $\nabla_{\gamma^{(t)}_z} \mu^{(t+1)}_{\ggamma}(s')$ converges to zero, which is sufficient to establish the claim.  
We have that
\begin{align*}
\nabla_{\gamma^{(t)}_z} \mu^{(t+1)}_{\ggamma}(s') &= \sum_{s \in \SSS_z, a} \mu^{(t)}_{\ggamma}(s) \times \left(\nabla_{\gamma^{(t)}_z} \hat\pi_{\gamma^{(t)}_{z}}(a) \right)\times p^{(t)}_{\mu^{(t)}_{\ggamma}}(s'|s, a) \\
&= \sum_a \left(\sum_{s \in \SSS_z} \mu^{(t)}_{\ggamma}(s) \times p^{(t)}_{\mu^{(t)}_{\ggamma}}(s'|s, a) \right) \times \nabla_{\gamma^{(t)}_z} \hat\pi_{\gamma^{(t)}_{z(s)}}(a) \\
&= \sum_a \left(\sum_{s \in \SSS_z} \mu^{(t)}_{\ggamma}(s) \times p^{(t)}_{\mu^{(t)}_{\ggamma}}(s'|s, a) \right) \times 
\left[\hat \pi^{(t)}_{\gamma^{(t)}_z} (a) \times \left(e_a - \hat \pi^{(t)}_{\gamma^{(t)}_z} \right) \right],
\end{align*}
where the last equality follows from (\ref{eq:scoresoftmax}).
Since $\ttheta^*$ is a pure policy there exists an $a(z)$ such that $\hat \pi^{(t)}_{\gamma^{(t)}_z}$ converges to $e_{a(z)}$, and so 
the term in brackets converges to $1(a = a(z)) \times (e_a - e_{a(z)})$.  Since this expression equals zero for all $a$, the result now follows.   
\end{proof}

Having shown that the last term in (\ref{eq:summandFt1t2}) converges to zero, we now argue that the matrices in each of the three preceding expressions 
are bounded, which is sufficient to establish the claim.
We first argue that the matrices $G^{(t)}(\ggamma)$ are eventually bounded
as $\ttheta(\ggamma)\to\ttheta^*$.
Fix some time $t$, state $s$ with $\mu^{(t)}_{*}(s)>0$, and action $a$.
Eventually $\mu^{(t)}_{\ggamma}(s)>0$, and so under the conditions of the theorem the supports of
$p^{(t)}_{\mu^{(t)}_{\ggamma}}(s'|s,a)$ and $p^{(t)}_{\mu^{(t)}_{*}}(s'|s,a)$ are eventually identical.
Due to the smooth extension property,
both $p^{(t)}_\mu(s'|s,a)$ and $\nabla_\mu p^{(t)}_\mu(s'|s,a)$ are continuous in $\mu$ on a neighborhood of $\mu^{(t)}_{*}$ within the probability simplex.
Since $\mu^{(t)}_{\ggamma} \to \mu^{(t)}_{*}$, it then follows that
    \begin{equation} \label{eq:probboundperfthm}
  \lim_{\ttheta(\ggamma) \to \ttheta^*} \min_{(s,a,s') \in R^{(t)}_{\mu_{\ggamma}^{(t)}}} p^{(t)}_{\mu_{\ggamma}^{(t)}}(s'|s,a) =\min_{(s,a,s') \in R^{(t)}_{\mu_{*}^{(t)}}} p^{(t)}_{\mu_{*}^{(t)}}(s'|s,a) >0.
\end{equation}
and
      \begin{equation}
        \label{eq:probboundperfthm2}
  \lim_{\ttheta(\ggamma) \to \ttheta^*} \max_{(s,a,s') \in R^{(t)}_{\mu_{\ggamma}^{(t)}}} \nabla_\mu p^{(t)}_{\mu_{\ggamma}^{(t)}}(s'|s,a) =\max_{(s,a,s') \in R^{(t)}_{\mu_{*}^{(t)}}} \nabla_\mu p^{(t)}_{\mu_{*}^{(t)}}(s'|s,a) <\infty.
\end{equation}
Equations (\ref{eq:probboundperfthm}) and (\ref{eq:probboundperfthm2}) together imply that
\begin{equation*}
\lim_{\ttheta(\ggamma) \to \ttheta^*} \max_{(s,a,s') \in R^{(t)}_{\mu_{\ggamma}^{(t)}}} \left\| \nabla_\mu \log p^{(t)}_{\mu^{(t)}_{\ggamma}}(s'|s,a)\right\|  <\infty,
\end{equation*}
which implies that the $M^{(t)}(\ggamma)$ are eventually bounded.
The same argument shows that $\nabla_{\mu^{(t)}_{\ggamma}} \mu^{(t+1)}_{\ggamma}$ is bounded for all $t$.
This implies, as can be seen via a straightforward recursive argument working backwards through time, that the $G^{(t)}(\ggamma)$ are bounded, too.
It also shows that, for $t_1<t_2$, the matrix 
\begin{equation*}
\nabla_{\mu^{(t_1)}_{\ggamma}} \mu^{(t_2+1)}_{\ggamma} = (\nabla_{\mu^{(t_1)}_{\ggamma}} \mu^{(t_1+1)}_{\ggamma}) \times \cdots\times
 (\nabla_{\mu^{(t_2)}_{\ggamma}} \mu^{(t_2+1)}_{\ggamma})
\end{equation*}
is bounded. 

It remains to show that the matrix $\tilde F^{(t_1,t_1)} (\ggamma)^\dagger \times \nabla_{\gamma^{(t_1)}} \mu^{(t_1+1)}_{\ggamma}$ in (\ref{eq:summandFt1t2}) is bounded.
We claim that element $((z,a), s')$ of this matrix is
\begin{equation*}
  \overline p^{(t)}_{\mu^{(t)}_{\ggamma}}(s'|z,a)
  - \frac1{|\AAA|} \sum_{\tilde a}  \overline p^{(t)}_{\mu^{(t)}_{\ggamma}}(s'|z,\tilde a),
\end{equation*}
where 
\begin{equation*}
  \overline p^{(t)}_{\mu}(s'|z,a) =
\begin{cases}
  \sum_{s\in\SSS_z} \mu(s|\SSS_z) \times p^{(t)}_\mu(s'|s,a) &\text{if } \mu(\SSS_z)>0\\
  0& \text{otherwise.}
  \end{cases}
\end{equation*}
The argument is almost identical to the main argument of Proposition~\ref{prop:naturalgradientmoesoft}. 
In lieu of the policy gradient theorem, the starting point is the identity
\begin{equation*}
\nabla_{\gamma^{(t)}} \mu^{(t+1)}_{\ggamma}(s') = \sum_{s, a} \mu^{(t)}_{\ggamma}(s) \times \left(\nabla_{\gamma^{(t)}} \hat\pi_{\gamma^{(t)}_{z(s)}}(a) \right)\times p^{(t)}_{\mu^{(t)}_{\ggamma}}(s'|s, a).
\end{equation*}
The rest of the argument goes through verbatim with $p^{(t)}_{\mu^{(t)}_{\ggamma}}(s'|s, a)$ playing the role of $Q_{\mu^{(t)}_{\ggamma},\sigma^{(t+1)}_{\ggamma}}^{(t)}(s, a)$.

\section{Efficient Gradient Calculation for QPLEX Models}
\label{app:effcalc}
This appendix shows how the policy gradient can be calculated efficiently
due to the structure underlying QPLEX models, specifically due to the structure of states and by leveraging the sum-product algorithm.
We work with a more general setting than the dynamic pricing example,
which covers a variety of QPLEX models:
in the language of \cite{qplexbook}, we consider simple transition dynamics
with a single label.
More general models can be worked out similarly.

To accommodate this generality, the terms $\min(z,n) $ and $\min(z',n)$
in the definition of $q(k'|z,d,z')$ in (\ref{eq:qk'}) need to be replaced by $x(z)$
and $x(z')$, respectively, where $x$ is a so-called \emph{size function}
\cite{qplexbook}.
We also allow a general \emph{routing pmf} $\rho^{(t)}(z'|z,d,a)$.
This more general setting does not affect the calculations below.

\subsection{State Structure}
As a consequence of the separability of the local reward 
(\ref{eq:rewardsep}) and the kernel structure in (\ref{eq:kernelqplex}),
the $Q$-function takes the form, 
\begin{align*}
Q_{\mu, \sigma'}^{(t)}(s, a) &= r_\mu^{(t)}(s, a) + \sum_{s'} p^{(t)}_\mu(s' | s, a) \times \sigma'(s') \\
&= c^{(t)}_\mu +\hat r_{\mu_{|z}}^{(t)}(z, a) + \sum_{z',\ell'} \hat p^{(t)}_{\mu_{|z}} (z',\ell' | z, a) \times \sigma'(z',\ell') \\
&= c^{(t)}_\mu +\hat Q^{(t)}_{\mu_{|z}, \sigma'}(z, a),
\end{align*}
where $\mu_{|z}$ should be interpreted as some fixed pmf $\xi_0$ on $\LLL$ if $\mu(z)=0$ and
\begin{equation} \label{eq:hatQdef}
\hat Q^{(t)}_{\xi, \sigma'}(z, a) = \hat r_{\xi}^{(t)}(z, a) + \sum_{z',\ell'} \hat p^{(t)}_{\xi} (z',\ell' | z, a) \times \sigma'(z',\ell').
\end{equation}

For a differentiable function $h$ of a pmf $\xi$ on $\LLL$, 
we write $\nabla^{\ctrd}_\xi h(\xi)$ for the mean-centered gradient $(I-\bm 1 \times \xi^\top)\times \nabla_\xi h(\xi)$ and denote element $\ell$ of $\nabla^{\ctrd}_\xi h(\xi)$ by $\frac{\partial^{\ctrd}}{\partial\xi(\ell)} h(\xi)$.
That is,
\begin{equation*}
  \nabla^{\ctrd}_\xi h(\xi) = \nabla_\xi h(\xi) - \bm 1 \times
  \sum_\ell \xi(\ell) \times \frac{\partial}{\partial \xi(\ell)} h(\xi).
  \end{equation*}
  
\begin{lemma}
  \label{lem:conditionalgrad}
Let $h$ be a differentiable real-valued function of a pmf $\xi$ over $\LLL$.  
Fix some $\mu$ and $\tilde z$ such that $\mu(\tilde z) > 0$. Then
  \begin{equation*}
    \frac{\partial}{\partial\mu(z,\ell)} h(\mu_{|\tilde z}) = \frac{1(z=\tilde z)}{\mu(\tilde z)}  \times  \frac{\partial^{\ctrd}}{\partial\xi(\ell)} h(\mu_{|\tilde z}).  
    \end{equation*}
  \end{lemma}
  \begin{proof}
We start from
    \begin{equation*}
  \frac\partial{\partial \mu(z,\ell)}
  \mu_{|\tilde z}(\tilde \ell) =
  \frac\partial{\partial \mu(z,\ell)}
  \frac{\mu(\tilde z, \tilde \ell)}{\sum_{\bar\ell} \mu(\tilde z, \bar\ell)} =
  \begin{cases}
    \frac{1(\tilde \ell=\ell) - \mu_{|z}(\tilde \ell)}{\mu(\tilde z)}
    & z=\tilde z\\
    0 & z\neq \tilde z,
    \end{cases}
  \end{equation*}
  which shows that we can focus on $z = \tilde z$. The chain rule then yields
  \begin{align*}
    \frac\partial{\partial \mu(\tilde z,\ell)} h(\mu_{|\tilde z})
    &= \frac{1}{\mu(\tilde z)}\times
      \sum_{\tilde \ell} (1(\tilde \ell=\ell) - \mu_{|\tilde z}(\tilde \ell)) \times \frac\partial{\partial \xi(\tilde \ell)} h(\mu_{|\tilde z})\\
    &= \frac{1}{\mu(\tilde z)}\times\left[ \frac\partial{\partial \xi( \ell)} h(\mu_{|\tilde z})-
      \sum_{\tilde \ell} \mu_{|\tilde z}(\tilde \ell) \times \frac\partial{\partial \xi(\tilde \ell)} h(\mu_{|\tilde z})\right],
    \end{align*}
as required.
    \end{proof}

Writing $\mu^{(t)}_{\ttheta, |z}$ for the conditional pmf of the label given the counter $z$ under $\mu^{(t)}_{\ttheta}$,
the updating rules (\ref{eq:expertforward}) and (\ref{eq:expertbackward}) become
\begin{equation*}
\begin{aligned}
  \mu^{(0)}_{\ttheta}(z',\ell')&=\mu^{(0)}(z',\ell')\\
  \mu^{(t+1)}_{\ttheta}(z',\ell') &= \sum_{z,a} \mu^{(t)}_{\ttheta}(z) \times \theta^{(t)}_z(a) \times \hat p^{(t)}_{\mu^{(t)}_{\ttheta, |z}} (z',\ell'| z, a),
\end{aligned}
\end{equation*}
and using Lemma~\ref{lem:conditionalgrad} we have that
\begin{equation}
  \label{eq:sigmaqplex}
  \begin{aligned}
    \sigma^{(T)}_{\ttheta}(z,\ell)&=c^{(T)}_{\mu^{(T)}_{\ttheta}}+ \hat r^{(T)}_{\mu^{(T)}_{\ttheta, | z}}(z)+
    \frac{\partial}{\partial\mu(z,\ell)} c^{(T)}_{\mu^{(T)}_{\ttheta}}  + 
    \frac{\partial^{\ctrd}}{\partial \xi(\ell)} \hat r^{(T)}_{\mu^{(T)}_{\ttheta, |z}}(z)\\
     \sigma^{(t)}_{\ttheta}(z,\ell) &=c^{(t)}_{\mu^{(t)}_{\ttheta}} +
\sum_a \theta^{(t)}_z (a) \times \hat Q^{(t)}_{\mu^{(t)}_{\ttheta, |z}, \sigma^{(t+1)}_{\ttheta}}(z, a) + \frac{\partial}{\partial\mu(z,\ell)} c^{(t)}_{\mu^{(t)}_{\ttheta}} + 
\sum_a \theta^{(t)}_{z} (a) \times \frac{\partial^{\ctrd}}{\partial \xi(\ell)}
\hat Q^{(t)}_{\mu^{(t)}_{\ttheta, |z}, \sigma^{(t+1)}_{\ttheta}}(z, a),
\end{aligned}
\end{equation}
again interpreting $\mu^{(t)}_{\ttheta,|z}$ as $\xi_0$ if $\mu^{(t)}_{\ttheta}(z)=0$.
Thus, in the generic forward iterative scheme (\ref{eq:expertforward}),
the sum over $s=(z,\ell)$ becomes a sum over $z$.
In the generic backward iterative scheme (\ref{eq:expertbackward}),
the sum over $\tilde s=(\tilde z,\tilde \ell)$ is completely eliminated.

The state structure also provides benefits for the algorithm,
specifically for calculating the $\overline Q^{(t)}_{z,\ttheta}$ defined in (\ref{eq:barQ}).
Since all values of $Q^{(t)}_{\mu,\sigma'}(s,a)$ are identical for $s\in \SSS_z$, $Q^{(t)}_{\mu^{(t)}_{\ttheta},\sigma^{(t+1)}_{\ttheta}}(s,a)$ can be pulled out of the sum and the resulting sum of $\mu^{(t)}(s|\SSS_z)$ over $s\in\SSS_z$ is trivially 1.
  (Moreover, since any constant can be added to $\overline Q^{(t)}_{z,\ttheta}$ in the updating rule (\ref{eq:expupdatepi}), the constant terms $c^{(t)}_\mu$ do not play a role in the algorithm except through the $\sigma^{(t)}_{\ttheta}$. This does not result in substantial savings.)

\subsection{Sum-Product Structure}
We show how to exploit the fact that the policy gradient can be written as a sum of product terms, so that the sum-product algorithm can be used for efficient calculation. 
We implement this algorithm implicitly through the algebraic definitions and summations but omit the explicit message-passing formulation. The algebraic manipulations we obtain are equivalent to sum-product message passing.
Throughout, let $\vartheta$ be a pmf on $\AAA$ and $\xi$ be a pmf on $\LLL$.

There are four terms created by substituting the expression (\ref{eq:hatQdef}) for $\hat  Q^{(t)}_{\xi, \sigma'}$ in (\ref{eq:sigmaqplex}).
Calculating two of these terms, 
$\sum_{a} \vartheta(a) \times\hat r_{\xi}(z,a) $ and $\sum_{a} \vartheta(a) \times \nabla_\xi^{\ctrd} \hat r_{\xi}(z,a) $, are problem-specific,
and so we only discuss how to efficiently calculate
the remaining two terms
  \begin{equation}
    \label{eq:effcalc1}
    \sum_{a,z',\ell'} \vartheta(a)  \times
    \hat p^{(t)}_\xi(z',\ell'|z,a) \times\sigma'(z',\ell')
  \end{equation}
  and
  \begin{equation}
    \label{eq:effcalc2}
    \sum_{a,z',\ell'} \vartheta(a)  \times
    \nabla_\xi^{\ctrd} \hat p^{(t)}_\xi(z',\ell'|z,a) \times\sigma'(z',\ell')
    \end{equation}
by exploiting the specific form of 
\begin{equation*}
\hat p^{(t)}_{\xi}(z', \ell' | z, a) = \sum_{d} q_{\xi}(d | z) \times \rho^{(t)}(z'|z, d, a) \times \sum_{k'} q(k' | z, d, z') \times q_{\xi}(\ell' | z, k')
\end{equation*}
from the QPLEX calculus. 
If $\xi(1) = 1$, then $q(\old | z, d, z') = 0$ and $q_{\xi}(\ell' | z, \old)$ is undefined, in which case the product of these terms should be interpreted as zero.

We begin with (\ref{eq:effcalc1}). 
After defining
\begin{equation*}
q^{(t)}_{\vartheta}(z'|z,d) = \sum_a \vartheta(a) \times \rho^{(t)}(z'|z,d,a),
\end{equation*}
the forward quantity
\begin{equation*}
f^{(t)}_{\xi,\vartheta}(z,z',k') =\sum_d q_{\xi}(d|z) \times q^{(t)}_{\vartheta}(z'|z,d) \times q(k'|z,d,z'),
\end{equation*}
and the backward quantity
\begin{equation*}
b_{\xi,\sigma'} (z,z',k') = \sum_{\ell'} q_{\xi}(\ell'|z,k')\times \sigma'(z',\ell'),
\end{equation*}
the sum-product algorithm leads to the following efficient representation:
\begin{equation}
  \label{eq:effcalc3}
  \sum_{a,z',\ell'} \vartheta(a)  \times \hat p^{(t)}_\xi(z', \ell'|z,a) \times \sigma'(z',\ell')=
\sum_{z',k'}  f^{(t)}_{\xi,\vartheta}(z,z',k') \times b_{\xi,\sigma'} (z,z',k'),
\end{equation}
where it is understood that the sum over $k'$ only involves $k' = \new$ if $\xi(1) = 1$.

We now turn to (\ref{eq:effcalc2}), for which we apply $\nabla^{\ctrd}_\xi$ to (\ref{eq:effcalc3}).
Define
  \begin{equation*}
    \hat f^{(t)}_{\xi,\vartheta}(z,\ell,z',k') =
    \frac{\partial^{\ctrd}} {\partial \xi(\ell)} f^{(t)}_{\xi,\vartheta}(z,z',k')
    \end{equation*}
and
      \begin{equation*}
\hat b_{\xi,\sigma'} (z,\ell, z') = \frac{\partial^{\ctrd}}{\partial \xi(\ell)} b_{\xi,\sigma'} (z,z',\old).
\end{equation*}
We thus obtain from the chain rule that
\begin{multline*}
\sum_{a,z',\ell'} \vartheta(a)  \times\frac{\partial^{\ctrd}}{\partial\xi(\ell)} \hat p^{(t)}_\xi(z', \ell'|z,a) \times \sigma'(z',\ell') \\= \sum_{z',k'} \hat f^{(t)}_{\xi,\vartheta}(z,\ell,z',k') \times b_{\xi,\sigma'} (z,z',k') + \sum_{z'}  f^{(t)}_{\xi,\vartheta}(z,z',\old) \times \hat b_{\xi,\sigma'} (z,\ell,z'), 
\end{multline*}
since the gradient  $\nabla_\xi b_{\xi,\sigma'} (z,z',k')$ is only nonzero for $k'=\old$.

First consider $\hat f^{(t)}_{\xi,\vartheta}(z,\ell,z',k') $. 
It follows from (\ref{eq:qxidgivenz}) that $\frac{\partial}{\partial \xi(\ell)} q_{\xi}(d|z) = 0$ if $\ell > 1$,
and so
\begin{equation*}
\hat f^{(t)}_{\xi,\vartheta}(z,\ell,z',k') = (1(\ell=1) -\xi(1)) \times 
\sum_d \frac{\partial}{\partial \xi(1)} q_{\xi}(d|z) \times q^{(t)}_{\vartheta}(z'|z,d) \times q(k'|z,d,z').
\end{equation*}
If $x(z) = 0$, then $q_{\xi}(d|z) = 1$, and so $\hat f^{(t)}_{\xi,\vartheta}(z,\ell,z',k') = 0$. Assuming $x(z) \ge 1$,
\begin{equation*}
  \frac{\partial}{\partial \xi(1)} q_{\xi}(d|z) =
\begin{cases}
-x(z) \times (1  - \xi(1))^{x(z)-1} & \text{if } d = 0 \\
x(z) \times \xi(1)^{x(z)-1} & \text{if } d = x(z) \\
1(\xi(1) > 0) \times \left( \frac{d}{\xi(1)} - \frac{x(z)-d}{(1 - \xi(1))} \right) \times q_{\xi}(d|z) & \text{if } 0 < d < x(z).
\end{cases}
\end{equation*}

          We next turn to $\hat b_{\xi,\sigma'} (z,\ell, z')$, which only is calculated when $\xi(1) < 1$. It follows from (\ref{eq:qxiell'givenzk'}) that
          \begin{equation*}
            \frac\partial{\partial \xi(\ell)} q_\xi (\ell'|z,\old) =
            \frac\partial{\partial \xi(\ell)} \frac{\xi(\ell'+1)}{1-\xi(1)} =
            \frac{1(\ell=1)\times \xi(\ell'+1)}{(1-\xi(1))^2}+\frac{1(\ell=\ell'+1)}{1-\xi(1)} ,
          \end{equation*}
and so
          \begin{align*}
            \frac{\partial^{\ctrd}}{\partial \xi(\ell)} q_\xi (\ell'|z,\old) &=
                                                                               \frac{(1(\ell=1)-\xi(1))\times \xi(\ell'+1)}{(1-\xi(1))^2} + \frac{1(\ell=\ell'+1)-\xi(\ell'+1)}{1-\xi(1)} \\
                                                                             &=   \frac{1(\ell=\ell'+1)}{1-\xi(1)}-\frac{1(\ell>1) \times \xi(\ell'+1)}{(1-\xi(1))^2} \\
&=            \frac {1(\ell>1)}{1-\xi(1)} \times 
[1(\ell=\ell'+1) - q_{\xi}(\ell'|z,\old)],
          \end{align*}
          where the last equality uses that $\ell'+1>1$.
We therefore obtain that
          \begin{align*}
            \hat b_{\xi,\sigma'} (z,\ell, z')
            &=  \sum_{\ell'} \frac{\partial^{\ctrd}}{\partial \xi(\ell)}
              q_{\xi}(\ell'|z,\old)\times \sigma'(z',\ell')\\
            &= \frac {1(\ell>1)}{1-\xi(1)} \times \sum_{\ell'} 
[1(\ell=\ell'+1) - q_{\xi}(\ell'|z,\old)]\times \sigma'(z',\ell') \\
&= \frac {1(\ell>1)}{1-\xi(1)} \times 
\left[\sigma'(z',\ell-1) - b_{\xi,\sigma'} (z,z',\old)\right].
          \end{align*}

In conclusion, collecting terms, 
\begin{multline*}
  \sigma^{(t)}_{\ttheta}(z,\ell)  = c^{(t)}_{\mu^{(t)}_{\ttheta}}+\frac{\partial}{\partial \mu(z,\ell)}c^{(t)}_{\mu^{(t)}_{\ttheta}}+\sum_a \theta^{(t)}_z(a)\times
  \left(\hat r^{(t)}_{\mu^{(t)}_{\ttheta, |z}}(z,a) +\frac{\partial^{\ctrd}}{\partial \xi(\ell)} \hat r^{(t)}_{\mu^{(t)}_{\ttheta, |z}}(z,a)\right) \\
  + \sum_{z',k'} \left(f^{(t)}_{\mu^{(t)}_{\ttheta, |z},\theta^{(t)}_z}(z,z',k') +\hat f^{(t)}_{\mu^{(t)}_{\ttheta, |z},\theta^{(t)}_z}(z,\ell,z',k') \right)\times b_{\mu^{(t)}_{\ttheta, |z},\sigma^{(t+1)}_{\ttheta}}(z,z',k')\\
+ \sum_{z'} f^{(t)}_{\mu^{(t)}_{\ttheta, |z},\theta^{(t)}_z}(z,z',\old)\times \hat b_{\mu^{(t)}_{\ttheta, |z},\sigma^{(t+1)}_{\ttheta}}(z,\ell,z').
\end{multline*}

  \printbibliography

  \end{document}